\author{Lennart Ronge}
\date{\today}
\documentclass[11pt]{report}
\usepackage{amsmath}
\usepackage{amssymb}
\usepackage{amsfonts}
\usepackage[ansinew]{inputenc}
\usepackage{amsthm}
\usepackage{geometry}
\usepackage{hyperref}
\usepackage[UKenglish]{babel}
\usepackage{bbm}
\usepackage{tikz}
\usepackage{graphicx}
\usepackage {wrapfig}
\usepackage[style=alphabetic, backend=bibtex]{biblatex}
\def\head#1{\textbf{#1}}
\def\limn{\lim_{n\rightarrow \infty}}
\def\<{\left<}
\def\>{\right>}
\def\Rp{{\mathbb{R}_{\geq0}}}
\def\Z{\mathbb{Z}}
\def\C{{\mathbb{C}}}
\def\R{{\mathbb{R}}}
\def\N{{\mathbb{N}}}
\def\mat#1#2#3#4{{\begin{pmatrix} #1&#2\\ #3&#4 \end{pmatrix}}}

\def\insum#1{\sum\limits_{#1=0}^\infty}
\def\intR{\int\limits_\R}
\def\intP{\int\limits_0^\infty}
\def\intM{\int\limits_M}
\def\O {\backslash \{0\}}
\def\F {\mathcal{F}}
\def\Tdot {\dot T^*}

\def\D{\mathcal{D}}
\def\K{\mathcal{K}}
\def\M{\mathcal{M}}
\def\casedist#1#2#3#4{
	\begin{cases}
		#1& \text{, if } #2\\
		#3& \text{, if } #4 \\
	\end{cases}}

\newtheorem{df}{Definition}[section]
\newtheorem{lemma}[df]{Lemma}
\newtheorem{thm}[df]{Theorem}
\newtheorem*{thm*}{Theorem}
\newtheorem{prop}[df]{Proposition}
\newtheorem{cor}[df]{Corollary}

\newtheorem{gn}[df]{Fixed Notation}
\newtheorem{rem}[df]{Remark}
\newtheorem{defprop}[df]{Definition/Proposition}
\newtheorem*{claim}{Claim}
\DeclareMathOperator{\Dom}{Dom}
\DeclareMathOperator{\Ran}{Ran}

\DeclareMathOperator{\spann}{span}

\DeclareMathOperator{\supp}{supp}

\DeclareMathOperator{\Exp}{Exp}
\DeclareMathOperator{\Vol}{Vol}
\DeclareMathOperator{\grad}{grad}
\DeclareMathOperator{\sign}{sign}
\DeclareMathOperator{\Res}{Res}
\DeclareMathOperator{\diag}{diag}
\DeclareMathOperator{\scal}{scal}
\DeclareMathOperator{\tr}{tr}
\DeclareMathOperator{\Op}{Op}
\DeclareMathOperator{\Div}{div}
\DeclareMathOperator{\singsupp}{singsupp}
\DeclareMathOperator{\rk}{rk}
\DeclareMathOperator{\WF}{WF}
\begin{document}
	\begin{titlepage}
		\begin{center}
			\vspace {1cm}
			\textbf{\huge{Extracting Hadamard Coefficients from Green's Operators}}
			
			\vspace{5cm}
			
			Dissertation\\
			zur\\
			Erlangung des Doktorgrades (Dr. rer. nat.)\\
			der\\
			Mathematisch-Naturwissenschaftlichen Fakultät\\
			der\\
			Rheinischen Friedrich-Wilhelms-Universität Bonn
			\vfill
			vorgelegt von\\
			{\Large{\textbf{Lennart Ronge}}}\\
			aus\\
			Flörsheim (Main)\\
			\vspace{2cm}
			Bonn, 2022
			
		\end{center}
	\end{titlepage}
		\begin{center}
			Angefertigt mit Genehmigung der Mathematisch-Naturwissenschaftlichen Fakultät\\
			der Rheinischen Friedrich-Wilhelms-Universität Bonn
		\end{center}
		\vfill
		1. Gutachter: Prof. Dr. Matthias Lesch\\
		2. Gutachter: Prof. Alexander Strohmaier\\
		\vspace{2cm}\\
		Tag der Promotion: 27.01.2023\\
		Erscheinungsjahr: 2023\\
		\vspace{1cm}
	
	\newpage
	\tableofcontents
	\newpage
	\chapter{Introduction}
Hadamard coefficients are sections in a vector bundle associated to a normally hyperbolic differential operator in Lorentzian geometry. They are the Lorentzian equivalent of the heat coefficients of a generalized Laplace operator (both coefficients are known under various other names, like Seeley-DeWitt or Schwinger-DeWitt coefficients). While both types of coefficients arise from the formally same transport equations and are given by the same formal expressions (up to a factor of $k!$ and in the conventions of \cite{BGV} a density factor), they arise in very different contexts: The heat coefficients, as the name suggests, arise in an asymptotic expansion of the heat semigroup $e^{-t\Delta}$ for small $t$, while Hadamard coefficients arise in the study of Green's operators. Hadamard coefficients occur in quantum field theory on curved spacetimes (see e.g. \cite{DeFo}) and in the Lorentzian version of the APS index theorem (\cite{BS}).

Heat coefficients play an important role in various areas of analysis, geometry and physics. They contain important information about the properties of the associated Laplace-type operator and the geometry of the underlying manifold. More information on the heat coefficients and their application to index theory can be found in \cite{BGV}, for applications in physics see \cite{Vas}. Their definition via the heat semigroup also makes them available in the setting of noncommutative geometry, where classical geometric quantities are no longer available. For example, they can be used to define the scalar curvature for a noncommutative manifold.

In the Lorentzian case, there is no well-defined heat semigroup and the Hadamard coefficients need to be defined via recursive differential equations known as transport equations. These do not make sense in a noncommutative setting, so the Hadamard coefficients cannot be used there. This thesis is an attempt to change this by developing a formula for the Hadamard coefficients (restricted to or integrated over the diagonal) that does not involve local aspects of the underlying geometry and that might thus be applicable in some Lorentzian version of noncommutative geometry.

The setting in which Hadamard coefficients originally arise is in the construction of Green's operators and in the description of their singularity structure (see \cite[Chapter 2]{BGP}). More precisely, there is an asymptotic expansion (\cite [Proposition 2.5.1]{BGP}) in differentiability orders of the form
\[\K(G^\pm)\sim\insum{k}V^kR_\pm(2k+2),\]
where $\K(G^\pm)$ is the kernel of the advanced/retarded Green's operator, the $V^k$ are the Hadamard coefficients and the $R_\pm(2k+2)$ are the so called Riesz-distributions, which can be thought of roughly as powers of the distance function. This expansion will serve as the foundation for this thesis. We want to extract from this a formula for the Hadamard coefficients on the diagonal in terms of the Green's operators. As it turns out, the expansion for a single Green's operator does not contain enough information to determine those values, so instead we need to look either at powers of the Green's operators or at the analogue of a resolvent.

\section*{Results}
Before turning to our formula for Hadamard coefficients, we first obtain some generalizations of the Hadamard expansion that may be of independent interest. First, we develop an expansion for powers of the Green's operator, then an expansion for the Green's analogue of a resolvent. Both can be combined in the following expansion (see \ref{Hadamexpz})
\[\K((G_{P-z}^\pm)^m)\sim \insum{k} \binom{m+k-1}{k}V^kR_\pm(z,2k+2m).\]
Here $\K$ refers to taking the Schwartz kernel, $G_{P-z}^\pm$ is the advanced/retarded Green's operator associated to $P-z$, where $z\in \C$ and $P$ is some normally hyperbolic operator, $V^k$ are the Hadamard coefficients associated to $P$ and $R(z,n)$ are a $z$-dependent generalization of the Riesz distributions, modeled on  fundamental solutions to $\square-z$ instead of $\square$, which are equal to the ordinary Riesz distributions in case $z=0$. The asymptotic expansion is up to functions of arbitrary differentiability.

We then proceed to show our main formula for the Hadamard coefficients. We obtain the following:
\begin{thm*}(\ref{finalgen})
	Let $P$ be a normally hyperbolic operator acting on sections of a vector bundle E over a globally hyperbolic manifold $M$.
	Let $o,K\in\N$. Let $w$ be a timelike unit speed geodesic in $M$ and $x=w(0)$. Let $f\in C_c^\infty(\R)$ be odd and assume that \[\M'(f):=\frac{\M(f)}{\Gamma(\frac{\cdot+1}{2})}\]
	(where $\M$ denotes the Mellin transform) is non-zero on $\Z$. Let $A:E\otimes E_x^*\rightarrow \C$ be a smooth fibrewise linear map. Let $V^K$ denote the $K$-th Hadamard coefficient of $P$ and $G_{P-z,x}$ denote the kernel of the causal propagator associated to $P-z$ for $z\in \C$, restricted to the point $x$ in the second variable. Then we have
	\begin{align*}
		A(V^{K}_x(x))
		=&\sum\limits_{m=0}^K C(m,K,o,d,f) w^*(AG_{P-z,x})[f(\tfrac{\cdot}{s})][[s^{2K+2m+2o-d+3}z^{m+o}]],
	\end{align*}
	where $[[\cdot]]$ denotes the operation of taking the coefficients in front of the monomial written in the brackets and
	\[C(m,K,o,d,f)=\frac{\pi^{\frac{d}{2}-1}4^{K+m+o}(m+o)!K!}{\M'(f)({2K+2m+2o-d+3})}\binom{\frac{d}{2}-1-o-K}{m}\binom{2K+o+1-\frac{d}{2}}{K-m}.\]
\end{thm*}
The pull back via $w$ can be thought of as integrating along a timeline. The map $A$ can be viewed as an arbitrary component function of the Vector bundle that the Hadamard coefficients are sections of. This is necessary, as vectors from different fibres cannot be integrated directly.

While this formula is still local, it can be turned into a non-local formula by multiplying with cut-off functions and integrating. This leads to the following (using the same setting as in the previous theorem):
\begin{thm*}(\ref{globalexp})
	Let $\tau$ denote some choice of translation along timelike unit speed geodesics on $M$.
	For $\chi_1,\chi_2\in C_c^\infty(M)$ (with $\mu_{\chi_i}$ denoting the corresponding multiplication operators), there is $s_0>0$, such that for $s<s_0$, the operator \[\intR\mu_{\chi_1}f(\tfrac{t}{s})\tau(t)G\mu_{\chi_2} dt\]
	extends continuously to a trace class operator in $L^2(E)$ (with respect to any hermitian structure on $E$). 
	We have for any $K,o\in \N$:
	\begin{align*}
		&\intM \chi_1(x)\chi_2(x)\tr(V^{K}_x(x))dx\\
		=&\sum\limits_{m=0}^K C(m,K,o,d,f) \tr\left( \intR\mu_{\chi_1}f(\tfrac{t}{s})\tau(t)G\mu_{\chi_2} dt\right)[[s^{2K+2m+2o-d+3}z^{m+o}]]
	\end{align*}
for $C(m,K,o,d,f)$ as above.
\end{thm*}
This formula has two main drawbacks.
 
The first is that the right hand side is fairly complicated. From a computational point of view, it will generally be more difficult to compute the kernel of the causal propagator than to just compute the Hadamard coefficients via the transport equations that define them. For abstract calculations, the involved combinatorial factors might make it difficult to equate the right hand side to anything more useful. 

The second drawback is the reliance on a notion of time translation along geodesics, which does not generally exist in a canonical way.

In principle, the formula no longer contains "local" ingredients. Whether it is really suitable for a noncommutative setting is hard to answer at the current time, as the right framework for Lorentzian noncommutative geometry is still a subject of ongoing research (see \cite{Fra} for an overview over some approaches). 
Such a framework would probably contain some notion of causal structure, which means it would be likely that the Green's operators could be defined as inverses on suitable "function spaces". The notion of time translation is more problematic. Theoretically, something like this could be defined in the language of noncommutative geometry by choosing a timelike geodesic unit vector field. 
Vector fields make sense in noncommutative geometry as derivations and the condition of being geodesic can be defined in terms of a covariant derivative, which also exists in the noncommutative setting. The main problem is that such a vector field may not exist globally and is not canonical even if it does, even for classical spacetimes. 

Such a vector field does exist, however, locally around each point in a Lorentzian spacetime. Thus one could generally obtain a formula for the Hadamard coefficients by use of a partition of unity, but this would make the formula even more complicated and would not really be in the spirit of having a more global approach.

On the other hand, the fact that the formula works locally
makes it likely that any results obtained by using it that no longer explicitly depend on $\tau$ would also hold on general globally hyperbolic manifolds and perhaps also on general Lorentzian noncommutative spaces. 

The result is thus to be seen as a tool that might be helpful in finding further formulas for non-commutative geometry, rather than an alternative way for actually computing the Hadamard coefficients.

In case the dimension is even, the above difficulties can be avoided for the first few Hadamard coefficients ($K<\tfrac{d}{2}$). For these, we have a simpler formula that works for arbitrary timelike curves:
\begin{thm*}
	(\ref{smallkthm})
	Assume that $d$ is even and $K<\frac{d}{2}$. Assume the above setting, but now we allow $w$ to be an arbitrary timelike curve with $w(0)=x$ rather than a unit speed geodesic. Then we have for $k<\tfrac{d}{2}$:
	\[A_xV^{k}_x(x)=\frac{(4\pi)^{\frac{d}{2}-1}{(\frac{d}{2}-1-k)}!k!}{\M'(f)(1)}w_x^*(A_xG_{P-z,x})[f_s][[s^1z^{\frac{d}{2}-1-k}]].\]
\end{thm*}
The integrated version of this also holds.
\section*{Structure of proof}
In Chapter 2, we will collect some preliminary facts and establish notation.

In Chapter 3, we develop the aforementioned asymptotic expansions for powers of the Green's operators and for the Green's operators associated to the operators $P-z$ for $z\in \C$ (in terms of the Hadamard coefficients of $P$).

In Chapter 4, we investigate the function given morally by
\[L(s):=\intR f(\tfrac{t}{s})\K(G)(w(t),w(0))dt\]
for some timelike curve $w$ and odd $C_c^\infty$-function $f$. As the Schwartz kernel of the causal propagator $G=G_+-G_-$ is a distribution, the above is not a priori well defined. We thus need to use wavefront calculus to rigorously define it. We find that $L$ has an asymptotic expansion of the form (assuming for simplicity that $P$ acts on scalar functions and $w$ is a unit speed geodesic)
\[L(s)\sim \insum{k}\insum{n}a_{k,n}(V^k_{w(0)}\circ w)^{(2n)}(0)s^{2k+2n+3-d}.\]
This is not quite what we want, as the higher time-derivatives of the Hadamard coefficients introduce error terms (the terms for $n\neq 0$) that prevent us from extracting the Hadamard coefficients (at $w(0)$) from this asymptotic expansion. 

To remedy this, we need to consider the same expansion for $P-z$ instead of $P$, which we will do in Chapter 5. This will give us enough additional information to extract the Hadamard coefficients from the $z$-dependent version of $L$, which enables us to show the local version of our main theorem. For the first few Hadamard coefficients for even $d$, this works fairly easily, while for the others it will be more involved and we will need $w$ to be a geodesic.

Finally in Chapter 6, we develop the global theorem stated above, as well as another formulation in terms of evolution operators.

\newpage
\chapter{Preliminaries}
\section{Miscellaneous}
\subsection{General information}
We assume that the reader is well acquainted with manifolds and distributions and has some degree of familiarity with locally convex topological vector spaces (i.e. working with seminorms). Background information about Lorentzian geometry and wavefront calculus required for this thesis, as well as a description of the main objects we will work with (Hadamard coefficients, Riesz distributions and Green's operators) and a few other facts will be provided in this chapter. 

An introduction to topological vector spaces and distributions on $\R^d$ may be found e.g. in \cite{Tre}. Some information on Distributions in Vector bundles can be found in the beginning of \cite{BGP}, which will also be our main reference for the preliminaries on Lorentzian geometry. Our main source for wavefront calculus will be \cite{BDH}.

In order to help the reader keep track of the various definitions and notation introduced throughout this thesis (as well as some standard notation that is being used), a notation index is included at the end. There will also be some abuse of notation introduced in subsection \ref{abuseofnotation} that the reader should be aware of, though most of it is fairly canonical. Much of this is not included in the notation index, as it consists of omitting notation rather than introducing it.

The rigorous part of the thesis is occasionally broken by a remark. These are meant to help the reader understand the motivation and background thought behind what is happening, without being a part of the logical line of argumentation, and are thus not as mathematically rigorous as the rest of the thesis. While any mathematically precise statement in a comment is correct to the authors best knowledge, it may not be as rigorously checked as the statements made elsewhere.

\subsection{Motivation: the spectral action}
The contents of this subsection are not logically necessary for the subsequent considerations, but are helpful for understanding the motivation both for the results and the methods of this thesis.

The (bosonic) spectral action associated to a Dirac Operator $D$ (or an abstract generalization thereof) is defined as
\[\tr(f(sD)),\]
where $s>0$ is some cut-off parameter and $f\in C_c^\infty(\R)$ is some even function that can be chosen freely. This plays an important role, for example, in the non-commutative geometry based approach to unify the standard model of particle physics with gravity (see e.g. \cite{CoCh}). The full spectral action also contains another summand that describes fermions, but as that part is not relevant to the considerations in this thesis, we will use the term "spectral action" to refer only to the bosonic part described above. 

The spectral action of a Dirac operator $D$ has an asymptotic expansion in terms of heat coefficients (see \cite[(2.14)]{CoCh})
\[\tr(f(sD))\stackrel{s\rightarrow0}{\sim}\insum{k}c(f,k)a_k(D^2)s^{2k-d},\]
where the coefficients $c(f,k)$ are independent of the geometry of the manifold and $a_k(D^2)$ are the global heat coefficients for the operator $D^2$. The latter are related to the local heat coefficients $a_k(D^2,x,y)$ (i.e. the Riemannian analogue of the Hadamard coefficients $V^k_x(y)$, which are what we are interested in) via
\[a_k(D^2)=\int\limits _M \tr(a_k(D^2,x,x))dVol(x).\]
This asymptotic expansion is not just some curious property of the spectral action, but actually the reason why it is desirable as an action from a physical perspective. For example, $a_1(D^2)$ corresponds to the Einstein-Hilbert action.

As the spacetime we actually live in is Lorentzian rather than Riemannian, it would be desirable to have this spectral action also for the Lorentzian case. However, the formula above is not well-defined in this case. A Lorentzian spectral action for a certain class of spacetimes, where the d'Alembert-operator is self-adjoint, has been developed in \cite{DaWr}.

Our aim is to explore the case where the wave operator is no longer self-adjoint (see \cite{Kam} for an example of a geodesically complete globally hyperbolic manifold where this happens) and thus there is no functional calculus available. The best criterion for an analogue of the spectral action (which would no longer be "spectral" in the absence of spectral theory) would be the existence of an asymptotic expansion similar to the one above, with Hadamard coefficients instead of local heat coefficients. 

As the Lorentzian manifolds one generally considers are not compact, however, the integral used to define the global coefficients will generally not exist. Thus the best one can hope for is a corresponding ``action density'' that could then be integrated over compact sets. This means, to get an analogue of the spectral action, one would like a function on the manifold, with extra parameter $s$, that has an asymptotic expansion
\[L(x,s)\stackrel{s\rightarrow0}{\sim}\insum{k}c(f,k)\tr(V_x^k(x))s^{2k-\delta}\]
and is defined purely in terms of functional analytic quantities that could also be used in noncommutative geometry. 

This is both a motivation for and a possible approach toward finding a functional analytic formula for the (trace of) Hadamard coefficients on the diagonal (i.e. $V^k_x(x)$). If one has a formula for Hadamard coefficients, the postulation of such an expansion might reveal what an analogue of the spectral action should look like.
Conversely, if a function with the above asymptotic expansion could be defined, one could extract from this a formula for the Hadamard coefficients by extracting the individual expansion coefficients. We will not succeed in finding a suitable analogue for the spectral action in this thesis. However, we will construct a function with a similar expansion (with additional error terms) in chapter \ref{timeint}. This will still allow us (though in a more complicated way) to extract a formula for the Hadamard coefficients.

\subsection{Some conventions and notation}
\label{abuseofnotation}
We assume that all manifolds throughout this thesis are imbued with a canonical volume measure that has a smooth density with respect to the Lebesgue measure in any coordinate chart (this holds for manifolds with smooth metrics, which really are the only cases we care about). For these we will write $dx$ instead of $d\Vol(x)$.
As usual, we will identify a locally integrable function $f$ on a space $X$ with a canonical volume measure with the distribution $T_f$ given by
\[T_f(\phi)=\int\limits_X\phi f d\Vol.\]
In order to avoid confusion from this abuse of notation, we will write evaluation of distributions with square brackets:
\[\eta[\phi]:=\eta(\phi)\]
for some distribution $\eta$. By our identification, we thus denote
\[f[\phi]:=T_f(\phi)\]
for any locally integrable function $f$.

Complex numbers will sometimes be identified with the corresponding multiple of the identity or the constant function with that value.

Occasionally, we will use $C$ or $C'$ without further specification. These will denote arbitrary constants, that may take new values for each new equation.

We will often encounter definitions and statements that work in both directions of time symetrically, usually distinguished by an index $+$ or $-$. In order to avoid writing everything twice, we use the symbol $\pm$ to denote both cases at once. The use of $\pm$ or $\mp$ in a definition or theorem indicates that this should hold both with the upper sign chosen everywhere or with the lower sign chosen everywhere in that definition or theorem. In some cases we also need to change the wording depending on the choice of time alignment. In that case we will use a "/" to indicate that the first word is to be used with the upper sign and the second word is to be used with the lower sign.

For example, the statement "If $A$ is past/future compact, $J_\pm(A)\cap J_\mp(x)$ is compact." means "If $A$ is past compact, $J_+(A)\cap J_-(x)$ is compact and if $A$ is future compact, $J_-(A)\cap J_+(x)$ is compact.".

We introduce some identifications and abuse of notations for vector bundles:
\begin{gn}
	\label{identify}
	Throughout this thesis, for any vector spaces $V$, $W$, $U$ (mostly, these will be fibres of vector bundles), we will canonically identify:
	\begin{itemize}
		\item $\C^*$ with $\C$ (identifying $\lambda Id$ with $\lambda$)
		\item $V^{**}$ with $V$
		\item $V\otimes \C$ with $V$
		\item $(V\otimes W)^*$ with $V^*\otimes W^*$
		\item $W\otimes V^*$ with $Hom(V,W)$
		\item $(V\otimes W)\otimes U$ with $V\otimes(W\otimes U)$
		\item $V\otimes W$ with $W\otimes V$ if $V$ and $W$ are distinct and this identification is required for definitions to make sense.
	\end{itemize} 
	Moreover, we omit composition and evaluation of dual vectors and fibrewise linear maps from the notation, so that e.g. for $v\in V$, $w\in W^*$ and $A\in Hom(V,W)\cong W\otimes V^*\cong ((W^*)\otimes V)^*$, we write:
	\[wAv=w(A(v))=A(w\otimes v).\]
	We identify scalar functions on some manifold $X$ with sections of the vector bundle $X\times \C$.
\end{gn}
The main purpose of most of this abuse of notation is to be able to "multiply" sections in suitable vector bundles without having to define a zoo of canonical bilinear maps and having to do case distinctions for all of them. 

We introduce notation for Schwartz kernels (or integral kernels) of operators:
\begin{df}
	\label{SKnotation}
	Let $E$ anf $F$ be vector bundles. For a continuous linear operator $T\colon \Gamma_c(E)\rightarrow \D'(F)$, we denote by $\K(T)$ its Schwartz kernel, i.e. the unique distribution in $F\boxtimes  E^*$ such that for any test functions $\psi\in \Gamma_c(F^*)$ and $\phi\in \Gamma_c(E)$, we have
	\[\K(T)[\psi\otimes \phi]=T(\phi)[\psi].\]
\end{df}
The Schwartz kernel theorem asserts that this is well-defined.

\begin{gn}
	For some space of sections or distributions $D(E)$ in some vector bundle $E$ and a finite dimensional vector space $W$, we identify the corresponding space of sections/distributions $D(E\otimes W)$ with $D(E)\otimes W$. For a linear operator $T$ acting on $D(E)$, $\eta\in D(E)$ and $w\in W$, we define
	\[T(\eta\otimes w):=T(\eta)\otimes w,\]
	i.e. we identify $T$ with $T\otimes id$. This corresponds to applying $T$ to every "component" with respect to $W$.
\end{gn}
This will be needed because we will often have to deal with distributions in $E\otimes E^*_x$ for some vector bundle $E$ and $x\in M$, which correspond to Schwartz kernels of operators in $E$, restricted to $x$ in the second component. Application in the first component as defined here then corresponds to composition of operators.

\subsection{Asymptotic expansions}
An important notion in the following will be that of an asymptotic expansion. It is customarily written with an infinite summation sign, even though the sum as written may not converge (at least not in a usual topology - one could construct non-Hausdorff topologies, in which these asymptotic expansions actually correspond to limits). There will be two types of asymptotic expansions in this thesis: Asymptotic expansions in a parameter, where the sum gives an approximation up to arbitrary powers of some expansion parameter $s$ and asymptotic expansions in differentiability orders, where the approximation is up to arbitrarily differentiable functions. In the latter case, we will also need extra notation for the case where one argument is fixed. The basic idea in all cases is that finite partial sums of the right hand side approximate the left hand side up to arbitrary "order", which is given either by decay in a parameter or by differentiability.
\begin{df}\ 
	\label{dfasympt}
	\begin{itemize}
		\item Assume that $F$ and $f_n$ for $n\in \N$ are functions from an open interval containing $0$ to $\C$. We write
		\[F(s)\stackrel{s\rightarrow 0}{\sim}\insum{n} f_n(s)\]
		if and only if for any $m\in \N$, there is $N_0\in \N$, $C\in \R$ and $\epsilon>0$ such that for $|s|<\epsilon$ and $N\geq N_0$, we have 
		\[\Big|F(s)-\sum\limits_{n=0}^Nf_n(s)\Big|\leq Cs^m.\]
		\item Assume that $F$ and $f_n$ for $n\in \N$ are distributions on a manifold $X$. We write
		\[F\sim\insum{n} f_n\]
		if and only if for any $m\in \N$, there is $N_0\in \N$, such that we have for $N\geq N_0$
		\[F-\sum\limits_{n=0}^Nf_n\in C^m(X).\]
		\item Let $X$ and $Y$ be manifolds and assume that $F(x)$ and $f_n(x)$ for $n\in \N$ are distributions on $Y$ for any $x\in X$. We write
		\[F(x)\sim_x\insum{n} f_n(x)\]
		if and only if for any $m\in \N$, there is $N_0\in \N$ and $R\in C^k(Y\times X)$, such that we have for all $x\in X$ and $N\geq N_0$
		\[F(x)-\sum\limits_{n=0}^Nf_n(x)=R(\cdot,x).\]
	\end{itemize}
In each case, iterated sums are defined accordingly: For every $m$, sufficiently large partial sums of the right hand side should approximate the left hand side up to any differentiability order or decay order in $s$. 
\end{df}
Asymptotic expansions are robust under reordering or repackaging the summands:
\begin{prop}
	Let $\sim_*$ be any of the three asymptotic expansion types defined above.
	\begin{enumerate}
		\item Whether an asymptotic expansion holds only depends on the set of summands on the right hand side, not on the order of their summation.
		\item Assume that $A$, $A_k$ and $a_{kl}$ are elements in some space where $\sim_*$ makes sense. Assume furthermore that 
		\[A_k\sim_* \insum{l}a_{kl}.\] 
		Then we have
		\[A\sim_*\insum{k,l}a_{kl}\]
		if and only if 
		\[A\sim_*\insum{k}A_k.\]
	\end{enumerate}
\end{prop}
\begin{proof}[Proof sketch]
	Let $m\in \N$ be arbitrary. If any asymptotic expansion holds in one of the above cases, all but finally many summands involved must be $C^m$ or $O(s^m)$ (depending on the type of asymptotic expansion), since otherwise we could not have equality up to $C^m$ resp. $O(s^m)$ for arbitrarily large partial sums. Thus, to check that the expansion holds up to degree $m$, it suffices to consider only finitely many summands. As the statements above hold for finite sums, they also hold for asymptotic expansions.
\end{proof}

\subsection{The Mellin transform}
\label{Mellin}
\begin{df}
	For a function $f\colon (0,\infty)\rightarrow \C$, define its Mellin transform via
	\[\M(f)(\alpha):=\intP f(x)x^{\alpha-1}dx\]
	for any $\alpha$ such that the integral exists. If $f$ is defined on a larger domain, its Mellin transform is defined as that of its restriction to $(0,\infty)$.
\end{df}
This is the analogue of the Fourier transform when identifying $\R$ with $(0,\infty)$ via the exponential map. For $f$ compactly supported in $(0,\infty)$, the integral always exists for any $\alpha$ and is holomorphic as a function of $\alpha$. For $f\in C_c^\infty(\R)$, the integral is still defined and holomorphic if $\Re(\alpha)>0$. We define it for $\Re(\alpha)<0$ by meromorphic continuation:
\begin{defprop}
	For $f\in C_c^\infty(M)$, the Mellin transform of $f$ admits a meromorphic extension to $\C$ with simple poles at most at negative integers (including zero), whose residues are given by
	\[\Res\limits_{\alpha=-k}\M(f)(\alpha)=\frac{(-1)^k}{k!}f^{(k)}(0)\]
	and no other poles. We also denote this extension by $\M(f)$.
\end{defprop}
\begin{rem}
	There is a more general connection between asymptotic expansions of a function at zero or infinity and the pole structure of its Mellin transform (see e.g. Proposition 2.1.2 and the preceding paragraph in \cite{Les} or \cite[Theorem 3]{FGD}).
\end{rem}
\begin{proof}
	For $\Re(\alpha)>0$, partial integration yields
	\[\M(f')(\alpha+1)=-\alpha \M(f)(\alpha).\]
	Using the left hand side of this equality as a definition for the right hand side, one can successively extend $\M(f)$ to the left meromorphically, with poles only at negative integers arising from the case $\alpha=0$. Taking the limit as $\alpha$ aproaches zero of the above equality and using the fundamental theorem of calculus, we obtain
	\[\Res_{\alpha=0}\M(f)(\alpha)=-\M(f')(1)=-\intP f'(x)dx=f(0).\]
	Applying the previous equality iteratively (by holomorphic continuation, it holds for arbitrary values of $\alpha$), we obtain
	\[\Res_{\alpha=-k}\M(f)(\alpha)=\frac{(-1)^k}{k!}\Res_{\alpha=0}\M(f^{(k)})(\alpha)=\frac{(-1)^k}{k!}f^{(k)}(0).\qedhere\]
\end{proof}
\subsection{Binomial coefficients}
We will need to use binomial coefficients for arbitrary complex arguments. Those are defined in terms of the Gamma function:
\[\binom{\alpha}{\beta}:=\frac{\Gamma(\alpha+1)}{\Gamma(\beta+1)\Gamma(\alpha-\beta+1)}\]
for any $\beta\in \C$ and $\alpha\in\C\backslash -\N$. In case both $\alpha$ is a negative integer and $\beta$ is also an integer, we define the binomial coefficient via continuous extension in $\alpha$, leaving $\beta$ fixed:
\[\binom{\alpha}{\beta}:=\lim\limits_{\alpha'\rightarrow\alpha}\binom{\alpha'}{\beta}.\]
This is holomorphic in $\alpha$ for fixed $\beta$. A consequence of this is that $\binom{\alpha}{\beta}$ vanishes whenever $\beta$ is a negative integer.

\section{Lorentzian geometry}
In this section we collect the facts about Lorentzian geometry that will be required for this thesis. More information on Lorentzian geometry can be found e.g. in \cite{One}.
\subsection{Lorentzian vector spaces}
We will first consider basic notions in Lorentzian vector spaces, before extending them to general Lorentzian manifolds.
\begin{df}
	A Lorentzian bilinear form on a d-dimensional vector space is a symmetric bilinear form $\eta$ with signature $(d-1,1)$. A Lorentzian vector space is a vector space with a Lorentzian bilinear form.
\end{df}
The basic example of a Lorentzian vector space is Minkowski space: $\R^d$ with the Lorentzian bilinear form 
\[\eta(x,y)=-x_0y_0+\sum\limits_{i=1}^{d-1}x_iy_i\]
(the indexing convention is from relativistic physics, where time gets the index $0$). Every Lorentzian vector space is isometric to Minkowski space (more or less directly by the definition of the signature), so most theorems only need to be shown for Minkowski space. Eventually, they will usually be applied to tangent spaces of a Lorentzian manifold.

In the following, let $V$ be a Lorentzian vector space with bilinear form $\eta$. The bilinear form singles out a ``general direction'' that has a different sign than the others (a negative sign with the above definition, but there are different sign conventions). This direction is regarded as ``time'', the others as ``space'', as made precise in the following definition:
\begin{wrapfigure}{r}{0.3\textwidth}
	\includegraphics{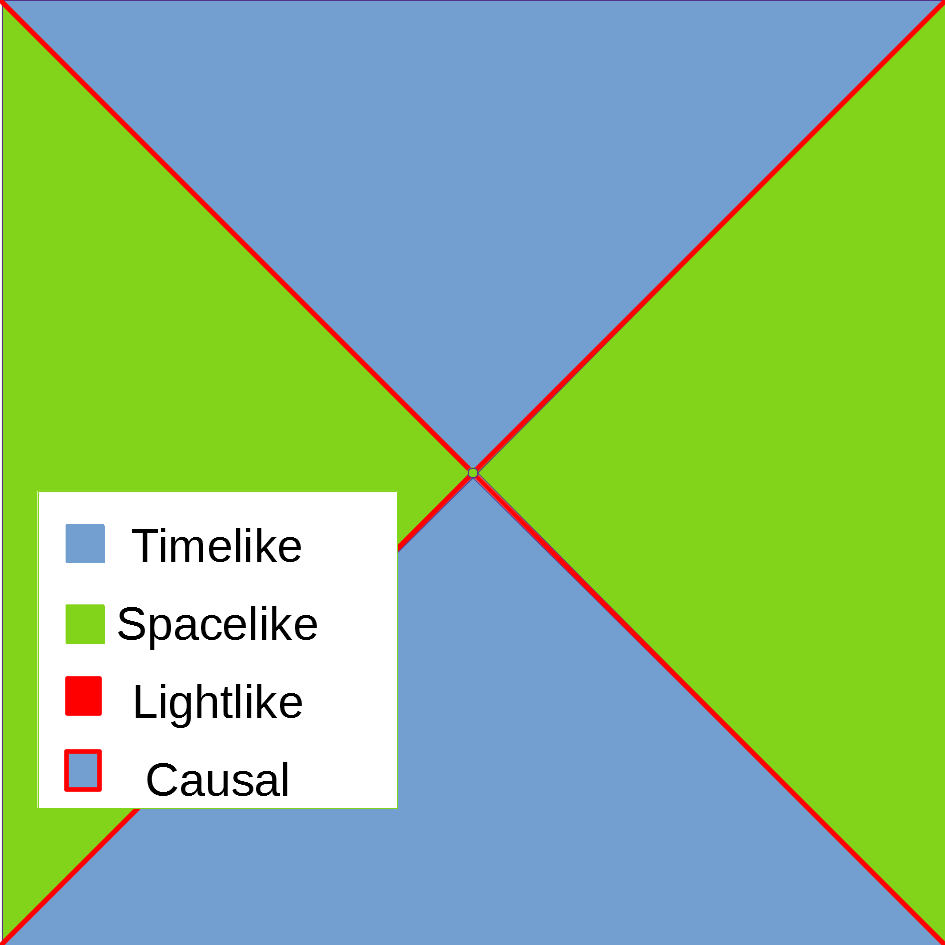}
	
\end{wrapfigure}
\begin{df}
	A vector $v\in V\O$ is called
	\begin{itemize}
		\item timelike, if $\eta(v,v)<0$
		\item spacelike, if $\eta(v,v)>0$
		\item lightlike, if $\eta(v,v)=0$
		\item causal, if $\eta(v,v)\leq 0.$
	\end{itemize}
$0$ is considered as spacelike.
\end{df}
The Minkowski-orthogonal complement of a timelike vector is spacelike:
\begin{lemma}
	\label{orthotime}
	If $x,y\in V$ satisfy $\eta(x,y)=0$ and $x$ is timelike, then $y$ is spacelike. If instead $y\in V^*$ and $y(x)=0$, then $y$ is spacelike with respect to the induced metric on $V^*$.
\end{lemma}
\begin{proof}
	Without loss of generality, assume that $V$ is Minkowski space. Let $x=(x_0,x_r)$ and $y=(y_0,y_r)$. $x$ being timelike is equivalent to $|x_0|>\|x_r\|$.
	We have
	\[|\eta(x,y)|=|-x_0y_0+\<x_r,y_r\>|\geq|x_0||y_0|-\|x_r\|\|y_r\|\geq|x_0|(|y_0|-\|y_r\|),\]
	with equality only if $y_r=0$. Thus if the left hand side is zero, we either have $y=0$ or $|y_0|-\|y_r\|<0$, which means that $y$ is spacelike. This concludes the proof of the first statement.
	
	If $y$ is an element of the dual space, it is induced by some $y^\sharp\in V$ via
	\[y(w)=\eta(y^\sharp,w).\]
	The induced form on the dual is defined so that
	\[\eta(y,y):=\eta(y^\sharp,y^\sharp),\]
	so $y$ is spacelike if and only if $y^\sharp$ is.
	If
	\[\eta(y^\sharp,x)=y(x)=0,\]
	this is the case by the previous part.
\end{proof}
The causal vectors form two cones, known as the (solid) light cones.
\begin{prop}
	The set of causal vectors in $V$ has two connected components, each of which is convex. If $V$ is Minkowski space, these consist of the vectors with positive/negative $0$-component.
	A timelike vector $x$ and a causal vector $y$ are in the same connected component if and only if $\eta(x,y)<0$.
\end{prop}
\begin{proof}	
	Without loss of generality, assume that $V$ is Minkowski space. Let $x=(x_0,x_r)$ and $y=(y_0,y_r)$ be causal vectors with $x_0,y_0>0$.
	We have
	\[\eta(x,y)=-x_0y_0+\<x_r,y_r\>\leq -x_0y_0+\|x_r\|\|y_r\|\leq -x_0y_0+x_0y_0=0\]	
	and thus for $a,b\in (0,1)$
	\[\eta(ax+by,ax+by)=a^2\eta(x,x)+b^2\eta(y,y)+2ab\eta(x,y)<0.\]
	Thus the set of causal vectors with positive $0$-component is convex and hence connected. The same holds for those with negative $0$ component, because $-x$ is causal if and only if $x$ is. As there are no causal vectors with $x_0=0$, those two sets are disconnected from each other. This proves the first two statements.
	
	If $x$ is timelike, Lemma \ref{orthotime} implies the function $\eta(x,\cdot)$ does not take the value 0 on the set of causal vectors. As it is also continuous, the preimage of $(0,\infty)$ and $(-\infty,0)$ must be both open and closed. As neither is empty, since they contain $-x$ resp. $x$, they must be the two connected components decribed above. Thus the connected component of $x$ consists of all those $y$ with $\eta(x,y)<0$.	
\end{proof}
In many cases, it is desirable to be able to have a time direction, i.e. to be able to talk about future and past. This is called a time-orientation. Rather than choosing a specific vector as the future direction, one only specifies which light cone points into the future and which one points into the past.
\begin{df}
	A time-orientation on a Lorentzian vector space $V$ is a choice of one connected component of the set of causal vectors. Vectors in this component will be called future oriented, vectors in the other component will be called past oriented. On Minkowski space, one canonically takes those vectors with positive $0$-component to be future-directed. In the following, we will assume that all Lorentzian vector spaces are time-oriented.
\end{df}
\subsection{Lorentzian manifolds}
We now get to Lorentzian manifolds. These are manifolds where each tangent space is a Lorentzian vector space.
\begin{df}
	A Lorentzian manifold is a smooth manifold $M$ (of dimension $d$) with a smooth $(0,2)$-tensor field $g$ that is a Lorentzian bilinear form on each fibre $T_x M$. g is called a Lorentzian metric.
\end{df}

In order to talk about future and past, one has to choose a time-orientation at each point. A timelike vector field naturally defines such a time-orientation in a continuous way. This does not exist on arbitrary Lorentzian maniflolds, so we have to assume its existence. Note that this vector field is far from unique or canonical, but only used to ensure the choice of orientations is continuous.
\begin{df}
	A continuous timelike vector field $\tau$ on $M$ induces a time-orientation on each tangent space: A causal vector $v\in T_x M$ is called future oriented (with respect to $\tau$), if $g(v,\tau(x))<0$ and past oriented otherwise. A time-orientation on $M$ is a collection of time-orientations on each $T_x M$ that arises in this way. $M$ is called time-orientable, if a time-orientation exists, and time-oriented, if one has been chosen.
\end{df}
\begin{gn}
	\label{gnM}
	Throughout this paper, $M$ will be a time-oriented Lorentzian manifold of dimension $d\geq2$ with metric $g$.
\end{gn}
We will later additionally assume that $M$ is globally hyperbolic (once we have defined what that means).

The definitions made above for vectors can be extended to curves by looking at their tangent vectors:
\begin{df}
	A piecewise smooth curve $\gamma$ is called future/past oriented, if the same is true for every tangent vector $\gamma'(t)$. A future or past oriented curve is called  timelike/lightlike/causal if the same is true for every tangent vector.
\end{df}
Using this, one can also talk about the future and past of a point $p\in M$, i.e. all points that can be reached from $p$ via a future/past oriented curve:
\begin{df}
	\label{futurepast}
	The causal future/past of $A\subset M$, denoted by $J_\pm(A)$, is the set of all points $y\in M$. such that there is a future/past oriented causal curve $\gamma\colon [s,t]\rightarrow M$ (for $s,t\in\R$) with $\gamma(s)\in A$ and $\gamma(t)=y$.
	The causal future/past of a point $x\in M$ is that of the corresponding one-point set
	\[J_\pm(x):=J_\pm(\{x\}).\]
	Set $J(A):=J_+(A)\cup J_-(A)$. Here we allow trivial curves, i.e. $x\in J_\pm(x)$.
\end{df}
This causal structure is central to Lorentzian geometry and will be used frequently throughout this thesis. Whenever it is unclear in which manifolds this is done, the manifold will be indicated as a superscript.
We observe that $J_\pm(J_\pm(A))=J_\pm(A)$, as we can concatenate future/past directed curves.

An important condition we have to impose on our manifold is that of global hyperbolicity. This means that there is a surface, called a Cauchy hypersurface, that can be thought of as all of space at one instant of time.

\begin{df}
	A Cauchy hypersurface in $M$ is a hypersurface that is intersected exactly once by any inextensible timelike curve. 
	
	A connected timeoriented Lorentzian manifold is called globally hyperbolic, if it contains a Cauchy hypersurface.
\end{df}
One example of a globally hyperbolic spacetime is Minkowski space.
 There are several equivalent ways of characterizing global hyperbolicity (see \cite[Theorem 1.3.10]{BGP})
\begin{thm}
	\label{Ghyper}
	The following are equivalent:
	\begin{enumerate}
		\item For any $x,y\in M$, $J_+(x)\cap J_-(y)$ is compact and every point in $M$ has a neighborhood basis of open sets $V$ such that every causal curve with endpoints in $V$ lies entirely in $V$ (the second part is known as strong causality).
		\item $M$ contains a Cauchy hypersurface.
		\item There is a $d-1$ dimensional manifold $\Sigma$, a family of Riemannian metrics $(g_t)_{t\in \R}$ on $\Sigma$ and a smooth positive function $\beta$ on $\Sigma\times \R$ such that $M$ is isometric to $\Sigma \times \R$ with metric $g_t\oplus -\beta dt^2$. This can be chosen so that each $\Sigma\times \{t\}$ is a smooth spacelike Cauchy hypersurface.
	\end{enumerate}
\end{thm}
\begin{gn}
	From now on, we will always assume that $M$ is globally hyperbolic.
\end{gn}
One consequence is that no point is in the interior of its own future or past:
\begin{prop}
	\label{xinbound}
	Each $x\in M$ is in the topological boundary of $J(x)$.
\end{prop}
\begin{proof}
	$x$ is always contained in $J(x)$. By Theorem \ref{Ghyper}, $x$ is contained in some Cauchy hypersurface $\Sigma$. As this has dimension $d-1\geq 1$, there is a sequence in $\Sigma\backslash\{x\}$ that converges to $x$. As every causal curve through $x$ cannot intersect $\Sigma$ in any other point (by definition of a Cauchy hypersurface), that sequence is in the complement of $J(x)$, so we can conclude that $x$ is in its boundary.
\end{proof}
We will often need to talk about sets that are compact in the past/future:
\begin{df}
	A set $A\subseteq M$ is called past/future compact, if for any $x\in M$, the intersection $A\cap J_\mp(x)$ is compact. It is called strictly past/future compact, if it is closed and contained in the future/past of a compact set.
\end{df}

Strictly past/future compact sets are indeed, as the name suggests, past/future compact. Moreover, one may replace the point in the definition with an arbitrary compact set.
\begin{prop}
	If $K$ and $K'$ are compact, then $J_+(K)\cap J_-(K')$ is compact. In particular, every strictly past/future compact set is past/future compact.
\end{prop}
(see Lemma A.5.7 in \cite{BGP}, with the second statement following by inserting a point for $K'$ resp. $K$)
\begin{prop}
	If $A$ is past/future compact and $K$ is compact, then
	\[A\cap J_{\mp}(K)\]
	is compact.
\end{prop}
(see Theorem 3.1 in \cite{Sa})

One further definition that we will need is that of causal compatibility, which essentially states that the causal structure of a subset is compatible with that of the manifold:
\begin{df}
	An open subset $U$ of $M$ is called causally compatible, if for any $A\subseteq U$
	\[J_\pm^U(A)=J_\pm^M(A)\cap U,\]
	where superscripts on $J_\pm$ indicate in which manifold the future/past is taken. 
\end{df} 
In other words, any two points in $U$ that can be connected by a causal curve in $M$ can also be connected by a causal curve in $U$.

As in the Riemannian case, one can define a Levi-Civita connection in the Lorentzian case and use this to define geodesics and exponential maps.
The exponential map is particularly useful on subsets where they are diffeomorphisms
\begin{df}
	An open subset $U\subseteq M$ is called geodesically convex (or convex, for short), if the exponential map defines a diffeomorphism from an open neighborhood of the zero section of $TU$ to $U\times U$.
\end{df}
Every point in $M$ has a neighborhood basis of convex subsets.
For such $U$, the exponential map defines a diffeomorphism
from an open subset of $TU$ to $U\times U$. 
The exponential map also preserves causality in the following sense (see \cite[Proposition 4.5.1]{HaEl}\footnote{The book works with $d=4$, as it is motivated by physical applications, but this is not used anywhere in the proof.}):
\begin{prop}
	\label{exJ}
	Let $U\subseteq M$ be convex. Any $x,y\in U$ that can be joined by a future/past oriented causal (resp. timelike) curve can also be joined by a future/past oriented causal (resp. timelike) geodesic in $U$. This means that
	\[J_\pm(x)=exp_x(J_\pm(0)).\]
\end{prop}

\section{Function spaces and distributions in vector bundles}
We now define spaces of functions and distributions that will be used throughout the paper. Readers will probably be familiar with at least some of them, but we also want to fix notation. Let $E$ be a vector bundle over a manifold $X$. Let $\kappa$ be a locally finite collection of charts covering $X$ and $\tau$ a locally finite collection of local trivializations.
\begin{df}
	For a compact set $K\subseteq X$,  and $n\in \N$ , we define $C^n$-seminorms on $n$ times continuously differentiable sections by
	\[\|f\|_{C^n(K)}:=\sup\{\|\Psi\circ f|_K\circ\phi^{-1}\|_{C^n}\mid\phi\in \kappa, \Psi\in\tau\}.\]
	Here we use the convention that all functions in a composition automatically have their domain restricted to those elements on which the composition is defined.
\end{df}
In general, this depends on the choice of $\tau$ and $\kappa$, but the seminorms for different choices are equivalent. As we do not care about numerical values in this thesis, we can thus assume that we have chosen a locally finite collection of charts and local trivializations for each vector bundle that will be considered in this thesis and interpret all $C^k$-norms as being taken with respect to those.

The following spaces are fairly standard:
\begin{df}
	\label{dfsections}
	Let $E$ be a vector bundle over a manifold $X$.
	For each $k\in\N\cup\{\infty\}$, let $\Gamma^k(E)$ be the space of all $k$ times continuously differentiable sections with topology induced by the seminorms $\|\cdot\|_{C^n(K)}$ for all compact subsets $K\subseteq M$ and $n= k$ in case $k\in \N$ or all $n\in \N$ in case $k=\infty$. 
	
	For $A\subseteq X$, let $\Gamma^k_A(E)$ be those sections whose support is contained in $A$, with the subspace topology.
	
	Let $\Gamma^k_c(E)$ be the space of all compactly supported $C^k$-sections, with the limit topology obtained from the spaces $\Gamma^k_K(E)$ for $K\subseteq X$ compact.
	
	In case $k=\infty$, we may omit the superscript and just write $\Gamma$ instead of $\Gamma^\infty$.
\end{df}
Somewhat less well-known variants in Lorentzian geometry are the following spaces:
\begin{df}
	Let $E$ be a vector bundle over $M$.
	For $k\in\N\cup\{\infty\}$, let $\Gamma^k_\pm(E)$ be the space of past/future compactly supported $C^k$-sections, with the limit topology obtained from all spaces $\Gamma^k_A(E)$ for $A\subseteq M$ past/future compact.
	
	Let $\Gamma^k_{s\pm}(E)$ be the space of strictly past/future compactly supported $C^k$-sections, with the limit topology obtained from all spaces $\Gamma^k_A(E)$ for $A\subseteq M$ strictly past/future compact.
	
	Again, we may omit $k$ if it is $\infty$.
\end{df}
As a special case of the above definitions, when $E=M\times \C$, we obtain spaces of complex valued functions, in which case we will write $C^k_\cdot(M)$ instead of $\Gamma^k_\cdot(E)$.
Taking duals, we obtain spaces of distributions:
\begin{df}
	Define the following spaces of distributions in a vector bundle $E$ over $M$:
	\begin{itemize}
	 \item The space of distributions $\D'(E)$ is the space of continuous linear maps from $\Gamma_c(E^*)$ to $\C$.
	 \item The space of past/future compactly supported  distributions $\D'_\pm(E)$ is the space of continuous linear maps from $\Gamma_{s\mp}(E^*)$ to $\C$.
	 \item The space of strictly past/future compactly supported distributions $\D'_{s\pm}(E)$ is the space of continuous linear maps from $\Gamma_{\mp}(E^*)$ to $\C$.
	 \end{itemize}
	 On each of these spaces, we use the topologies induced by the seminorms
	 \[\|\eta\|_B:=\sup\limits_{\phi\in B}\|\eta(\phi)\|\]
	 for all bounded subsets of the corresponding test space.
\end{df}
\begin{rem}
	Instead of the (strong) topology used here, the weak*-topology is often used on spaces of distributions. As both topologies induce the same convergence of sequences (and thus also the same notion of holomorphicity), they are essentially equivalent for the deliberations in this thesis. The reason we will use the strong topology is that when we get to wavefront calculus, this will make certain operations continuous rather than just sequentially continuous, which makes this choice of topology seem more natural.
\end{rem}	 
	As the names would suggest, the space of (strictly) past/future compactly supported distributions is in canonical bijection with distributions whose support is (strictly) past/future compact, via restricting to $\Gamma_c(E)$ (see \cite[Lemma 2.13]{GH}).

Moreover, we have (see \cite[Lemma 2.15]{GH}):
\begin{prop}
	Smooth compactly supported sections are (sequentially) dense in any of the above spaces of distributions.
\end{prop}

There are two ways of transporting distributions from one space to another via some function $f$: pullback and pushforward
\begin{df}
	\label{pullforward}
	Let $f\colon X\rightarrow Y$ be a map between manifolds. 
	
	For a function $g$ on $Y$, define the pullback by 
	\[f^*g:=g\circ f.\] 	
	For a compactly supported distribution $\eta$ on $X$, define the pushforward by
	\[f_*\eta[\phi]:=\eta[\phi\circ f]\]
	for every testfunction $\phi$.
	If $E$ is a vector bundle on $Y$, these also define maps
	\[f^*\colon \Gamma(E)\rightarrow\Gamma(f^*(E))\]
	and
	\[f_*\colon \Gamma(f^*(E))\rightarrow \Gamma(E).\]
\end{df}
Neither of the two operations is defined on arbitrary distributions for general $f$: The pullback requires the existence of a continuous extension, while the pushforward requires $f$ to be proper on the support of the argument. 
Both will be well defined for all distributions in case $f$ is a diffeomorphism. Note that, even then, $f^*$ and $f^{-1}_*$ will not generally coincide, since the pullback of a distribution depends on the choice of canonical volume, which may not be preserved by $f$.

For sufficiently regular kernels, the trace of an operator is given by integrating its kernel over the diagonal:

\begin{thm}
	\label{Mercer}
	If $T$ is an operator acting on sections of a vector bundle $E$ over a manifold $X$ whose Schwartz kernel is smooth and compactly supported, then $T$ is trace class in $L^2(E)$ (with respect to any hermitian metric on $E$) and
	\[\tr(T)=\int\limits _X\tr(\K(T)(x,x))dx.\]
\end{thm}
\begin{proof}
	\cite[Theorem 1.1]{DeRu} asserts a stronger statement in the case of operators acting on scalar functions in a closed manifold $M$. This can be extended to the case at hand.
	First, choose a precompact open set with smooth boundary such that $\K(T)$ is supported in $U$. Consider the double of $\bar U$, i.e. two copies of $\bar U$ glued along $\partial U$ (see, for example \cite[Section VI.5]{Kos}). This gives us a closed  manifold containing $U$ as a subset. As everything outside of $U$ is irrelevant for the mapping propeties of $T$, this reduces the theorem to the case where $X$ is compact.
	
	By Swan's Lemma, every vector bundle over a compact manifold is a direct summand of a trivial bundle (see \cite[Corollary 5]{Swa}). Choosing a Hermition structure on the other summand and using the direct sum of the Hermitian structure on the trivial bundle (which is then isometric to $X\times R^N$), we obtain that $L^2$-sections in $E$ are isometrically embedded into those in $X\times \R^N$. All other function spaces are also embedded in a compatible way. As traces are preserved by isometric embeddings, it suffices to consider the case $E=X\times \R^N$.
	
	Let $p_i$ and $\iota_i$ denote composition with the projection on and the inclusion in the $i$-th component of $X\times \R^N$ (on distributions, $p_i$ can be defined by applying $\iota_i$ to the test function). Consider the components
	\[T_{ij}:=p_i\circ T\circ \iota_j.\]
	This has Schwartz kernel 
	\[\K(T_{ij})=\K(T)_{ij},\]
	as
	\[\K(T)_{ij}[\phi\otimes\psi]=\K(T)[\iota_i\phi\otimes\iota_j\psi]=p_iT\iota_j(\psi)[\phi]=\K(T_{ij})[\phi\otimes\psi].\]	
	As the claim holds for scalar functions, we can conclude that $T_{ij}$	
	defines a trace class operator with trace
	\begin{align*}
		\tr(T_{ij})&=\int\limits_X \K(T_{ij})(x,x) dx\\
		&= \int\limits_X \K(T)_{ij}(x,x)dx.
	\end{align*}
	As the $\iota_i$ are isometric embeddings into mutually orthogonal subspaces  (and $p_i$ are the corresponding projections), we have
	\[\tr(\iota_iT_{ij}p_j)=\delta_{ij}\tr(T_{ii}).\]	
	We can thus conclude that
	\[T=\sum_{i,j=1}^N\iota_iT_{ij}p_j\]
	is trace class with
	\begin{align*}
		\tr(T)&=\sum\limits_{i,j=1}^N\delta_{ij}\tr(T_{ii})\\
		&=\sum\limits_{i=1}^N\tr(T_{ii})\\
		&=\sum\limits_{i=1}^N \int\limits_X \K(T)_{ii}(x,x)dx\\
		&=\int\limits_X \tr(\K(T)(x,x))dx. \qedhere
	\end{align*}

\end{proof}

\section{Green's operators}
Normally hyperbolic operators are the Lorentzian equivalent of  Laplace-type operators. They carry information about the geometry of the manifold. Green's operators are something like inverses for these. The standard example for a normally hyperbolic operator on $M$ is the d'Alembertian. A canonical analogue for a normally hyperbolic operator should be available in the noncommutative setting as well, given as the square of the Dirac operator.
\begin{df}
	A normally hyperbolic operator on $M$ is a second order differential operator on a vector bundle $E$ over $M$ whose principal symbol is given by minus the metric $g$ (times the identity on fibres).
\end{df}
\begin{gn}
	For the rest of this thesis, let $P$ be a normally hyperbolic operator on a vector bundle $E$ over $M$. \label{gnP}
\end{gn}
\begin{rem}
	As $M$ and $P$ are arbitrary, theorems proved for $M$ or $P$ will hold for arbitrary globally hyperbolic manifolds and arbitrary normally hyperbolic operators. This will be exploited occasionally. Note that the restriction of a normally hyperbolic operator to an open subset is again normally hyperbolic.
\end{rem}
We shall occasionally need the connection on $E$ induced by $P$ (see \cite[Lemma 1.5.5]{BGP}):
\begin{prop}
	\label{nabla}
	There is a unique connection $\nabla$ on $E$ and a bundle endomorphism $B$ such that
	\[P=\square^\nabla+B,\]
	where $\square^\nabla$ denotes the connection d'Alembert operator associated to $\nabla$. 
\end{prop}
\begin{gn}
	From now on, let $\nabla$ denote the connection described above.
\end{gn}
The definition of the connection d'Alembertian may be found in \cite[Example 1.5.2]{BGP}, but is not relevant for our purpose. All we need about $\nabla$ is the following Leibniz rule (see \cite[Lemma 1.5.6]{BGP}):
\begin{prop}
	For $f\in C^\infty(M)$ and $g\in \Gamma(E)$, we have
	\[P(fg)=fPg-2\nabla_{grad(f)}g+(\square f)g,\]
	where $\square=-\Div\grad$ denotes the standard d'Alembertian.
\end{prop}

On the usual spaces of functions, $P$ will not be invertible. Instead, we have a well-posed Cauchy problem (see \cite[Theorem 3.2.11]{BGP}). This means, to get something like invertibility, we need to fix ``initial conditions''. Fixing these to be zero at ``time $\pm\infty$'', we obtain the advanced and retarded Green's operators (see \cite[Definition 3.4.1 and Corollary 3.4.3]{BGP}):
\begin{prop}
	There are unique operators 
	\[G^\pm\colon \Gamma_c(E)\rightarrow \Gamma(E)\]
	such that for any $\phi\in \Gamma_cy(E)$:
	\begin{enumerate}
		\item $PG^\pm\phi=\phi$
		\item $G^\pm P\phi=\phi$
		\item $\supp(G^\pm\phi)\subseteq J_\pm(\supp(\phi))$
	\end{enumerate}
\end{prop}
\begin{df}
	\label{dfG}
	The Operators $G^\pm$ are called the advanced/retarded Green's operator for $P$. They will be denoted as above, or, if it is not clear which normally hyperbolic operator is referred to, by $G_P^\pm$.
\end{df}
The Green's operators can be extended to other function spaces. The extension sketched here is described in more detail in \cite{GH}. First, we note that the value $G^\pm\phi(x)$ only depends on the value of $\phi$ in the past/future of $x$, as a change in $\phi$ only affects the future/past of the difference. Thus, rather than requiring $\supp(\phi)$ to be compact, it suffices to require that $\supp(\phi)\cap J_\mp(x)$ is compact in order to define $G^\pm\phi(x)$: We can multiply by a cut-off function $\chi$ that is 1 around $\supp(\phi)\cap J_\mp(x)$.  
$G^\pm(\chi\phi)(x)$ will not depend on the choice of $\chi$ and can thus be used as a definition for $G^\pm\phi(x)$. That means,  if $\phi$ has past/future compact support, $G^\pm\phi(x)$ can be defined for every $x$. In this way $G^\pm$ can be extended to all past/future compactly supported smooth functions.
$G^\pm$ can then also be extended to distributions in the usual way using duality. Overall, we get a map on past/future compactly supported distributions that still has all the defining properties.

\begin{prop}
	\label{extendedG}
	$G^\pm$ has unique continuous extensions to maps $\D'_\pm(E)\rightarrow \D'_\pm(E)$ and $\Gamma_\pm(E) \rightarrow \Gamma_\pm(E)$ (that will be denoted the same way).
	Moreover, for any $\phi\in \D'_\pm(E)$, we still have
	\begin{enumerate}
		\item $PG^\pm\phi=\phi$
		\item $G^\pm P\phi=\phi$
		\item $\supp(G^\pm\phi)\subseteq J_\pm(\supp(\phi))$
	\end{enumerate}
\end{prop}
\begin{proof}
	See \cite[Theorem 3.8, Corollary 3.11 and Lemma 4.1]{GH}
\end{proof}
\begin{rem}
	As a consequence of the above, we can view $G_P^\pm$ as an actual inverse to $P$, by considering both as operators $\Gamma_\pm(M) \rightarrow \Gamma_\pm(M)$.
\end{rem}
For suitable subsets, restricting the Green's operator on $M$ gives the Green's operator on the subset. Global hyperbolicity ensures that the subset has a unique Green's operator, while causal compatibility ensures that there is no influence from the remainder of $M$.
\begin{prop}
	Let $U\subseteq M$ be causally compatible and globally hyperbolic. Then for $\phi\in \D'_\pm(E)$, we have
	\[G^\pm_{P|_U}(\phi|_U)=(G^\pm_{P}\phi)|_U\]
\end{prop}
\begin{proof}
	For compactly supported smooth $\phi$, this follows from \cite[Proposition 3.5.1.]{BGP}. By continuity, the statement extends to distributions.
\end{proof}
\begin{df}
	\label{df*}
	For an operator $T\colon \Gamma_c(E)\rightarrow \Gamma(E)$, define the formal adjoint (if it exists) to be the operator $T^*\colon \Gamma_c(E^*)\rightarrow \Gamma(E^*)$ satisfying for all $\phi\in \Gamma_c(E^*)$, $\psi\in \Gamma_c(E)$
	\[\int\limits_M \phi T\psi d\Vol=\int\limits_M T^*\phi \psi d\Vol.\]
\end{df}
\begin{rem}
	Note that taking the adjoint is linear, not antilinear as the Hilbert space adjoint. The reason for this is that for Hilbert spaces the identification with the dual is antilinear. Beware that the adjoint as defined here does not coincide with the Hilbert space adjoint for scalar valued functions when linearly identyfiyng $\C^*$ with $\C$.
\end{rem}
The formal adjoint of a normally hyperbolic operator is again normally hyperbolic (as partial integration doesn't change the principal symbol of a second order differential operator).

The dual of the Green's operators are the Green's operators of the dual operator. However, the advanced/retarded sign gets flipped in the process. As one might expect, this also works when taking powers.
\begin{prop}
	\label{Gdual}
	For $m\in \N$, $(G^\mp_{P^*})^m$ is a formal adjoint for $(G^\pm_P)^m$.
\end{prop}
\begin{proof}
	Let $\phi\in \Gamma_c(E)$ and $\psi\in \Gamma_c(E^*)$. Choose $\chi\in C_c^\infty(M)$ that is constantly $1$ around $J_\pm(\supp(\phi))\cap J_\mp(\supp(\psi))$. Then $\chi$ is $1$ around \[\supp((G^\pm_P)^m\phi)\cap \supp((G^\mp_{P^*})^m\psi)\] and thus we can calculate
	\begin{align*}
		&\int\limits_M\psi(x)((G^\pm_P)^m\phi(x))dx\\
		&=\int\limits_M (P^*)^m(G^\mp_{P^*})^m\psi(x)((G^\pm_P)^m\phi(x))dx\\
		&=\int\limits_M (P^*)^m\chi(G^\mp_{P^*})^m\psi(x)(\chi (G^\pm_P)^m\phi(x))dx\\
		&=\int\limits_M \chi(G^\mp_{P^*})^m\psi(x)(P^m\chi (G^\pm_P)^m\phi(x))dx\\
		&=\int\limits_M (G^\mp_{P^*})^m\psi(x)(P^m (G^\pm_P)^m\phi(x))dx\\
		&=\int\limits_M (G^\mp_{P^*})^m\psi(x)(\phi(x))dx. \qedhere
	\end{align*}
\end{proof}

\section{Riesz distributions and Hadamard coefficients}
\label{Rieszhadamard}
In this section, we introduce Riesz distributions and Hadamard coefficients, which will play a central role in this thesis. We will follow \cite{BGP} here, they are also described in \cite {Gun} and \cite {Fri}.
As all constructions rely heavily on the exponential map, the objects will only be defined on convex subsets of $M$. Note that their values will be dependent on the choice of convex subset.

\begin{gn}
	For this section fix a geodesically convex open subset $U$ of $M$.
\end{gn}

The Riesz distributions on $U$ are roughly speaking a Lorentzian analogue to "powers of the distance functions". They are described in detail in \cite[sections 1.2 and 1.4]{BGP}. In this thesis, they  will play a role comparable to that of powers of $x$ in a Taylor series. Before we can define Riesz distributions on subsets of $M$, we will first define them on Lorentzian vector spaces (in particular each tangent space of $M$), where they are roughly speaking "powers of the norm".
\begin{df}
	\label{dfgamma}
	Let $V$ be a Lorentzian vector space with Lorentzian bilinear form $\eta$.
	For $x\in V$ define 
	\[\gamma^ V(x):=-\eta(x,x).\]
	Usually, there will be no ambiguity with respect to what $V$ is and we will drop it from the notation.
\end{df}
Note that $\gamma(x)$ strictly positive/negative if and only if $x$ is timelike/spacelike.
\begin{defprop}
	\label{dfRiesz}
	Let $V$ be a d-dimensional Lorentzian vector space.
	For $\alpha\in \C$ with $\Re(\alpha)\geq 0$, define Riesz distributions $R^{V}_\pm(\alpha)$ as the distribution on $V$ given by the function
	\[R^V_\pm(\alpha)(x)=\casedist{c_\alpha\gamma(x)^{\frac{\alpha-d}{2}}}{x\in J_\pm(0)}{0}{x\notin J_\pm(0)}\]
	with
	\[c_\alpha:=\frac{2^{1-\alpha}\pi^\frac{2-d}{2}}{\Gamma(\frac{a}{2})\Gamma(\frac{\alpha-d+2}{2})}.\]
	The map $\alpha\mapsto R^{V}_\pm(\alpha)$ is holomorphic as a map into $\D'(V)$ and extends uniquely to a holomorphic map on all of $\C$. For arbitrary $\alpha\in\C$, define $R^{V}_\pm(\alpha)$ to be the value of this holomorphic extension. (The prefactor $c_\alpha$ is chosen such that $\square R^{V}_\pm(\alpha+2)=R^{V}_\pm(\alpha)$.) In case it is clear what $V$ is, we omit it from the notation.
\end{defprop}
\begin{rem}
	Another way of thinking about the Riesz distribution $R^\pm(2k)$ is as (advanced or retarded) fundamental solutions to $\square^k$ (c.f Proposition \ref{RFS}). In this sense, $R^\pm(2)$ can be interpreted as Lorentzian versions of the Newtonian potential. 
\end{rem}
\begin{proof}
	See \cite[Lemma 1.2.2]{BGP} and note that everything remains valid under isometries, so we may replace Minkowski space with an arbitrary Lorentzian vector space.
\end{proof}
We transport this from tangent spaces to our convex subset $U$ using the exponential map:
\begin{df}
	\label{dfGamma}
	For $x,y\in U$ and $\alpha\in \C$,
	define the "squared geodesic distance"
	\[\Gamma^U_x(y):=\gamma\left((exp_x^U)^{-1}(y)\right)\]
	and define the Riesz distributions
	\[R^U_\pm(\alpha,x):=(exp^U_x)^{-1*}(R^{T_xM}_\pm(\alpha)|_{Dom(exp^U_x)})\]
	on $U$. We may drop the choice of neighborhood $U$ from the notation if it is clear from the context (in this section it will always be the one we fixed).
\end{df}
Note that $\Gamma^U_x$ is positive precisely on $J^U(x)$ and $R^U_\pm(\alpha,x)$ is supported in $J^U_\pm(x)$ by Proposition \ref{exJ}.
\begin{rem}
	When identifying a Lorentzian vector space $V$ with its tangent space at $0$ in the canonical way, the exponential map at $0$ becomes the identity and thus $\Gamma_0^V=\gamma^V$ and $R^V_\pm(\alpha,0)=R^V_\pm(\alpha)$.
\end{rem}
The Riesz distributions have the following properties (see \cite[Proposition 1.4.2]{BGP}\footnote{for \ref{RMSq}, use that both sides are continuous to also obtain the statement for $\alpha=0$})
\begin{thm}For $x\in U$ and $\alpha\in \C$, we have the following
	\begin{enumerate}
		\item $\alpha\mapsto R_\pm(\alpha,x)$ is holomorphic as a distribution-valued map.
		\item $\supp(R_\pm(\alpha,x))\subseteq J^{U}_\pm(x)$
		\item For $\Re(\alpha)\geq d$, $R_\pm(\alpha,x)$ is given by \[R_\pm(\alpha,x)(y)=\casedist{c_\alpha\Gamma_x(y)^{\frac{\alpha-d}{2}}}{y\in J^{U}_\pm(x)}{0}{y\notin J^{U}_\pm(x)}.\]
		In particular, for any $k\in \N$, $R_\pm(\alpha,x)$ is $C^k$ for any $\Re(\alpha)\geq d+2k$.
		\item\label{RMSq}$2\alpha\square R_\pm(\alpha+2,x)=\left(\square\Gamma_x-2d+2\alpha\right)R_\pm(\alpha,x)$
		\item$2\alpha\grad{R_\pm(\alpha+2,x)}=(\grad\Gamma_x)R_\pm(\alpha,x)$
		\item$R_\pm(0,x)=\delta_x$
	\end{enumerate}
\end{thm}
We now turn to the second class of objects we want to define in this section: the Hadamard coefficients.
The motivation for defining the Hadamard coefficients is that we want to have the formal equality
\[P\insum{k}V^{k}(x)R_\pm(2k+2,x)=\delta_x\]
for some smooth sections $V^{k}(x)$. This requires us to investigate the application of $P$ to a product involving Riesz distributions. To keep the formulas shorter, we introduce the following differential operator that we'll use to describe the remainder term arising in the following.
\begin{df}
	\label{dfrho}
	Define for $V\in \Gamma(E\otimes E_x)$:
	\[\rho_x^U V:= \nabla_{\grad\Gamma_x^U}V-\left(\frac{1}{2}\square\Gamma^U_x-d\right)V.\]

\end{df}
The reader may immediately forget the definition of $\rho_x^U$ and only remember the following property:
\begin{prop}
	\label{Prho}
	We have for $\alpha\in C$, $x\in U$ and $V\in \Gamma(E|_U\otimes E^*_x)$
	\[\alpha P(VR_\pm(\alpha+2,x))=\alpha(PV)R_\pm(\alpha+2,x)-((\rho_x^U-\alpha)V)R_\pm(\alpha,x)\]
\end{prop}
\begin{proof}
	We calculate for $\Re(\alpha)>d+4$ (so $R_\pm(\alpha,x)$ is $C^2$ and we don't have to worry about distributions)
	\begin{align*}
	\alpha P (VR_\pm(\alpha+2,x))&=\alpha\left((P V)R_\pm(\alpha+2,x)-2\nabla_{\grad R_\pm(\alpha+2,x)} V+V\square R_\pm(\alpha+2,x)\right)\\
	&=\alpha(P V)R_\pm(\alpha+2,x)-(\nabla_{\grad\Gamma} V) R_\pm(\alpha,x)+\frac{1}{2} V\left(\square\Gamma_x-2d+2\alpha\right)R_\pm(\alpha,x)\\
	&=\alpha(P V)R_\pm(\alpha+2,x)+\left((\frac{1}{2}\square\Gamma-d+\alpha)V-\nabla_{\grad\Gamma} V\right)R_\pm(\alpha,x)\\
	&=\alpha(P V)R_\pm(\alpha+2,x)-((\rho_x^U-\alpha)V)R_\pm(\alpha,x).
	\end{align*}
	As both sides are holomorphic, the equation holds for arbitrary $\alpha$.
\end{proof}
	We are now ready to define the Hadamard coefficients (see \cite[Definition 2.2.1 and Proposition 2.3.1]{BGP}):
	\begin{defprop}
		\label{Vdef}
		For every $x\in U$ there is a unique family of sections $V^{k,U}_x\in \Gamma(E\otimes E^*_x|_U)$ (indexed by $k\in \N$) that satisfies the transport equations 
		\[(\rho_x^U-2k) V^{k,U}_x=2kP V^{k-1,U}_x\]
		(for $k=0$, the right hand side is set to $0$) subject to the initial condition
		\[V^{0,U}_x(x)=1.\]
		These sections are called the Hadamard coefficients (for $P$). The map $(y,x)\mapsto V^{k,U}_x(y)$ is a smooth section in $E\boxtimes E^*|_{U\times U}$. As before, we will drop the superscript where it is not needed and write $V^k_x$ for $V^{k,U}_x$.
	\end{defprop}
These are defined so that we have the equality:	\[P\sum\limits_{k=0}^{N}V^k_xR_\pm(2k+2,x)=\delta_x+(PV^{N}_x)R_\pm(2N+2,x).\]
Morally we want to replace $N$ by $\infty$. In that case, we would (formally) obtain a solution to the equation
\[PF=\delta,\]
which (together with support conditions) determines the Schwartz kernels of the Green's operators.

However, in general the remainder $(PV^{N}_x)R_\pm(2N+2,x)$ may not vanish for $N\rightarrow \infty$. All we get is that the remainder becomes arbitrarily differentiable for $N$ large enough, i.e. we get an asymptotic expansion in differentiability order.
This is the motivation for the definition of the Hadamard coefficients and one of the key results that will be used in this thesis: We have an asymptotic expansion
\[G^\pm\delta_x|_U\sim_x\insum{k}V^{k,U}_xR^{U}_\pm(2k+2,x)\]
in case $U$ satisfies some further assumptions.
This expansion is shown in \cite [Proposition 2.5.1]{BGP} on some suitable small neighborhoods and we will show a version for slightly different neighborhoods later on.
\begin{rem}
	Note that both $G^\pm \delta_x$ and $R^U_\pm(2k+2,x)$ are supported inside $J_\pm(x)$ hence the same is true for the remainder term in the expansion. Thus all derivatives that the remainder term has must be vanishing at the boundary of $J_\pm(x)$, in particular at $x$. This means that differentiability of the remainder leads to decay near $x$ of the corresponding order (see \cite[Theorem 2.5.2]{BGP} for a detailed statement).
	
	There is another propagator associated to $P$, known as the Feynman propagator, that also has a similar asymptotic expansion involving the Hadamard coeficients. However, in the Feynman case nothing needs to vanish anywhere, so differentiability will not transform into decay estimates. That is the reason why we need the causal propagator in the approach presented here.
\end{rem}

In general, as the transport equations depend on $\Gamma^U$, the value $V^{k,U}_x(y)$ will depend on the choice of neighborhood $U$. This is due to the fact that, in different convex neighborhoods, the unique geodesic joining $x$ and $y$ might be different. If, however, one of the two sets is contained within the other, their coefficients coincide:
\begin{prop}
	Let $U$ and $U'$ be open convex subsets of $M$ with $U'\subseteq U$. Then for any $x\in U'$,
	\[V^{k,U}_x|_{U'}=V^{k,U'}_x\]
	and
	\[R^U_\pm(\alpha,x)|_{U'}=R^{U'}_\pm(\alpha,x).\]
\end{prop}
\begin{proof}
	Any geodesic in $U'$ is also in $U$, so for $x$ and $y$ in $U'$ the unique geodesics joining them in $U$ and $U'$ are the same. This means that $exp_x^{-1}(y)$ and hence $\Gamma_x(y)$ coincide for $U$ and $U'$.  Thus their transport equations are the same, so $V^{k,U}_x|_{U'}$ also solves the transport equations for $U'$.
	
	Moreover, as $\Gamma$ coincides for both sets, the same is true for $R_\pm(\alpha,x)$ if $\Re(\alpha)$ is large enough. By uniqueness of holomorphic continuation, the Riesz distributions also coincide for arbitrary $\alpha$.
\end{proof}
\begin{cor}
	For any two convex open neighborhoods $U_1$ and $U_2$ of $p\in M$, there is a neighborhood $U_0\subseteq U_1\cap U_2$ of $p$ such that for any $k\in \N$
	\[V^{k,U_1}_x|_{U_0}=V^{k,U_2}_x|_{U_0}.\]
\end{cor}
\begin{proof}
	Take $U_0$ to be any open convex neighborhood contained in $U_1\cap U_2$. Then both sides of the equation are equal to $V^{k,U_0}_x$. 
\end{proof}
Thus the germ of the Hadamard coefficients at the diagonal (i.e. at $y=x$) is independent of the choice of neighborhood, allowing us to define:
\begin{df}
	For any local operator $L$, write $LV^k_x(x)$ for the value of $LV^{k,U}_x(x)$ for any (and hence all) convex open neighborhood(s) $U$ of $x$.
\end{df}
	
\section{Evolution operators}
\label{evopchap}
For this section, we will assume that $M$ is foliated into smooth spacelike Cauchy hypersurfaces $(\Sigma_t)_{t\in \R}$.
 By Theorem \ref{Ghyper}, such a foliation always exists for globally hyperbolic manifolds. Let $n$ denote the future directed unit normal vector field to the hypersurfaces $\Sigma_t$.

 Given initial data on any $\Sigma_t$, there is a unique solution to $Pu=0$ with these initial data (see \cite[Theorem 3.2.11]{BGP}):
\begin{thm}[Well-posedness of the Cauchy problem]
	For any $t\in \R$ and $f,g\in \Gamma_c(E|_{\Sigma_t})$, there is a unique $u\in \Gamma(E)$ such that
	\[u|_{\Sigma_t}=f,\]
	\[\nabla_n u|_{\Sigma_t}=g\]
	and
	\[Pu=0.\]
	This $u$ is supported in $J(\supp(f)\cup\supp(g))$.
\end{thm}

This allows us to define the following family of evolution operators:
\begin{df}
	\label{dfQ}
	For $s,t\in \R$, define 
	\[Q(t,s)\colon \Gamma_c(E|_{\Sigma_s})\rightarrow \Gamma_c(E|_{\Sigma_t})\]
	by setting for $\psi\in \Gamma_c(E|_{\Sigma_s})$
	\[Q(t,s)\psi(x):=u(x,t)\]
	for the unique $u\in \Gamma(E)$ with
	\[u|_{\Sigma_s}=0\]
	\[\nabla_n u|_{\Sigma_s}=\psi\]
	and
	\[Pu=0.\]
\end{df}
\begin{rem}
	The full Cauchy evolution operator would be the operator 
	\begin{align*}
		\Gamma_c(E|\sigma_s)\oplus\Gamma_c(E|\sigma_s)&\rightarrow \Gamma_c(E|\sigma_t)\oplus \Gamma_c(E|\sigma_t)\\
		(u|_{\Sigma_s},\nabla_n u|_{\Sigma_s})&\mapsto(u|_{\Sigma_t},\nabla_n u|_{\Sigma_t})\text{ for }Pu=0.
	\end{align*}
The operator we consider here is the top right entry of the full operator viewed as a $2\times2$-matrix.
\end{rem}
We want to relate this system of evolution operators to the Green's operators of $P$. For that we need a slightly changed version of \cite[Lemma 3.2.2]{BGP}. The original version in \cite{BGP} phrased in terms of fundamental solutions rather than Green's operators and only for suitable small neighborhoods, as \cite{BGP} has only shown existence of fundamental solutions (which are basically kernels of Green's operators) on these neighborhoods at that point. However, as we have already taken the global existence of Green's operators for granted, we may state the lemma for all of $M$.
\begin{lemma}
	\label{322}
	For $\psi\in \Gamma(E|_{\Sigma_s})$, let $u$ be a solution to
	\[u|_{\Sigma_s}=0\]
	\[\nabla_n u|_{\Sigma_s}=\psi\]
	and
	\[Pu=0.\]
	Then we have for every $\phi\in \Gamma_c(E^*)$:
	\[\intM \phi(x) u(x)dx=-\int\limits_{\Sigma_s}G_{P^*}\phi(x) \psi(x)dx.\]
\end{lemma}
\begin{proof}
	In the proof of \cite[Lemma3.2.2]{BGP}, set $S:=\Sigma_s$, $u_1:=\psi$ and $u_0:=0$, replace $\Omega$ with $M$ and redefine $\Psi$ as $G^+_{P^*}(\phi)$ and $\Psi'$ as $G^-_{P^*}(\phi)$. The proof then shows the claimed result.
\end{proof}
With this we can show the following:
\begin{prop}
	\label{GtoQ}
	For any $\psi\in \Gamma(E|_{\Sigma_s})$, we have
	\[Q(t,s)\psi=G(\iota_{s*}\psi)|_{\Sigma_t},\]
	where $\iota_s$ denotes the inclusion of $\Sigma_s$ into $M$.
\end{prop}
\begin{proof}
	Let $u$ be as above. Then by definition
	\[Q(t,s)\psi=u|_{\Sigma_t}.\]
	Recall that the Green's operators of $P^*$ are adjoints to that of $P$ with flipped $\pm$ (see Proposition \ref{Gdual}), so $G^*=-G_{P^*}$. Thus, using Lemma \ref{322}, we get for any $\phi\in \Gamma_c(E^*)$:
	\[u[\phi]=-\int\limits_{\Sigma_s}G_{P^*}\phi(x) \psi(x)dx=\iota_{s*}\psi[-G_{P^*}\phi]=G(\iota_{s*}\psi)[\phi].\]
	Thus 
	\[u=G(\iota_{s*}\psi)\]
	and hence
	\[Q(t,s)\psi=u|_{\Sigma_t}=G(\iota_{s*}\psi)|_{\Sigma_t}.\qedhere\]
\end{proof}
Morally, this means that
\[\K(Q(t,s))(y,x)=\K(G)(y,x)\]
for any $x\in \Sigma_s$ and $y\in \Sigma_t$.

\section{Wavefront calculus}
\label{WFcalc}
\subsection{Wavefront calculus on $\R^n$}
\label{wfrn}
The wavefront set of a distribution $\eta$ is a refinement of its singular support. It captures not only at which points a distribution is not smooth, but also in which directions. It will be a subset of the cotangent bundle of the manifold without the zero section. The notion of wavefront sets goes back to Lars Hörmander (\cite {HorII}).
\begin{df}
	For a manifold $X$, let $\Tdot(X)$ denote the set of all non-zero vectors in the cotangent bundle $T^*(X)$. 
	In this part, we will assume that $X$ and $Y$ are open subset of $\R^n$ and $\R^m$, for $m,n\in \N$. In this case, we identify $\Tdot(X)$ with $X\times(\R^n\O)$. Compactly supported functions on $X$ are sometimes implicitly extended by $0$ to functions on all of $\R^n$. All definitions made in terms of $X$ and $Y$ should be understood as definitions for arbitrary manifolds, if this is possible.
\end{df}

We define the wavefront of a distribution by characterizing its complement:
\begin{df}
	\label{dfwfrn}
	Let $\eta$ be a distribution on $\R^n$.
	A pair $(x,\xi)\in \Tdot(X)$ is \textbf{not} in the wavefront set $\WF(\eta)\subseteq\Tdot(X)$, if and only if the following is satisfied:
	
	There is a conic neighborhood $O$ of $\xi$ and $\chi\in C_c^\infty(X)$ such that
	\begin{itemize}
		\item $x\in \supp(\chi)$
		\item the Fourier transform of $\chi\eta$ decays rapidly on $O$, i.e. we have for all $k\in \N$:
		\[\sup\limits_{\xi'\in O}(1+|\xi'|)^k\|\F(\chi\eta)(\xi')\|<\infty.\]
	\end{itemize}
Here conic means that $\lambda\xi'\in O$ for $\xi'\in O$ and $\lambda\in (0,\infty)$.
\end{df}
As multiplication with $\xi$ in Fourier space corresponds to differentiation in the corresponding direction, the wavefront set can be interpreted as the set of directions at each point in which the function is not smooth.

The space of all distributions with a given wavefront set can then be topologized by taking something like the best constants in the above condition as additional seminorms:
\begin{df}
	\label{wfnorm}
	For a distribution $\eta$ on $X$, $\chi\in C_c^\infty(X)$, $V\subseteq\R^n\O$ closed and conic and $k\in \N$, define
	\[\|\eta\|_{k,V,\chi}:=\sup\limits_{\xi\in V}(|\xi|+1)^k|\F(\chi\eta)(\xi)|.\]
	For $\Lambda\subseteq\Tdot(\R^n)$ a closed conic subset (in the sense that $(x,\mu\xi)\in \Lambda$ for $(x,\xi)\in \Lambda$ and $\mu\in (0,\infty)$), let $\D'_\Lambda(X)$ be the space of all distributions on $X$ with wavefront contained in $\Lambda$. This is topologized by the seminorms for the (strong) topology on $\D'(X)$ and additional seminorms
	\[\|\cdot\|_{k,V,\chi}\]
	for any $k,V,\chi$ as above such that
	\[\supp(\chi)\times V\cap\Lambda=\emptyset.\]
\end{df}
The use of the strong topology on $\D'$ is required to make some operations continuous on these spaces. Moreover, with this topology, $D'_\Lambda$ is complete (see \cite[Corollary 25]{DaBr}). However, for convergence of sequences it is equivalent to the weak*-topology. As sequential continuity would be sufficient for the purpose of this paper, we could also work with the weak*-topology instead.

As for the usual space of distributions, $C_c^\infty(X)$ is dense in $\D'_\Lambda(X)$ for any $\Lambda$ (see \cite[8.23]{HorI}).
Obviously, for $\Theta\subseteq \Lambda$, $\D'_\Theta(X)$ is continuously embedded in $\D'_\Lambda(X)$, as any seminorm of the latter is also a seminorm of the former.

The wavefront set is a refinement of the singular support, in the sense that $\singsupp (\eta)$ is the projection of $\WF(\eta)$ onto the first component.
Moreover, $\D'_\emptyset(X)$ and $C^\infty(X)$ coincide as topological vector spaces (see \cite[Lemma 7.2]{BDH}).

To discuss continuity properties of certain bilinear operations on wavefront spaces, we shall need the concept of hypocontinuity:
\begin{df}
	A bilinear map $b\colon V_1\times V_2\rightarrow V_3$, for $V_i$ topological vector spaces, is called hypocontinuous, if for every neighborhood $W$ of 0 in $V_3$ and any bounded sets $A_1\subset V_1$ and $A_2\subset V_2$ there are neighborhoods of zero $U_1\subseteq V_1$  and $U_2\subseteq V_2$ such that
	$b^{-1}(W)$ contains $U_1\times A_2$ and $A_1\times U_2$.
\end{df}
If the topologies are given in terms of seminorms, this can be rephrased in terms of those as follows:
For every bounded subset $A\subset V_1$ and every seminorm $\rho$ of $V_3$, there are seminorms $(\rho_i)_{i\leq N}$ of $V_2$ and $C_A\in\R$ such that for all $v\in A$, $w\in V_2$, we have
\[\rho(b(v,w))\leq C_A\sum\limits_{i=0}^N\rho_i(w)\]
and the same holds with $V_1$ and $V_2 $ interchanged.

As any bounded set is contained in a suitably scaled version of every neighborhood of zero, this is weaker than continuity. As each bounded set $A_i$ is an open subset of itself and translations of bounded sets are bounded, a hypocontinuous map is continuous when restricted to $A_1\times V_2$ or $V_1\times A_2$. In particular, as points are bounded, it is separately continuous in each argument when keeping the other argument fixed. Moreover, since all convergent sequences are bounded, hypocontinuous maps are sequentially continuous. As differential quotients can be defined via limits of sequences, hypocontinuous maps preserve differentiability and holomorphicity. As images of bounded sets under continuous linear maps are bounded and preimages of open sets are open, compositions of hypocontinuous bilinear maps with continuous linear maps (either postcomposition or precomposition in either argument) are again hypocontinuous. The bilinear operations we will encounter in wavefront calculus are generally not continuous, but hypocontinuous.

One of the advantages of wavefront calculus is that it allows us to perform operations that are not well-defined for general distributions, like multiplications or pull-backs, on distributions with the right wavefront.
The principle is always (mostly) the same: If the image of some closed conic source wavefront set(s) is contained in a target wavefront set, then the operation is well-defined and continuous between the associated wavefront spaces. In particular, the operation is well-defined if the image of the wavefront set(s) does not intersect the zero section. We now define the images of wavefront sets under the operations that will be needed later:
\begin{df}
	\label{WFops}
	For a smooth map $f\in C^\infty (X,Y)$ and closed conic sets $\Lambda,\Lambda'\subseteq \Tdot(X)$ and $\Theta\subseteq \Tdot(Y)$, define
	\begin{itemize}
		\item $f^*(\Theta):=\{(df_x)^*(\xi)\mid x\in X, \xi\in \Theta\cap T_{f(x)}(Y)\}$
		\item $f_*(\Lambda):=\{\xi\in \Tdot(Y)\mid\exists x\in X\colon \xi\in T_{f(x)}(Y)\wedge (df_x)^*\xi\in \Lambda \cup 0\}$
		\item $\Lambda\bar +\Lambda':=\{\xi+\xi'\mid\xi\in\Lambda,\xi'\in\Lambda'\} \cup\Lambda\cup \Lambda'$
		\item $\Lambda\bar \times \Theta:=((\Lambda\cup 0)\times(\Theta\cup 0))\backslash 0$
	\end{itemize}
(Here $(df_x)^*$ denotes the adjoint of the total differential of $f$ at $x$).
\end{df}
In the case of pushforwards, there is an additional restriction that is required even in the case of smooth functions:
\[f_*(\phi)[\psi]:=\phi[\psi\circ f]\]
is not well defined in general, as $\psi\circ f$ may not be compactly supported. We thus have to restrict to those distributions with a suitable support.
\begin{df}
	For $C\subseteq X$ and $\Lambda\subseteq \Tdot(X)$, let $\D'_\Lambda(X)|_C$ denote the subspace of $\D'_\Lambda(X)$ consisting of those elements with support in $C$. Let $\Lambda|_C$ denote the set of all elements of $\Lambda$ whose basepoint lies in $C$.
\end{df}
\begin{thm}[wavefront calculus on $\R^n$]
	\label{WcalcRn}
	Let $X$ and $Y$ be open subsets of $\R^n$ and $\R^m$, with $n,m\in \N$. Let $\Lambda, \Lambda'\subseteq \Tdot X$ and $\Theta\subseteq \Tdot(Y)$ be closed conic subsets. Let $f\colon X\rightarrow Y$ be a smooth map.
	\begin{itemize}
		\item If $f^*(\Theta)$ does not intersect the zero section, $f^*$ extends to a continuous map from $\D'_\Theta(Y)$ to $\D'_{f^*(\Theta)}(X)$.
		\item If $C\subseteq X$ is a closed set such that $f|_{C}$ is proper, $f_*$ defines a continuous map from $\D'_{\Lambda}(X)|_C$ to $\D'_{f_*(\Lambda)}(Y)$.
		\item If $\Lambda\bar +\Lambda'$ does not intersect the zero section, multiplication of functions extends to a hypocontinuous bilinear map
		\[\D'_\Lambda(X)\times\D'_{\Lambda'}(X)\rightarrow\D'_{\Lambda\bar +\Lambda'}(X).\]
		\item The tensor product defines a hypocontinuous bilinear map
		\[\D'_\Lambda(X)\times\D'_\Theta(Y)\rightarrow \D'_{\Lambda\bar \times\Theta}(X\times Y).\]
	\end{itemize}
\end{thm}
This is shown in \cite[Proposition 5.1, Theorem 6.3, Theorem 6.1 and Theorem 4.6 ]{BDH}, with one difference: For their push-forward theorem, \cite{BDH} assume that their wavefront set $\Lambda$ only contains covectors with basepoint in $C$. This implies the version above, if we can show the following Lemma:
\begin{lemma}
	For $C\subseteq X$ closed, we have \[\D'_\Lambda(X)|_C=\D'_{\Lambda|_C}(X)|_C\] as topological vector spaces.
\end{lemma}
\begin{proof}
	Note that $\Lambda|_C=\Lambda\cap(C\times (\R^n\O))$ is the intersection of two closed conic subsets of $\Tdot(\R^n)$ and hence closed and conic again.
	As $\Lambda|_C\subseteq \Lambda$, we automatically have a bounded inclusion \[\D'_{\Lambda|_C}(X)|_C\subseteq \D'_\Lambda(X)|_C.\]
	Fix $\eta\in \D'_\Lambda(X)|_C$. As every element of $\WF(\eta)$ has basepoint in $$\singsupp(\eta)\subseteq\supp(\eta)\subseteq C,$$
	 $\eta$ has wavefront in $\Lambda|_C$, so the two spaces agree as sets. It remains to show that they have the same topology. Consider some seminorm $\|\cdot\|_{k,V,\chi}$ of $\D'_{\Lambda|_C}(X)|_C$. Then $\supp(\chi)\times V$ does not intersect $\Lambda|_C$ and hence $(\supp(\chi)\cap C)\times V$ does not intersect $\Lambda$. Let $V_1$ denote the set of unit vectors in $V$. As $\Lambda$ is closed, every
	\[\xi\in (\supp(\chi)\cap C)\times V_1\]
	has an open neighborhood $U(\xi)$ that is disjoint from $\Lambda$. As product neighborhoods form a basis, we may assume that this is of the form
	\[U(\xi)=U_1(\xi)\times U_2(\xi).\]
	As $V_1$ is compact, for each $x\in (\supp(\chi)\cap C)$ there are finitely many $(v_i)_{i\in I}\in V_1$ such that the corresponding neighborhoods $U_2(x,v_i)$ cover $V_1$. Set $U'(x):=\bigcap\limits_{i\in I}U_1(x,v_i)$. Then $U'(x)\times V_1$ does not intersect $\Lambda$, since any $U'(x)\times U_2(x,v_i)$ doesn't. As $(\supp(\chi)\cap C)$ is compact, we may find finitely many $(x_j)_{j\in J}\in (\supp(\chi)\cap C)$ such that 
	$$O:=\bigcup\limits_{i\in I}U'(x_i)$$
	 contains $(\supp(\chi)\cap C)$. By construction $O\times V_1$ does not intersect $\Lambda$. As $\Lambda$ is conic, the same holds for $O\times V$. Choose $\psi\in C_c^\infty(O)$ that is $1$ on $\supp(\chi)\cap C$. Then for arbitrary $\eta\in\D'_\Lambda(X)|_C$, we have $\psi\chi\eta=\chi\eta$ and hence
	\[\|\eta\|_{k,V,\chi}=\|\eta\|_{k,V,\psi\chi}.\]
	Moreover, $\supp(\psi\chi)\times V\subseteq O\times V$ does not intersect $\Lambda$, so the right hand side is a seminorm of $\D'_\Lambda(X)|_C$. Thus the two spaces have the same seminorms and hence the same topology.
	
\end{proof}
\subsection{Wavefront calculus on manifolds}
\label{wfM}
We now turn from functions on subsets of $\R^n$ to sections in vector bundles.  Defining the wavefront set and extending the previous results for sections can be done in a fairly straightforward way by considering their components in some coordinate chart:
\begin{df}
	By a "bundle coordinate" on a vector bundle $F$ over a manifold $X$, we shall mean a smooth, fibrewise linear map $F\rightarrow \C$ (these can be canonically identified with sections of $F^*$ or bundle homomorphisms $F\rightarrow X\times \C$ over the identity).
	Let $A$ be a bundle coordinate or a bundle homomorphism over the identity. For a distribution $\eta \in \D'(F)$ and any scalar test function $\phi$, define
	\[(A\eta)[\phi]:= \eta[A^*\circ\phi],\]
	where $A^*$ is the fibrewise adjoint of $A$.
\end{df}
We have for any $g\in \Gamma(F)$, test section $\phi$ and $\eta\in \D'(M)$:
\[A(g\eta)[\phi]=g\eta[A^*\circ\phi]=\eta[x\mapsto ((A^*\phi(x))(g(x)))]=\eta[x\mapsto (\phi(x)(A(g(x))))]=(A\circ g)\eta[\phi].\]
In particular (set $\eta=1$), we have $Ag=A\circ g$.

With this, we define the wavefront set of a distribution in a vector bundle:
\begin{df}
	\label{dfwfmfd}
	Let $E$ be a vector bundle over a manifold $X$.
	 For $\eta\in \D'(E)$ define the wavefront set of $\eta$ as the union of all sets of the form
	\[\psi^*(\WF((\psi^{-1})^*( A\eta|_{\Dom(\psi)})))\]
	for coordinate charts $\psi$ and bundle coordinates $A$ defined on $\Dom(\psi)$.
	The space $\D'_\Lambda(E)$ of all distributions with wavefront in a given closed conic set $\Lambda$ is topologized by the (strong) seminorms of $\D'(E)$ and all seminorms
	\[\|\eta\|_{k,V,\chi,\psi,A}:=\|(\psi^{-1})^*(A\eta|_{\Dom(\psi)})\|_{k,V,\chi\circ\psi^{-1}}=\sup\limits_{\xi\in V}(|\xi|+1)^k\big|\F\big((\psi^{-1})^*(\chi A\eta)\big)(\xi)\big|,\]
	where $\psi$ and $A$ are as above, $\chi\in C_c^\infty(\Dom(\psi))$ is some cut-off, $k\in \N$ and $V$ is a closed cone in $\R^n$ such that $\supp(\chi\circ\psi^{-1})\times V$ does not intersect $(\psi^{-1})^*(\Lambda)$.
	Write $\D'_\Lambda(X)$ for $\D'_\Lambda(X\times \C)$.
\end{df}
Basically all the results on subsets of $\R^n$ generalize to this setting. See the appendix (\ref{append}) for the proof of the following theorem:
\begin{thm}[Wavefront calculus on manifolds]
	\label{WFMFD}
	The results obtained for scalar functions on $\R^n$ carry over to sections of vector bundles on manifolds. Assume that $X$ and $Y$ are manifolds. For $i\in \{1,2,3\}$, let $E_i$ be a vector bundle over $X$ and let $F$ be a vector bundle over $Y$. Let $\Lambda, \Lambda'\subseteq \Tdot(X)$ and $\Theta\subseteq \Tdot(Y)$ be closed conic subsets.
	\begin{enumerate}
		\item The seminorms defined for $\D'_\Lambda(E_1)$ are finite. In case $X$ is a subset of $\R^n$ and $E_1=X\times \C$, the new definition coincides with the earlier one.
		\item $\D'_\Lambda(E_1)$ is complete.
		\item $\Gamma_c(E_1)$ is dense in $\D'_\Lambda(E_1)$.
		\item $\D'_\emptyset(E_1)$ coincides with $\Gamma(E_1)$ as a topological vector space.
		\item Differential operators between sections of $E_1$ and $E_2$ map $\D'_\Lambda(E_1)$ to $\D'_\Lambda(E_2)$ continuously.
		\item Let $f\colon X\rightarrow  Y$  be a smooth map.
		\begin{itemize}
			\item If $f^*(\Theta)$ does not intersect the zero section, $f^*$ extends to a continuous map from $\D'_\Theta(F)$ to $\D'_{f^*(\Theta)}(f^*(F))$.
			\item If $C\subseteq X$ is a closed set such that $f|_C$ is proper, $f_*$ defines a continuous map from $\D'_\Lambda(f^*F)|_C$ to $\D'_{f_*(\Lambda)}(F)$.
			\item If $\Lambda\bar + \Lambda'$ does not intersect the zero section,  pointwise "multiplication" of sections in $E_1\otimes E_2^*$ and $E_2\otimes E_3$, defined at each point by
			\[(a\otimes b)\cdot(c\otimes d):=b(c)\cdot(a\otimes d)\]
			extends to a hypocontinuous bilinear map
			\[\D'_\Lambda(E_1\otimes E_2^*)\times \D'_{\Lambda'}(E_2\otimes E_3)\rightarrow \D'_{\Lambda\bar +\Lambda'}(E_1\otimes E_3).\]
			\item The tensor product defines a hypocontinuous bilinear map
			\[\D'_\Lambda(E_1)\times D'_\Theta(F)\rightarrow \D'_{\Lambda\bar\times \Theta}(E_1\boxtimes F).\]
			\item If $\Lambda\bar + \Lambda'$ does not intersect the zero section and $K\subseteq X$ is compact, evaluation of distributions extends to hypocontinuous bilinear maps
			\[\D'_\Lambda(E_1)\times \D'_{\Lambda'}(E_1^*)|_K\rightarrow \C\]
			and
			\[\D'_{\Lambda'}(E_1^*)|_K\times \D'_\Lambda(E_1)\rightarrow \C.\]
			These are symmetric, in the sense that
			\[\zeta[\eta]=\eta[\zeta].\]
			\end{itemize}
	\end{enumerate}
\end{thm}
\begin{rem}
	The multiplication map described here corresponds to pointwise composition when identifying tensor products with spaces of linear maps as in \ref{identify}. Using the other identifications of \ref{identify}, this contains the following as special cases (by setting all non-occuring vector bundles to $X\times \C$):
	\begin{itemize}
	\item Products of scalar functions
	\[\D'_\Lambda(X)\times \D'_{\Lambda'}(X)\rightarrow \D'_{\Lambda\bar +\Lambda'}(X).\]
	\item Scalar multiplication
	\[\D'_\Lambda(X)\times \D'_{\Lambda'}(E_3)\rightarrow \D'_{\Lambda\bar +\Lambda'}(E_3)\]
	\item Pointwise tensor product
	\[\D'_\Lambda(E_1)\times \D'_{\Lambda'}(E_3)\rightarrow \D'_{\Lambda\bar +\Lambda'}(E_1\otimes E_3).\]
	\item Pointwise dual pairing
	\[\D'_\Lambda(E_2^*)\times \D'_{\Lambda'}(E_2)\rightarrow \D'_{\Lambda\bar +\Lambda'}(X).\]
\end{itemize}
\end{rem}

The tensor product and pushforward interact in the way one would expect:
\begin{prop}
	If $f\colon X\rightarrow Y$ and $g\colon X'\rightarrow Y'$ are smooth maps between manifolds, we have for distributions $\eta$ and $\zeta$ in vector bundles over $X$ and $X'$ whose pushforward under $f$ respectively $g$ is well-defined:
	\[(f_*\eta)\otimes(g_*\zeta)=(f\times g)_*(\eta\otimes \zeta),\]
	where $f\times g(x,x'):=(f(x),g(x'))$.
\end{prop}
\begin{proof}
	If $f$ and $g$ are smooth and $\phi$ and $\psi$ are test sections, we have
	\[(f_*\eta)\otimes(g_*\zeta)[\phi\otimes\psi]=\eta[\phi\circ f]\cdot\zeta[\psi\circ g]=\eta\otimes\zeta[(\phi\otimes\psi)\circ(f\times g)]=(f\times g)_*(\eta\otimes \zeta)[\phi\otimes\psi].\]
	Thus 
	\[(f_*\eta)\otimes(g_*\zeta)=(f\times g)_*(\eta\otimes \zeta).\]
	As smooth functions are dense, this holds for arbitrary distributions.
\end{proof}

\newpage
\chapter{Asymptotic expansions related to Green's operators}
In this section, we will develop asymptotic expansions in differentiability orders for the kernels of powers of $G^\pm_P$ and that of $G^\pm_{P-z}$ for $z\in \C$. This is inspired by \cite{DaWr}, where similar results are obtained for Feynman inverses. However, the methods used here are different. While the only things necessary for the proof of our formula for the Hadamard coefficients are the standard Hadamard expansion and the formula for the $z$-dependent Hadamard coefficients shown in Proposition \ref{Vz}, the other results might be of independent interest.
\section{Powers of Green's operators}
We first derive an expansion for powers of the advanced/retarded Green's operators. Before we can start the main proof, we need to do some technical work to guarantee that the Green's operators do not decrease differentiability orders too much (locally). This will be necessary to control the remainder terms of our asymptotic expansion. The result is basically a consequence of mapping smooth sections to smooth sections continuously (a better result could be obtained by investigating $G$ as a Fourier integral operator).
\begin{prop}
	Let $A\subset M$ be past/future compact compact and $U\subset M$ be open and relatively compact. Let $r_U$ denote restriction to $U$. Then for every $n\in \N$ there is $m\in \N$ such that $r_U\circ G^\pm$ maps $\Gamma^m_A(E)$ to $\Gamma^n(E|_U)$ continuously.
\end{prop}
\begin{proof}
	We know (see \ref{extendedG}) that
	\[G^\pm\colon \Gamma_\pm(E)\rightarrow \Gamma_\pm(E)\]
	is continuous. Since the topologies of $\Gamma_{A}(M)$ and $\Gamma_{J_\pm(A)}(M)$ coincide with the subspace topology from $\Gamma_\pm$ (as the latter carries the limit topology), we know that the restriction
	\[G^\pm\colon \Gamma_{A}(E)\rightarrow \Gamma_{J_\pm(A)}(E)\]
	is also continuous. $C^k$- norms on ${\overline{U}}$ are seminorms on the target space. Thus by the seminorm-characterization of continuity, we know that for any $n\in \N$, there is $c\in\R$, $m\in \N$ and $K\subseteq M$  such that for all $\phi\in \Gamma_A(E)$
	\[\|G^\pm\phi\|_{C^n(\overline{U})}\leq c\|\phi\|_{C^m(K)}.\]
	The right hand side is also a seminorm of $\Gamma^m_A(E)$.
	Thus $r_U\circ G^\pm|_{\Gamma_{A}(E)}$ extends continuously to a map $\Gamma^m_A(E)\rightarrow \Gamma^n(E|_U)$. As $C^k$- convergence implies distributional convergence, this continuous extension must coincide with $r_U\circ G^\pm|_{\Gamma^m_{A}(E)}$, as the latter is distributionally continuous.
\end{proof}
As Hadamard coefficients and Green's kernels are sections in $E\boxtimes E^*$ and we want differentiability in both coordinates, we need to extend the above to the case where $G^\pm$ acts in the first component on sections of a box product.
\begin{prop}
	\label{Greg}
	Let $A\subset M$ be past/future compact and $U\subset M$ be open and relatively compact. Let $O$ be some manifold and $F$ some vector bundle over $O$. Then for each $n\in \N$, there is $m\in \N$ such that for any $f\in \Gamma^m_{A\times O}(E\boxtimes F)$, we have $G^\pm f|_{U\times O}\in \Gamma^n((E\boxtimes F)|_{U\times O})$, where $G^\pm f(y,x):=G^\pm (f(\cdot,x))(y)$.
\end{prop}
\begin{proof}
	As differentiability is local, it suffices to show this only on a chart domain in $O$ on which $F$ is trivial. As a section is smooth if and only if its composition with any bundle coordinate is smooth and the application of an operator in the first component commutes with taking bundle coordinates in the second component, it suffices to consider the case $O=\R^j$, $F=\R^j\times\R$. Likewise, we may also assume that $U$ is contained in a chart with local trivialization for $E$ and thus partial derivatives are well-defined.
	
	For $n\in \N$, let $m\geq n$ be large enough such that for any $k\leq n$, $G^\pm $ maps $\Gamma^{m-k}_A(E)$ to $\Gamma^{n-k}(E|_U)$ continuously. This is possible by applying the previous proposition $n$ times. We will inductively show the following in the domain of this chart: 
	
	\begin{claim}
		For any multiindices $\alpha$ and $\beta$ with $|\alpha|+|\beta|\leq k\leq n$, we have
		\[\partial_x^\beta \partial_y^\alpha G^\pm (f(\cdot,x))(y)=\partial_y^\alpha G^\pm (\partial_x^\beta f(\cdot,x))(y)\]
		with the derivatives on the left hand side existing. 
	\end{claim}
	 We proceed by induction over $k$, the case $k=0$ being immediate. Suppose we have shown the statement for some $k<n$ and assume $|\alpha|+|\beta|\leq k$. We have to show for every index $i$:
	\[\partial_{x_i}\partial_y^\alpha G^\pm (\partial_x^\beta f(\cdot,x))(y)=\partial_y^\alpha G^\pm (\partial_{x_i}\partial_x^\beta f(\cdot,x))(y)\]
	(including existence of the derivative). 
	We have for $h>0$ ($e_i$ denoting the $i$-th standard unit vector) and any compact set $K$:
	\begin{align*}
		&\|\frac{\partial_x^\beta f(\cdot,x+he_i)-\partial_x^\beta f(\cdot,x)}{h}-\partial_{x_i}\partial_x^\beta f(\cdot,x)\|_{C^{m-|\beta|-1}(K)}\\
		&=\|\frac{1}{h}\int\limits_0^h \partial_{x_i}\partial_x^\beta f(\cdot,x+te_i)-\partial_{x_i}\partial_x^\beta f(\cdot,x)dt\|_{C^{m-|\beta|-1}(K)}\\
		&\leq\sup\limits_{t<h}\|\partial_{x_i}\partial_x^\beta f(\cdot,x+te_i)-\partial_{x_i}\partial_x^\beta f(\cdot,x)\|_{C^{m-|\beta|-1}(K)}.
	\end{align*}
	All derivatives of $\partial_{x_i}\partial_x^\beta f(y,x)$ of order up to $m-|\beta|-1$ are continuous and hence uniformly continuous on the compact set
	\[K\times\{x+te_i\mid t\in[0,1]\}.\]
	Thus the end result of the above equation converges to $0$ for $h\rightarrow 0$. This means that the differential quotient 
	\[\frac{\partial_x^\beta f(\cdot,x+he_i)-\partial_x^\beta f(\cdot,x)}{h}\]
	converges to $\partial_{x_i}\partial_x^\beta f(\cdot,x)$ in $\Gamma^{m-|\beta|-1}_A(E)$. As $G^\pm $ maps $C^{m-|\beta|-1}$-sections to $C^{n-|\beta|-1}$ sections continuously and $\partial_y^\alpha$ maps those to $C^0$-sections (over $U\times O$) continuously, we have
	\begin{align*}
		&\partial_{x_i}\partial_y^\alpha G^\pm (\partial_x^\beta f(\cdot,x))(y)\\
		&=\lim\limits_{h\rightarrow 0}\partial_y^\alpha G^\pm \left(\frac{\partial_x^\beta f(\cdot,x+he_i)-\partial_x^\beta f(\cdot,x)}{h}\right)(y)\\
		&=\partial_y^\alpha G^\pm \left(\lim\limits_{h\rightarrow 0}\frac{\partial_x^\beta f(\cdot,x+he_i)-\partial_x^\beta f(\cdot,x)}{h}\right)(y)\\
		&=\partial_y^\alpha G^\pm (\partial_{x_i}\partial_x^\beta f(\cdot,x))(y),
	\end{align*}
	which finishes the proof of the claim.
	
	To conclude the proof of the proposition, note that an arbitrary partial derivative has the form
	\[\prod\limits_{j=0}^l\partial_x^{\beta_j}\partial_y^{\alpha_j}\]
	and iterated application of the claim yields that
	\[\prod\limits_{j=0}^l\partial_x^{\beta_j}\partial_y^{\alpha_j}G^\pm (f(\cdot,x))(y)=\prod\limits_{j=0}^l\partial_y^{\alpha_j}G^\pm (\prod\limits_{j=0}^l\partial_x^{\beta_j}f(\cdot,x))(y)\]
	exists and is continuous.
\end{proof}
We now introduce the precise form of the objects for which we want an asymptotic expansion. Morally, they can be thought of as a rigorous version of
\[\K(G^{\pm m})(\cdot,x).\]
\begin{df}
	\label{dfGx}
	For $x\in M$, $m\in \Z$ define distributions in $\D'(E\otimes E_x^*)$ via
	\[G^{\pm m}_x:=(G^\pm)^m\delta_x\]
	(where $\delta_x$ is the Dirac distribution at $x$, defined by $\delta_x[\psi\otimes v]:=\psi(x)(v)$ for $\psi\in \Gamma_c(E)$ and $v\in E^*_x$).
	In case $m=1$, we omit it:
	\[G^\pm_x:=G^{\pm 1}_x.\]
\end{df}
\begin{rem}
	For negative $m$, we obtain Schwartz kernels for powers of $P$ ("evaluated" at $x$) as both $G^+$ and $G^-$ are an inverse for $P$ on a suitable space.
\end{rem}
Note that here we used the notation introduced at the end of \ref{identify}, i.e. $G^\pm$ acts as the identity in the $E^*_x$-part. There are multiple other ways to think of $G^{\pm m}_x$. As we will show later, it is the pullback of $\K(G)$ with respect to the inclusion of $M$ at $x$ (see Lemma \ref{resGx}). Another way to see it is evaluation at $x$ composed with the adjoint of $(G^\pm)^m$:
\begin{prop}
	For any $\psi\in \Gamma_c(E^*)$, we have
	\[G^{\pm m}_x[\psi]=(G^\mp_{P^*})^m\psi(x)\]
\end{prop}
\begin{proof}
	Let $f_n$ be a sequence of test functions converging distributionally to $\delta_x$. Then we have by Proposition \ref{Gdual} and continuity of $G^\pm$:
	\begin{align*}
	G^{\pm m}_x(\psi)&=(G^\pm)^m\delta_x(\psi)\\
	&=\limn\int\limits_M\psi((G^\pm)^mf_n)\\
	&=\limn\int\limits_M(G^\mp_{P^*})^m\psi(f_n)\\
	&=\delta_x((G^\mp_{P^*})^m\psi)\\
	&=(G^\mp_{P^*})^m\psi(x). \qedhere
	\end{align*}
\end{proof}
To obtain the asymptotic expansion we want, we first need to restrict to a suitable subset of $M$. Such subsets will be reffered to as GE sets (it is left to the reader to decide whether this stands for "Green's expansion" or for "good enough").
\begin{df}
	\label{dfGE}
	A subset $U$ of $M$ shall be called GE, if it is open, geodesically convex, globally hyperbolic, and causally compatible. 
\end{df}
\begin{rem}
	Openness is required to restrict distributions to $U$. Geodesic convexity ensures the existence of Riesz distributions and of Hadamard coefficients. Global hyperbolicity ensures that $U$ has unique Green's operators, while causal compatibility ensures that those agree with the restriction of the Green's operators on $M$.
\end{rem}
	We shall make use of the following fact:
	\begin{prop}
		\label{GEN}
		Every $p\in M$ has a basis of GE neighborhoods.
	\end{prop}
	\begin{proof}
		This is basically the statement of \cite[Corollary 2 and Remark 14]{Min}, which asserts the existence of a neighborhood basis of geodesically convex, globally hyperbolic, causally convex sets. Causally convex means that for any neighborhood $U$ from this basis, its intersection with every causal curve is connected. We only need to show that this implies causal compatibility, which is not too difficult: If $y\in J_\pm^M(x)$ for $x\in U$, there must be a future/past directed causal curve $\gamma$ in $M$ that connects $x$ and $y$. As the intersection of $\gamma$ with $U$ is connected and contains $x$ and $y$, it must also contain the segment between $x$ and $y$, so the two can be joined by a future/past directed causal curve in $U$. Thus $x\in J_\pm^U(y)$ and hence $U$ is causally compatible.
	\end{proof}
We can now prove the main theorem of this section
\begin{thm}
	\label{PowExp}
	Let $U\subseteq M$ be a GE set and let $m\in \Z$. Then we have the asymptotic expansion
	\[G^{\pm m}_x|_U\sim_x\insum{k} \binom{m+k-1}{k}V^{k,U}_xR^U_\pm(2k+2m,x).\]
	For $m<0$, the right hand side is a finite sum, as all summands with $k+m>0$ vanish, and we have equality.
\end{thm}
\begin{rem}
	Note that the special case $m=1$ reproduces the expansion of \cite[Proposition 2.5.1]{BGP}, though on slightly different domains (this means that even though they are morally the same, it is not easy to reduce either one to the other). Unlike in \cite{BGP}, there is no assumption on the volume of $U$ and it may be chosen independently of $P$.
\end{rem}
\begin{rem}
	In case $m<0$, this is actually a formula for the kernel of  powers of $P$. All even Riesz distributions of order $0$ or less are the same in the advanced and retarded case and supported only at the diagonal, so this expansion is not as surprising as it might seem on first glance. From the case $m=-1$, we obtain that 
	\[\K(P)(\cdot,x)=V^0_xR(-2,x)-V^1_x\delta_x.\]
	In particular, $P$ is uniquely determined by the metric $g$ (which determines the Riesz-distributions) and the first two Hadamard coefficients $V^0$ and $V^1$.
\end{rem}
\begin{proof}
	We proceed by two-way induction on $m$, showing that if the statement holds for $m\in \Z$, it also holds for $m+1$ if $m\geq 0$ and for $m-1$, if $m\leq 0$. The case $m=0$ follows from $R_\pm(0,x)=\delta_x$. Fix $x\in U$.
	
	We make some preliminary calculations. For $k+m\neq 0$ we can use the characterization of $\rho_x^U$ (\ref{Prho}) and the transport equations (see \ref{Vdef}) to calculate
	\begin{align*}
		&\binom{k+m}{k}P(V^{k,U}_xR^U_\pm(2k+2m+2,x))\\
		&=\binom{k+m}{k}\Big(\frac{-1}{2k+2m}(\rho_x^U-2k-2m)V^{k,U}_xR^U_\pm(2k+2m,x)+(PV^{k,U}_x)R^U_\pm(2k+2m+2,x)\Big)\\
		&=\binom{k+m}{k}\Big(\frac{-1}{2k+2m}((2kPV^{k-1,U}_x-2mV^{k,U}_x)R^U_\pm(2k+2m,x))\\
		&+(PV^{k,U}_x)R^U_\pm(2k+2m+2,x)\Big)\\
		&=\binom{k+m}{k}\frac{2m}{2k+2m}V^{k,U}_xR^U_\pm(2k+2m,x)-\binom{k+m}{k}\frac{2k}{2m+2k} (PV^{k-1,U}_x)R^U_\pm(2k+2m,x)\\
		&+\binom{k+m}{k}(PV^{k,U}_x)R^U_\pm(2k+2m+2,x)\\
		&=\binom{k+m-1}{k}V^{k,U}_xR^U_\pm(2k+2m,x)-\binom{k+m-1}{k-1} (PV^{k-1,U}_x)R^U_\pm(2k+2m,x)\\
		&+\binom{k+m}{k}(PV^{k,U}_x)R^U_\pm(2k+2m+2,x))\\
	\end{align*}
	In the special case $k=m=0$, we obtain:
	\begin{align*}
		P(V^{0,U}_xR_\pm(2,x))&=\lim\limits_{\alpha\rightarrow 0}P(V^{0,U}_xR_\pm(\alpha+2,x))\\
		&=\lim\limits_{\alpha\rightarrow 0}P(V^{0,U}_x)R_\pm(\alpha+2,x)-(\frac{1}{\alpha}\rho_x^U V^{0,U}_x- V^{0,U}_x)R_\pm(\alpha,x)\\	 
		&=\lim\limits_{\alpha\rightarrow 0}P(V^{0,U}_x)R_\pm(\alpha+2,x)+ V^{0,U}_xR_\pm(\alpha,x)\\	
		&=P(V^{0,U}_x)R_\pm(2,x)+ V^{0,U}_xR_\pm(0,x)\\
		&=P(V^{0,U}_x)R_\pm(2,x)+ V^{0,U}_x(x)\delta_x\\
		&=\delta_x+P(V^{0,U}_x)R_\pm(2,x).
	\end{align*}
either way, we obtain for $k=0$:
\[P(V^{0,U}_xR^U_\pm(2m+2,x))=V^{0,U}_xR^U_\pm(2m,x)+(PV^{0,U}_x)R^U_\pm(2m+2,x)\]
	
	Putting everything together, for any $N\in\N$ in case $m\geq 0$ and for any $N<-m$ if $m<0$, we have
	\begin{align*}
		&P\sum\limits_{k=0}^N\binom{m+k}{k}V^{k,U}_xR^U_\pm(2k+2m+2,x)\\
		&=V^{0,U}_xR^U_\pm(2m,x)+(PV^{0,U}_x)R^U_\pm(2m+2,x)+\sum\limits_{k=1}^N\binom{k+m-1}{k}V^{k,U}_xR^U_\pm(2k+2m,x)\\
		&-\binom{k+m-1}{k-1} (PV^{k-1,U}_x)R^U_\pm(2k+2m,x) +\binom{k+m}{k}(PV^{k,U}_x)R^U_\pm(2k+2m+2,x)\\
		&=\sum\limits_{k=0}^N\binom{k+m-1}{k}V^{k,U}_xR^U_\pm(2k+2m,x)-\sum\limits_{k=0}^{N-1}\binom{k+m}{k}(PV^{k,U}_x)R^U_\pm(2k+2m+2,x)\\
		&+\sum\limits_{k=0}^{N}\binom{k+m}{k}(PV^{k,U}_x)R^U_\pm(2k+2m+2,x))\\
		&=\sum\limits_{k=0}^N\binom{k+m-1}{k}V^{k,U}_xR^U_\pm(2k+2m,x)+\binom{N+m}{N}(PV^{N,U}_x)R^U_\pm(2N+2m+2,x)\\
		&=:\sum\limits_{k=0}^N\binom{k+m-1}{k}V^{k,U}_xR^U_\pm(2k+2m,x)+E_N(\cdot,x)\\
	\end{align*}
where 
\[E_N(\cdot,x):=\binom{N+m}{N}(PV^{N,U}_x)R^U_\pm(2N+2m+2,x)\]
is $C^{n}$ in both variables for $N\geq \frac{d}{2}+n-m$.

\textbf{Case 1: $m+1\Rightarrow m$, $m<0$}\\
Suppose for induction that
\[G^{\pm m+1}_x=\sum\limits_{k=0}^{-m-1}\binom{m+k}{k}V^{k,U}_xR^U_\pm(2k+2m+2,x)\]
(This implies the theorem for $m+1$, as all further summands are 0). Let $N=-m-1$. 
Using the fact (see \cite[below Proposition 2.3.1]{BGP} that 
\[PV^{N,U}_x(x)=-V^{N+1,U}_x(x)\]
and the identity
\[\binom{-a}{b}=(-1)^b\binom{a+b-1}{b}\] 
for integer $b$, we calculate
\begin{align*}
	E_N(\cdot,x)&=\binom{-1}{N}(PV^{N,U}_x)R^U_\pm(0,x)\\
	&=(-1)^N\binom{N}{N}(PV^{N,U}_x)\delta_x\\
	&=(-1)^NPV^{N,U}_x(x)\delta_x\\
	&=(-1)^{N+1}V^{N+1,U}_x(x)\delta_x\\
	&=(-1)^{N+1}\binom{N+1}{N+1}V^{N+1,U}_x(x)\delta_x\\
	&=\binom{-1}{N+1}V^{N+1,U}_xR^U_\pm(0,x)\\
	&=\binom{-1}{-m}V_x^{-m,U}R_\pm^U(0,x).
\end{align*}
We obtain from the inductive hypothesis and our previous calculations:
\begin{align*}
	G^{\pm m}_x
	&=PG^{\pm m+1}_x\\
	&=P\sum\limits_{k=0}^{N}\binom{m+k}{k}V^{k,U}_xR^U_\pm(2k+2m+2,x)\\
	&=\sum\limits_{k=0}^N\binom{k+m-1}{k}V^{k,U}_xR^U_\pm(2k+2m,x)+E_N(\cdot,x)\\
	&=\sum\limits_{k=0}^{-m-1}\binom{k+m-1}{k}V^{k,U}_xR^U_\pm(2k+2m,x)+\binom{-1}{-m}V^{-m,U}_xR^U_\pm(0,x)\\
	&=\sum\limits_{k=0}^{-m}\binom{k+m-1}{k}V^{k,U}_xR^U_\pm(2k+2m,x)\\
\end{align*}

\textbf{Case 2: $m\Rightarrow m+1$, $m\geq 0$}\\
Fix $n\in \N$ and assume the theorem holds for some $m\geq 0$. We first show that the expansion holds on relatively compact GE subsets of $U$. Let $W,O\subseteq U$ be relatively compact (in $U$) and GE. Proposition \ref{Greg} implies that there is $n'\in \N$ such that $G^\pm_{P|_U}$ maps  $\Gamma^{n'}_{J^U_\pm(\overline O)\times O}(E|_U\boxtimes E^*|_U)$ to sections that are $C^n$ on $W\times O$. Assume $x\in O$. By the inductive hypothesis, we may choose $N\geq \frac{d}{2}+n'$ such that
\[F_N(\cdot,x):=G^{\pm m}_x|_U-\sum\limits_{k=0}^N \binom{m+k-1}{k}V^{k,U}_xR^U_\pm(2k+2m,x)\]
defines a $C^{n'}$-section $F_N$ on $U$. We then have
\begin{align*}
	&G^{\pm m+1}_x|_U\\
	&=(G^\pm_P G^{\pm m}_x)|_U\\
	&=G^\pm_{P|_U}G^{\pm m}_x|_U\\
	&=G^\pm_{P|_U}\left(\sum\limits_{k=0}^N \binom{m+k-1}{k}V^{k,U}_xR^U_\pm(2k+2m,x)+F_N(\cdot,x)\right)\\
	&=G^\pm_{P|_U}\left(P\sum\limits_{k=0}^N\binom{m+k}{k}V^{k,U}_xR^U_\pm(2k+2m+2,x)-E_N(\cdot,x)+F_N(\cdot,x)\right)\\
	&=\sum\limits_{k=0}^N\binom{m+k}{k}V^{k,U}_xR^U_\pm(2k+2m+2,x)+G^\pm_{P|_U}(F_N(\cdot,x)-E_N(\cdot,x))\\
	&=:\sum\limits_{k=0}^N\binom{m+k}{k}V^{k,U}_xR^U_\pm(2k+2m+2,x)+\tilde F_N(\cdot,x).\\
\end{align*}
with
\[\tilde F_N(\cdot,x):=G^{\pm m+1}_x|_U-\sum\limits_{k=0}^N\binom{m+k}{k}V^{k,U}_xR^U_\pm(2k+2m+2,x)=G^\pm_{P|_U}(F_N(\cdot,x)-E_N(\cdot,x)).\]
By our choice of ${n'}$ and $N$, $\tilde F_N$ is $C^n$ on $W\times O$, as both $E_N|_{U\times O}$ and $F_N|_{U\times O}$ are $C^{n'}$ and supported in $J_\pm(O)\times O$. Note that the only choices we made after choosing $W$ and $O$ (apart from assuming $x\in O$) were that of ${n'}$ and $N$. We thus want to make the expansion independent of the choice of $N$.

Let $N_0\geq\frac{d}{2}+n$ be arbitrary. For $k\geq N_0$, all Riesz distributions $R^U_\pm(2k+2m+2,x)$ are given by $C^n$-functions. We see that
\begin{align*} 
	\tilde F_{N_0}(\cdot,x)&=G^{\pm m+1}_x|_U-\sum\limits_{k=0}^{N_0}\binom{m+k}{k}V^{k,U}_xR^U_\pm(2k+2m+2,x)\\
	&=\sum\limits_{k=N_0+1}^{N}\binom{m+k}{k}V^{k,U}_xR^U_\pm(2k+2m+2,x)+\tilde F_N(\cdot,x)
\end{align*}
is $C^n$ on $W\times O$. As this is independent of $W$ and $O$, these can replaced by arbitrary relatively compact GE sets. All points in $U$ have relatively compact GE neighborhoods (choose any relatively compact neighborhood and then choose a GE neighborhood contained in it by Proposition \ref{GEN}). Thus any $(y,x)\in U\times U$ has a neighborhood on which $\tilde F_{N_0}$ is $C^n$, which means it is $C^n$ on all of $U\times U$. As $n$ was arbitrary, we have the desired asymptotic expansion for $m+1$ and have thus concluded our induction.
\end{proof}

\section{Green's ``resolvent''}
\label{Greschap}
We now consider the Green's operators for the operator $P-z$ for $z\in \C$, which is still normally hyperbolic. If we view $G^\pm_P$ as something like an inverse for $P$, then $G^\pm_{P-z}$ is something like a resolvent. We seek to derive an expansion similar to the one above, where only the Hadamard coefficients of $P$ rather than $P-z$ appear and the $z$-dependence is instead shifted to a generalzed version of Riesz distributions. For this section, fix a convex open set $U$ and let $p\in U$ be some point. Any Hadamard coefficients, Riesz distributions and similar objects depending on a convex open subset will be defined on $U$ unless otherwise indicated and we will omit the superscript $U$.
\begin{df}
	\label{Rieszres}
	Let $\square_V$ denote the d'Alembertian on a Lorentzian vector space $V$.
	
	For $z\in \C$ and $m \in \N$, we define resolvent Riesz distributions as
	\[R_\pm(z,2m,x):=(\exp_x)^{-1*}(G^{\pm m}_{\square_{V}-z,0})\]
\end{df}
The motivation for this is that it reproduces the original Riesz-distributions in the case $z=0$:
\begin{prop}
	\label{RFS}
	On a Lorentzian vector space $V$, we have
	\[R^V_\pm(2m)=G^{\pm m}_{\square_{V},0}.\]
	Thus we have
	\[R_\pm(0,2m,x)=R_\pm(2m,x).\]
\end{prop}
\begin{proof}
	We have (\cite[Lemma 1.2.2]{BGP})
	\[\square_V R_\pm^V(2m+2)=R_\pm^V(2m)\]
	and hence
	\[R_\pm^V(2m+2)=G^\pm_{\square_V}\square_V R_\pm^V(2m+2)=G^\pm_{\square_V}R_\pm^V(2m).\]
	Together with
	\[R_\pm^V(0)=\delta_0,\]
	we find that the $R_\pm^V(2m)$ satisfy the recursion defining $G^{\pm m}_{\square_V,0}$, hence we have shown the first part. 
	
	The second part follows by inserting the first part into the definition of the Riesz distributions on $U$ and comparing with that of the resolvent Riesz distributions.
\end{proof}
We want to exploit the asymptotic expansion that we already have in order to obtain the one we would like to get. For this, we need to know the Hadamard coefficients for $P-z$ in terms of those for $P$.
\begin{df}
	\label{defVz}
	For $z\in \C$,  let $V^{k}_{x}(z)$ be the corresponding Hadamard coefficients for $P-z$.
\end{df}
\begin{prop}
	\label{Vz}
	We have
	\[V^{k}_{x}(z)=\sum\limits_{m=0}^k\binom{k}{m}z^mV^{k-m}_x.\]
\end{prop}
\begin{rem}
	The corresponding formula for the heat coefficients
	\[a_k(\Delta-z)=\sum\limits_{m=0}^k\frac{1}{m}z^ma_{k-m}(\Delta)\]
	can be shown analogously. Alternatively, that formula can also be deduced by Taylor expanding $e^{tz}$ and multiplying out the expansions in 
	\[e^{-t(\Delta-z)}=e^{tz}e^{-t\Delta}.\]
\end{rem}
\begin{proof}
	 Let
	 \[V_k(z):=\sum\limits_{m=0}^k\binom{k}{m}z^mV^{k-m}_x\]
	 We need to show that $V_0(z)(x,x)=1$ and the $V_k(z)$ satisfy the transport equations.
	The former holds, as $V_0(z)=V^0_x$. For the latter, we calculate using the transport equation for $P$:
	\begin{align*}
		&(\rho_x^U-2k)V_k(z)\\
		&=\sum\limits_{m=0}^k\binom{k}{m}z^m(\rho_x^U-2k)V^{k-m}_x\\
		&=\sum\limits_{m=0}^k\binom{k}{m}z^m((\rho_x^U-2(k-m))V^{k-m}_x-2mV^{k-m}_x)\\
		&=\sum\limits_{m=0}^k\binom{k}{m}z^m2(k-m)PV^{k-m-1}_x-\sum\limits_{m=0}^k\binom{k}{m}z^m2mV^{k-m}_x\\
		&=\sum\limits_{m=0}^k\binom{k-1}{m}z^m2kPV^{k-m-1}_x-\sum\limits_{m=1}^k\binom{k-1}{m-1}z^m2kV^{k-m}_x\\
		&=\sum\limits_{m=0}^{k-1}\binom{k-1}{m}z^m2kPV^{k-m-1}_x-\sum\limits_{m=0}^{k-1}\binom{k-1}{m}z^{m+1}2kV^{k-m-1}_x\\
		&=2k(P-z)\sum\limits_{m=0}^{k-1}\binom{k-1}{m}z^mV^{k-m-1}_x\\
		&=2k(P-z)V_{k-1}(z).
	\end{align*}
	The $V_k(z)$ thus satisfy the transport equations for $P-z$, so they are the desired Hadamard coefficients.
\end{proof}
The other ingredient we need is an expansion for the resolvent Riesz distributions in terms of the standard ones.
\begin{lemma}
	On Minkowski space, we have the asymptotic expansion
	\[G^{\pm k}_{\square-z,0}\sim\insum{m}\binom{k+m-1}{m}z^mR_\pm(2m+2k).\]
\end{lemma}
\begin{proof}
	As $R_\pm(2m+2k)=R^{\R^d}_\pm(2m+2k,0)$, Theorem \ref{PowExp} applied to $\square-z$ gives us the asymptotic expansion
	\[G^{\pm k}_{\square-z,0}\sim\insum{m}\binom{k+m-1}{m}V_m(z)R_\pm(2m+2k),\]
	where the $V_m(z)$ are the Hadamard coefficients at $0$ for $\square-z$, i.e. the unique solutions to 
	\[V_0(z)(0)=1\]
	and
	\[\rho_0^{{\R^d}} V_m(z)-2mV_m(z)=2m(\square-z) V_{m-1}(z).\]
	These equations are solved by the constant functions $V_m(z):=z^m$, as
	\[\rho_0^{\R^d}z^m=\<\grad\gamma,\grad z^m\>+(\square\gamma-d)z^m=0\]
	and
	\[\square z^m=0,\]
	so we obtain
	\[G^{\pm k}_{\square-z,0}\sim\insum{m}\binom{k+m-1}{m}z^mR_\pm(2m+2k).\qedhere\]
\end{proof}
This carries over from Minkowski space to $M$ in a straightforward way:
\begin{lemma}
	\label{Rzex}
	We have the asymptotic expansion
	\[R_\pm(z,2k,x)\sim_x \insum{m}\binom{k+m-1}{m}z^mR^{U}_\pm(2m+2k,x).\]
\end{lemma}
\begin{proof}
	As $U$ is convex, it is contractible and thus $TU$ is trivial. Let $\phi\colon U\times\R^d\rightarrow TU$ be an isometric trivialization that preserves time-orientation.
	As $\phi$ is a time-orientation preserving isometry on every fibre, it preserves all objects defined only in terms of the metric and time-orientation, i.e.
	\[\phi_x^{-1*}(R^{\R^ d}_\pm(2m+2k))=R^{T_x M}_\pm(2m+2k)\]
		and
	\[\phi_x^{-1*}(G^{\pm k}_{\square_{\R^d}-z,0})=G^{\pm k}_{\square_{T_xM}-z,0}\]
	Let \[\Exp\colon \Dom(\Exp)\subseteq TU\rightarrow U\times U\]
	denote the exponential map with variable basepoint, i.e. $\Exp(v)=(x,\exp_x(v))$ for $v\in T_xU$. This is smooth.
	
	Let $n\in\N$ be arbitrary. For $N$ large enough such that
	\[G^{\pm k}_{\square_{\R^d}-z,0}=\sum\limits_{m=0}^{N}\binom{k+m-1}{m}z^mR^{\R^ d}_\pm(2m+2k)+F\]
	with $F\in C^n(\R^d)$,
	we have
	\begin{align*}
	&R_\pm(z,2k,x)\\
	&=(\exp_x)^{-1*}(G^{\pm k}_{\square_{T_xM}-z,0})\\
	&=(\exp_x\circ\phi_x)^{-1*}(G^{\pm k}_{\square_{\R^d}-z,0})\\
	&=(\exp_x\circ\phi_x)^{-1*}\left(\sum\limits_{m=0}^{N}\binom{k+m-1}{m}z^mR^{\R^ d}_\pm(2m+2k)+F\right)\\
	&=\sum\limits_{m=0}^{N}\binom{k+m-1}{m}z^m(exp_x)^{-1*}R^{T_xM}_\pm(2m+2k)+F\circ\phi_x^{-1}\circ exp_x^{-1}\\
	&=\sum\limits_{m=0}^{N}\binom{k+m-1}{m}z^mR^{U}_\pm(2m+2k,x)+F\circ\phi^{-1}\circ \Exp^{-1}(\cdot,x).\\
	\end{align*}
	as $F\circ\phi^{-1}\circ \Exp^{-1}$ is $C^n$ and $n$ was arbitrary, we have the desired expansion.
\end{proof}
We now have everything in place to prove the main result of this section: If, in the asymptotic expansions for the Green's operators of $P$, we replace the standard Riesz distributions with the resolvent Riesz distributions for $z\in \C$, then we obtain asymptotic expansions for the Green's operators of $P-z$.
\begin{thm}
	\label{Hadamexpz}
	If $U\subseteq M$ is GE, we have the following asymptotic expansion in $U$:
	\[G^{\pm}_{P-z,x}|_U\sim_x \insum{k}V^{k,U}_xR^U_\pm(z,2k+2,x).\]
	More generally,
	\[G^{\pm m}_{P-z,x}|_U\sim_x \insum{k}\binom{k+m-1}{k}V^{k,U}_xR^U_\pm(z,2k+2m,x).\]
\end{thm}
\begin{proof}
	Applying Theorem \ref{PowExp} to $P-z$ and inserting our formula for the $z$-dependent Hadamard coefficients (Proposition \ref{Vz}), we obtain
	\begin{align*}
		&G^{\pm m}_{P-z,x}|_U\\
		&\sim_x \insum{k}\binom{k+m-1}{k}V^{k}_x(z)R_\pm(2k+2m,x)\\
		&=\insum{k}\binom{k+m-1}{k}\sum\limits_{l+n=k}\binom{k}{n}z^nV^{l}_xR_\pm(2k+2m,x)\\
		&=\insum{l}\insum{n}\binom{n+l+m-1}{n+l}\binom{n+l}{n}z^nV^{l}_xR_\pm(2(n+m+l),x)\\
		&=\insum{l}\insum{n}\binom{n+l+m-1}{n}\binom{l+m-1}{l}z^nV^{l}_xR_\pm(2(n+m+l),x)\\
		&=\insum{l}\binom{l+m-1}{l}V^{l}_x\insum{n}\binom{n+l+m-1}{n}z^nR_\pm(2(n+m+l),x)\\
		&\sim_x\insum{l}\binom{l+m-1}{l}V^{l}_xR_\pm(z,2(m+l),x).\\
	\end{align*}
 Here we used the identity
\[\binom{a}{b}\binom{b}{c}=\binom{a}{c}\binom{a-c}{b-c}\]
and Lemma \ref{Rzex}.

This proves the second claim, the first claim follows as the special case $m=1$.
\end{proof}
\newpage
\chapter{Constructing an action-like function}
\label{timeint}
In this chapter we will take our first step towards extracting the Hadamard coefficients from the asymptotic expansion we have seen. We will construct a function from the Green's operator that has an asymptotic expansion in terms of the Hadamard coefficients. This is similar to an analogue of a spectral action, but with additional terms coming from the Hadamard coefficients' derivatives. 

For technical reasons, we'll be looking not at the two Green's operators individually, but at their difference, sometimes referred to as the causal propagator.
\begin{df}
	\label{dfcp}
	Set
	\[G:=G^+-G^-,\]
	\[G_x:=G^+_x-G^-_x,\]
	\[R(\alpha):=R_+(\alpha)-R_-(\alpha)\]
	and
	\[R^U(\alpha,x):=R^U_+(\alpha,x)-R^U_-(\alpha,x).\]
\end{df}
Taking the difference of the asymptotic expansions for $G^\pm_x$, we obtain
	\[G_{x}|_U\sim_x\insum{k}V^{k,U}_xR^{U}(2k+2,x).\]
We are motivated by the following formal calculation:

Let $w$ be a timelike geodesic in $M$ and let $x=w(0)$. Then we have (for $\Re(\alpha)$ large enough)
\[R^U(\alpha,x)(w(t))=c_\alpha |t|^{\alpha-d}\sign(t)\]
Writing 
\[\tilde V_k(t):=V^{k,U}_x(w(t)),\]
we can taylor expand
\[\tilde V_k(t)\sim\insum{n}\frac{1}{n!}\tilde V_k^{(n)}(0)t^n.\]
For some function $f\in C_c^\infty(\R)$ and $s>0$, we can then calculate (formally):
\begin{align*}
	&\intR G_x(w(t)) f(\tfrac{t}{s})dt\\
	&\sim \insum k \intR V^{k,U}_x(w(t))R^U(2k+2,x)(w(t))f(\tfrac{t}{s})dt\\
	&= \insum k \intR \tilde V_k(t)c_{2k+2}t^{2k+2-d}\sign(t)f(\tfrac{t}{s})dt\\
	&\sim \insum{k}c_{2k+2}\intR  \insum{n}\frac{1}{n!}\tilde V_k^{(n)}(0)t^{2k+n+2-d}\sign(t)f(\tfrac{t}{s})dt\\
	&=\insum{k}\insum{n}\frac{c_{2k+2}}{n!}\tilde V_k^{(n)}(0)s^{2k+n+3-d}\intR \frac{1}{s}  \left(\frac{t}{s}\right)^{2k+n+2-d}\sign(\tfrac{t}{s})f(\tfrac{t}{s})dt\\
	&=\insum{k}\insum{n}\frac{c_{2k+2}}{n!}\tilde V_k^{(n)}(0)s^{2k+n+3-d}\intR  t^{2k+n+2-d}\sign(t)f({t})dt\\
	&=:\insum{k}\insum{n}a_{k,n}(f)\tilde V_k^{(n)}(0)s^{2k+n+3-d}
\end{align*}
with some coefficients $a_{k,n}(f)$ independent of $s$, $M$ and $P$ (except through the dimension). From this, or something similar, one might hope to extract the values of the Hadamard coefficients at the diagonal, i.e. $\tilde V_k(0)$. The problem of performing this extraction is postponed to the next chapter, as this chapter is devoted to making the above calculation precise. The main problem in this is that, for small $k$, the Riesz-distributions are not given by functions, so evaluation at $w(t)$ is ill-defined.

\section{Holomorphicity of Riesz distributions for some wavefront}
While evaluating the Riesz-distributions at a single point is problematic, we will show that restricting them to a timelike geodesic is possible. By wavefront calculus, it suffices to show that the Riesz-distributions' wavefront does not intersect the conormal bundle of the curve. To actually compute what the restriction does, we also need that the Riesz-distributions are holomorphic in the space of distributions with a suitable wavefront, so that the restriction is still holomorphic. Showing these two things is the goal of this section.

The conormal bundle of a timelike geodesic consists of spacelike vectors. Therefore we want to show that the wavefront set of the Riesz distributions does not contain spacelike vectors. This is the point where we have to use that we work with the difference $R=R_+-R_-$, rather than the advanced or retarded Riesz distributions individually. In this difference, the worst singularities at $0$ will cancel out. As wavefront calculus is invariant under diffeomorphisms, it will be sufficient to work on Minkowski space.

We thus consider a spacelike vector $\xi\in \R^d$. Deciding whether $\xi$ is in the wavefront set of a function involves integration against $e^{-i\<\xi,\cdot\>}$. Our strategy is to find, for each $x\in \R^d$, another point $O_\xi(x)$ such that 
\[e^{-i\<\xi,O_\xi(x)\>}=e^{i\<\xi,x\>}\]
and such that the value of $R_+(\alpha)$ at $x$ cancels with that of $R_-(\alpha)$ at $O_\xi(x)$. Thus the integral of $R(\alpha)$ over each level set of $e^{-i\<\xi,\cdot\>}$ vanishes ($O_\xi$ will also be volume preserving). We will construct $O_\xi$ on each level set as a reflection through (multiples of) a suitable vector $y_\xi$. $y_\xi$ and $O_\xi$ will be constructed in the next two theorems and are illustrated in the following graphic.
\newline
\begin{center}
	\begin{tikzpicture}
		\draw [dotted](-4,-4) -- (4,4);
		\draw [dotted](-4,4) -- (4,-4);
		\draw [green, -stealth] (0,0) -- (2,-1) node[below]{$\xi$};
		\draw [blue, dotted] (-1,-4) -- (3,4);
		\node at (0.2,-3) [blue]{$x+\xi^\bot$};
		\filldraw [blue] (7/3,8/3) circle (1pt);
		\node [blue, left] at (7/3,8/3){$x$};
		\draw [blue, -stealth] (7/3,8/3) -- (1/3,-4/3) node[below]{$O_\xi(x)$};
		\draw [orange, -stealth] (0,0) -- (4/3,2/3) node[right]{$\<\xi,x\> y_\xi$};
		\draw [orange, dotted] (-4,-2) -- (4,2);
		\filldraw [cyan] (-3/4,1.5) circle (1pt);
		\node [cyan, left] at (-3/4,1.5){$x'$};
		\draw [cyan, -stealth] (-3/4,1.5) -- (-3.25,-3.5) node[right]{$O_\xi(x')$};
		\filldraw [cyan] (-1/2,1/2) circle (1pt);
		\node [cyan, left] at (-1/2,1/2){$x''$};
		\draw [cyan, -stealth] (-1/2,1/2) -- (-3/2,-3/2) node[right]{$O_\xi(x'')$};
	\end{tikzpicture}
\end{center}	
\begin{defprop}
	\label{yxi}
	For any spacelike unit vector $\xi\in \R^d$, there is a spacelike $y_\xi\in \R^d$ such that $\<y_\xi,\xi\>=1$ and for all $v\bot\xi$ and $\lambda\in\R$, we have
	\[\gamma(\lambda y_\xi+v)=\gamma(\lambda y_\xi-v).\]
	(scalar product and orthogonality refer to the euclidean scalar product on $\R^d$.)
\end{defprop}
\begin{rem}
	On each $\xi$-orthogonal hyperplane $\<\cdot, \xi\>=\lambda$, $\gamma$ is given by a (higher-dimensional) parabola. $\lambda y_\xi$ is the minimum of this parabola, around which it is symmetric.
\end{rem}
\begin{proof}
	Without loss of generality, we may assume that \[\xi=\cos(\theta)e_0+\sin(\theta)e_1,\] 
	where $e_i$ is the $i$-th standard unit vector and
	 \[\theta:=\arccos(\<\xi,e_0\>)\in(\frac{\pi}{4},\frac{3\pi}{4}),\] 
	 as the problem is symmetric under rotation in the spatial coordinates (the condition on $\theta$ is equivalent to $|\xi_0|<|\xi_1|$). Let $R$ be the rotation in the first two coordinates that maps $e_0$ to $\xi$, i.e. 
	\[R=\mat{\cos(\theta)}{-\sin(\theta)}{\sin(\theta)}{\cos(\theta)}\oplus 1.\]
	$R$ to maps planes orthogonal to $\xi$ to planes where $x_0$ is constant. There the problem basically reduces to finding the minimum of a quadratic function in $x_1$:
	Writing $x_r$ to denote the last $d-2$ components of $x\in \R^d$, we have
	\begin{align*}
		\gamma(Rx)&=\gamma(\cos(\theta) x_0-\sin(\theta) x_1,\sin(\theta) x_0+\cos(\theta) x_1,x_r)\\
		&=(\cos(\theta) x_0-\sin(\theta) x_1)^2-(\sin(\theta) x_0+\cos(\theta) x_1)^2-\|x_r\|^2\\
		&=(\cos(\theta)^2-\sin(\theta)^2)(x_0^2-x_1^2)-4\sin(\theta) \cos(\theta) x_0x_1-\|x_r\|^2\\
		&=\cos(2\theta)(x_0^2-x_1^2)-2\sin(2\theta)x_0x_1-\|x_r\|^2\\
		&=-\cos(2\theta)((x_1+\tan(2\theta)x_0)^2-(1+\tan(2\theta)^2)x_0^2)-\|x_r\|^2.
	\end{align*}
	On the plane $x_0=\lambda$, this takes its minimum at $\lambda e_0-\lambda\tan(2\theta)e_1$, which we will thus take as our point for reflection.
	We have for $w\bot e_0$ and $\lambda\in \R$:
	\begin{align*}
		\gamma(R(\lambda e_0-\lambda \tan(2\theta)e_1+w))&=-\cos(2\theta)((w_1)^2-(1+\tan(2\theta)^2)\lambda^2)-\|w_r\|^2\\
		&=-\cos(2\theta)((-w_1)^2-(1+\tan(2\theta)^2)\lambda^2)-\|-w_r\|^2\\
		&=\gamma(R(\lambda e_0-\lambda \tan(2\theta)e_1-w)).
	\end{align*}
	To solve our original problem on the plane $\<x,\xi\>=\lambda$, we simply rotate back.
	Define $y_\xi:= R(e_0-\tan(2\theta)e_1)$. If $v\bot\xi$, then $R^{-1}v\bot e_0$, so we have
	\[\gamma(\lambda y_\xi+v)=\gamma(R(\lambda e_0-\lambda \tan(2\theta)e_1+R^{-1}v))=\gamma(R(\lambda e_0-\lambda \tan(2\theta)e_1-R^{-1}v))=\gamma(\lambda y_\xi-v)\]
	as desired. It remains to show that $y_\xi$ is indeed spacelike. We can reuse the calculation above to obtain
	\begin{align*}
		\gamma(y_\xi)&=\gamma(R(1e_0-\tan(2\theta)e_1))\\
		&=-\cos(2\theta)(0-(1+\tan(2\theta)^2)1)-0\\
		&=\cos(2\theta)\left(\frac{\cos(2\theta)^2+\sin(2\theta)^2}{\cos(2\theta)^2}\right)\\
		&=\frac{1}{\cos(2\theta)}.
	\end{align*}
	As $2\theta\in(\frac{\pi}{2},\frac{3\pi}{2})$, this is negative, hence $y_\xi$ is spacelike.
\end{proof} 
We can now define $O_\xi$ by reflecting through $y_\xi$ in each $\xi$-orthogonal hyperplane:
\begin{defprop}
	\label{Oxi}
	For each spacelike $\xi$ in $\R^d$, the linear map $O_\xi$ on $\R^d$ defined by
	\[O_\xi(x):=2\<x,\xi\>y_\xi-x\]
	satisfies:
	\begin{enumerate}
		\item \label{O0} For any $x\in \R^d$, there are $\lambda\in \R$ and $v\bot \xi$ such that $x=\lambda y_\xi+v$ and $O_\xi(x)=\lambda y_\xi-v$
		\item \label{O1}$O_\xi$ leaves hyperplanes orthogonal to $\xi$ invariant, i.e. for all $x\in \R^d$:
		\[\<x,\xi\>=\<O_\xi(x),\xi\>\]
		\item \label{O2}$O_\xi$ is a Lorentz-transformation, i.e. \[\gamma\circ O_\xi=\gamma\]
		\item \label{O3}$O_\xi$ preserves volume, i.e.
		\[|det(O_\xi)|=1\]
		\item \label{O4}$O_\xi$ maps future oriented vectors to past oriented vectors, i.e. \[\sign(O_\xi(x)_0)=-\sign(x_0)\] for any causal vector $x$.
	\end{enumerate}
\end{defprop}
\begin{proof}
	\begin{enumerate}
		\item $y_\xi$ is not in $\xi^\bot$ and hence 
		\[\spann (\xi^\bot\cup\{y_\xi\})=\R^d.\] 
		Thus any $x\in \R^d$ can be written as in Part \ref{O0}.
		As 
		\[O_\xi(\lambda y_\xi+v)=2\<\lambda y_\xi,\xi\>y_\xi-(\lambda y_\xi+v)=\lambda y_\xi-v,\]
		we get Part \ref{O0}.
		We also get Part \ref{O1}, as
		\[\<\lambda y_\xi+v,\xi\>=\<\lambda y_\xi,\xi\>=\<\lambda y_\xi-v,\xi\>.\]
		Part \ref{O2} immediately follows from Part \ref{O0} and Lemma \ref{yxi}.
		The more customary definition of a Lorentz transform, that
		\[\eta(O_\xi(x),O_\xi(y))=\eta(x,y),\]
		where $\eta$ is the Minkowski metric, follows from Part \ref{O2} and the polarization identity
		\[-\eta(x,y)=\frac{1}{2}(\gamma(x+y)-\gamma(x)-\gamma(y)).\]
		Part \ref{O3} is true for any Lorentz transformation and follows form applying determinants to the matrix version of the Lorentz condition
		\[O_\xi^T\eta O_\xi=1.\]
		Finally, note that
		\[\frac{1}{2}(x+O_\xi(x))=\<x,\xi\>y_\xi\]
		is spacelike and in the convex hull of $x$ and $O_\xi(x)$. Lorentz transformations map causal vectors to causal vectors and the past and future solid lightcone are each convex. Thus, for any causal vector $x$, $x$ and $O_\xi(x)$ must be contained in different lightcones, otherwise their convex hull could not contain a spacelike vector. This proves Part \ref{O4}.\qedhere
	\end{enumerate}
\end{proof}
\begin{df}
	\label{dfsigma}
	For $x\in \R^d$ define $\sigma(x):=sign(x_0)$.
\end{df}
We will consider general functions of the form
\[\sigma\gamma^*(f)(x)=\sigma(x)f(\gamma(x))\]
for functions $f$ on $\R$ with support in $[0,\infty)$ and singular support at $0$. This includes $R(\alpha)$ for $\Re(\alpha)$ large enough as the special case $f(x)=c_\alpha\mathbbm1 _{x\geq0}x^\frac{\alpha-d}{2}$. We want to show that no spacelike vector $\xi$ is in the wavefront of $\sigma\gamma^*f$.

Morally, what we want to do is to use the transformation formula and invariance (up to sign) of the integrand under $O_\xi$ to compute (ignoring integrability concerns)
\begin{align*}
	\F(\sigma\gamma^*f)(\xi)&=\int\limits_{\R^d} e^{-i\<x,\xi\>}\sigma(x)f(\gamma(x))dx\\
	&=\int\limits_{\R^d} e^{-i\<O_\xi(x),\xi\>}\sigma(O_\xi(x))f(\gamma(O_\xi(x)))dx\\
	&=\int\limits_{\R^d} e^{-i\<x,\xi\>}(-\sigma(x))f(\gamma(x))dxdz\\
	&=-\F(\sigma\gamma^*f)(\xi).
\end{align*}
and conclude that $\F(\sigma\gamma^*f)(\xi)$ is 0 and hence rapidly decaying. However, we need to multiply the integrand with arbitrary cutoff-functions to obtain wavefront norms of $\sigma\gamma^*(f)$. These will not be invariant under $O_\xi$, so we cannot immediately apply the reasoning above. The strategy is to first introduce a suitable family of $\xi$-invariant cutoff-functions and then use those to show the statement for arbitrary cutoffs.
\begin{lemma}
	\label{phixi}
	Let $V$ be a closed cone of spacelike vectors. Then there is a family of $\psi_\xi\in C^\infty_c(\R^d)$ for ${\xi\in V\cap S^{d-1}}$  satisfying the following:
	\begin{enumerate}
		\item \label{phi1} $\psi_\xi\circ O_\xi=\psi_\xi$
		\item \label{phi2} all $\psi_\xi$ are constantly 1 on the unit ball $B_1(0)$
		\item \label{phi3} there is a compact set $K$ such that all $\psi_\xi$ are supported in $K$.
		\item \label{phi4} all $\|\psi_\xi\|_{C^k}$ are bounded uniformly in $\xi$.
	\end{enumerate}
\end{lemma}
\begin{proof}
	Let $\theta$ again denote the angle between $\xi$ and $e_0$. As $V\cap S^{d-1}$ is compact, $\tan(2\theta)$ is bounded on $V$ by some constant $C-1$. Then $y_\xi$, as defined in \ref{yxi} has norm at most $C$ for $\xi$ in $V$.  Choose $\chi\in C^\infty_c(\R)$ that is 1 on $[-1,1]$ and $0$ outside $[-2,2]$. Choose $\phi\in C_c^\infty(\Rp)$ such that $\phi(x)$ is $1$ for $x\leq C+1$ and $0$ for $x\geq C+2$. Define 
	\[\psi_\xi(x):=\chi(\<x,\xi\>)\phi(\|x-\<x,\xi\>y_\xi\|).\]
	For $\lambda\in \R$, $v\in\xi^\bot$ and $x=\lambda y_\xi+v$, we have $\<x,y_\xi\>=\lambda$ and hence
	\[\psi_\xi(x)=\chi(\lambda)\phi(\|\lambda y_\xi+v-\lambda y_\xi\|)=\chi(\lambda)\phi(\|v\|)=\chi(\lambda)\phi(\|-v\|)=\psi_\xi(\lambda y_\xi-v)=\psi_\xi(O_\xi(x)).\]
	Thus \ref{phi1}. holds.
	
	For $x\in B_1(0)$, we have
	\[|\<x,\xi\>|\leq 1\]
	and
	\[\|x-\<x,\xi\>y_\xi\|\leq\|x\|+\|y_\xi\|\leq C+1.\]
	Thus $\psi_\xi$ is 1 on the unit ball and hence \ref{phi2}. holds. 
	
	Now assume $x\in supp(\psi_\xi)$. Then we have
	\[|\<x,\xi\>|\leq 2\]
	and
	\[\|x-\<x,\xi\>y_\xi\|\leq C+2.\]
	Thus
	\[\|x\|\leq C+2+2\|y_\xi\|\leq 3C+2\]
	and we have $\supp(\psi_\xi)\subseteq K:= B_{3C+2}(0)$, so \ref{phi3}. holds.
	
	As $(\xi,x)\mapsto \partial^\alpha\psi_\xi(x)$ is continuous, it is bounded on the compact set $(V\cap S^{d-1})\times K$. Thus we obtain \ref{phi4}. as well.
\end{proof}
With these $\psi_\xi$, we can now make the moral calculation from before rigorous:
\begin{prop}
	\label{WFcancel}
	For any $f\in C^0(\R)$ supported in $\Rp$, $\lambda\in \R$ and $\xi\in V\cup S^{d-1}$, we have
	\[\F(\psi_\xi\sigma\gamma^*f)(\lambda\xi)=0\]
\end{prop}	
\begin{proof}
	Using first the transformation formula and then invariance under $O_\xi$ up to sign of $\<\cdot,\xi\>$, $\psi_\xi$, $\sigma$ (on causal vectors) and $\gamma$, we obtain
	\begin{align*}
		\F(\psi_\xi\sigma\gamma^*f)(\lambda\xi)&=\int\limits_{\R^d} e^{-i\lambda\<x,\xi\>}\psi_\xi(x)\sigma(x)f(\gamma(x))dx\\
		&=\int\limits_{\R^d} e^{-i\lambda\<O_\xi(x),\xi\>}\psi_\xi(O_\xi(x))\sigma(O_\xi(x))f(\gamma(O_\xi(x)))dx\\
		&=\int\limits_{\R^d} e^{-i\lambda\<x,\xi\>}\psi_\xi(x)(-\sigma(x))f(\gamma(x))dx\\
		&=-\F(\psi_\xi\sigma\gamma^*f)(\lambda\xi). \qedhere
	\end{align*}
\end{proof}
We will now need the notations and results of wavefront calculus described in  Section \ref{WFcalc}.
In order to generalize estimates from $\psi_\xi$ to more general cutoff functions, we will have to delve into more technical considerations on wavefront norms.
\begin{prop}
	\label{Hormestimate}For any $u\in L^1_{loc}(\R^d)$, $n\in\N$, closed cones $V, V'\subseteq \R^d\O$ such that $V$ contains an open neighborhood of $V'$ and functions $\psi,\chi\in C_c^\infty(\R^d)$ we have
	\[\|\chi  u\|_{n,V',\psi}\leq C \|\chi\|_{C^{n}} (\|u\|_{n,V,\psi}+\|\psi u\|_{L^1})\]
	with some constant $C$ depending on $V$, $V'$ and $m$.
	
\end{prop}
\begin{proof}
	Let $v:=\psi u$.
	Choose $c$ such that for $k\in V'$ with $|k|=1$, we have $B_c(k)\subset V$. By scaling invariance, we have for arbitrary $k\in V'$ that $B_{c|k|}(k)\subseteq V$.
	We have
	\begin{align*}
		(|k|+1)^n\F (\chi v)(k)&= (|k|+1)^n(\F(\chi)*\F(v))(k)\\
		&=\int\limits_{|\xi|<c|k|} (|k|+1)^n \hat \chi(\xi)\hat v(k-\xi)d\xi+\int\limits_{|\xi|\geq c|k|} (|k|+1)^n \hat\chi(\xi)\hat v(k-\xi)d\xi\\
	\end{align*}
	As for arbitrary $\xi$, we have
	\[1+|k|\leq1+|k-\xi|+|\xi|\leq1+|k-\xi|+|\xi|+|\xi||k-\xi|=(|k-\xi|+1)(|\xi|+1),\]
	we can estimate the first summand by 
	\begin{align*}
		|\int\limits_{|\xi|<c|k|} (|k|+1)^n \hat \chi(\xi)\hat v(k-\xi)d\xi|&\leq| \int\limits_{|\xi|<c|k|} (|\xi|+1)^n \hat \chi(\xi)(|k-\xi|+1)^n\hat v(k-\xi)d\xi|\\
		&\leq\|(|\cdot|+1)^n\hat\chi\|_{L^1}\sup\limits_{k'\in V}((|k'|+1)^n|\hat v(k')|)\\
		&\leq C\|\chi\|_{C^n}\|u\|_{n,V,\psi}
	\end{align*}
	For the second summand, we use that
	\[|\hat v(k-\xi)|\leq \|v\|_{L^1}.\]
	Thus we get 
	\begin{align*}
		|\int\limits_{|\xi|\geq c|k|} (|k|+1)^n \hat\chi(\xi)\hat v(k-\xi)d\xi|&\leq
		|\int\limits_{|\xi|\geq c|k|} (|c^{-1}\xi|+1)^n \hat\chi(\xi)\|v\|_{L^1}d\xi|\\
		&\leq |\int\limits_{|\xi|\geq c|k|} (|c^{-1}\xi|+1)^n \hat\chi(\xi)\|v\|_{L^1}d\xi|\\
		&\leq C'\|(|\cdot|+1)^{n}\hat\chi\|_{L^1}\|v\|_{L^1}\\
		&\leq C\|\chi\|_{C^{n}}\|v\|_{L^1}.\\
	\end{align*}
	Overall, we can conclude
	\[\|\chi  u\|_{n,V',\psi}=\sup\limits_{k\in V'} (|k|+1)^n\F (\chi v)(k)\leq C\|\chi\|_{C^n}(\|u\|_{n,V,\psi}+\|v\|_{L^1}).\qedhere\]
\end{proof}
We want to show that the Riesz distributions are holomorphic in the space of distributions with some suitable wavefront $\Lambda$. As the Riesz distributions are the image of the functions $c_\alpha x^\alpha \mathbbm{1}_{(0,\infty)}$ under the operator $\sigma\gamma^*$, we want to show that this operator is continuous between suitable spaces. The domain to consider should be large enough to contain the above functions (for some values of $\alpha$) but no larger than necessary, so that it is as easy as possible to show continuity of $\sigma\gamma^*$ on this domain. This motivates the following definition: 
\begin{df}
	Let $C^+_\infty$ denote the space of all continuous functions on $\R$ that vanish on $(-\infty,0]$ and are smooth on $(0,\infty)$, endowed with both the seminorms of $C^\infty(\R\O)$ (i.e. the $C^k(K)$-norms for compact $K$ with $0\notin K$) and the seminorms of $C_0(\R)$ (supremum norms on compact sets).
	
	Define for the remainder of this chapter: \label{deflambda}
	\[\Lambda:=\{(x,v)\in \dot T^*_0(\R^d)\mid\text {$x$ lightlike and $v$ parallel to $d\gamma(x)$ or $x=0$ and $v$ not spacelike}\}.\]
\end{df}
We can now show the technical centerpiece of this section:
\begin{thm}
	The map $\sigma\gamma^*$ defined by $\sigma\gamma^*(f)=\sigma f\circ\gamma$ is a continuous map from $C^+_\infty$ to $\D'_\Lambda(\R^d)$. 
\end{thm}
\begin{proof}
	As 
	\[\sup\limits_{x\in K}|\sigma f(\gamma(x))|=\sup\limits_{y\in \gamma(K)} |f(y)|\]
	and $\gamma(K)$ is compact for $K$ compact, $\sigma\gamma^*$ is continuous as a map into $L^\infty_{loc}(\R^d)$. This embeds continuously in the space of distributions, so $\sigma\gamma^*$ maps continuously into $\D'(\R^d)$ (with respect to the strong topology). Thus we only need to check that $\sigma\gamma^*$ is bounded with respect to the wavefront-set related norms
	$\|u\|_{n,V,\psi}$
	for $\supp(\psi)\times V\cap \Lambda=\emptyset$. \\	
	\textbf{Wavefront norms away from 0}:\\	
	First assume that $0\notin\supp (\phi)$. Then we may as well exclude $0$ from the domain of $\sigma f\circ \gamma$ when calculating the corresponding norms. But on $\R^d\O$, $\gamma$ is a submersion, which allows us to use wavefront calculus: As $d\gamma$ is non-vanishing, we find that
	\[\gamma|_{\R^d\O}^*(T_0(\R))=\{(x,v)\in \dot T^*_0(\R^d\O)\mid\gamma(x)=0, v=cd\gamma(x) \text{ for $c\in\R\O$} \}\]
	does not intersect the zero section, so $\gamma|_{\R^d\O}^*$ maps continuously from $\D'_{T^*_0(\R)}(\R)$ into $\D'_{\gamma|_{\R^d\O}^*(T_0(\R))}(\R^d\O)=\D'_{\Lambda|_{\R^d\O}}(\R^d\O)$. 
	
	The seminorms of $\D'(\R)$ are bounded by those of $C^0(\R)$. The remaining seminorms of $\D'_{T_0(\R)}(\R)$, i.e. the seminorms $\|\cdot\|_{n,V,\psi}$ for $0\notin \supp(\psi)$ are also seminorms for $\D'_\emptyset(\R\O)= C^\infty(\R\O)$. Thus all seminorms of $\D'_{T^*_0(\R)}(\R)$ are bounded by those of $C^+_\infty$, so we can conclude that $\gamma|_{\R^d\O}^*$ also maps $C^+_\infty$ to $\D'_{\Lambda|_{\R^d\O}}(\R^d\O)$ continuously. 
	
	Choose a partition of unity $(\chi_1,\chi_2,\chi_3)$ subordinate to the open cover \[\left(\{x\in \R^d\mid x_0>0\},\{x\in \R^d\mid x_0<0\}, \R^d\backslash\gamma^{-1}(\Rp)\right)\]
	of $\R^d\O$. Then $\tilde\sigma:=\chi_1-\chi_2$ is a smooth function on $\R^d\O$ that coincides with $\sigma$ on $\gamma^{-1}(\Rp)\O$. Thus on $C^+_\infty$, we have \[\sigma\gamma|_{\R^d\O}^*=\tilde\sigma\gamma|_{\R^d\O}^*\]
	and this maps continuously from $C^+_\infty$ to $\D'_{\Lambda|_{\R^d\O}}(\R^d\O)$, since multiplication with smooth functions is continuous on the latter space. We can conclude that all seminorms $\|\sigma\gamma^*(f)\|_{n,V,\psi}$ 
	with $0\notin \supp(\psi)$ and $\supp(\psi)\times V\cap \Lambda=\emptyset$ are bounded by the $C^+_\infty$ seminorms of $f$.\\
	\textbf{Wavefront norms at 0}:\\
	It now remains to bound the norms with $0\in \supp(\psi)$. The condition $\supp(\psi)\times V\cap \Lambda=\emptyset$ is then equivalent to $V$ containing only spacelike vectors. Thus we now fix $\psi\in C_c^\infty(\R^d)$ and a closed cone $V$ of spacelike vectors in $\R^d$. For spacelike vectors Proposition \ref{WFcancel} almost gives us the estimates we want, only with $\psi_\xi$ instead of $\psi$. Thus what we need to do now is replace these special $\xi$-dependent test functions with arbitrary ones. We will first replace the $\psi_\xi$ with a fixed cutoff $\phi$ (indepedent of $\xi$), which in turn will be replaced with the arbitrary cutoff $\psi$.
	
	We will successively have to make our cones smaller in the following estimate, so we need to choose slightly bigger cones. Let $V'$ be a closed spacelike cone whose interior contains $V$ and let $V''$ be a closed spacelike cone whose interior contains $V'$.
	Use the notation of Lemma \ref{phixi} (i.e. construct $\psi_\xi$) for $V''$ instead of $V$.
	Choose $\phi\in C_c^\infty(\R^d)$ that is one on the unit ball  $B_1(0)$ and supported in $K$ (the compact set from \ref{phixi}). Let $\chi$ be a function in $C_c^\infty(\R^d)$ that is 1 on $(K\cup\supp(\psi))\backslash B_1(0)$ and vanishes around 0.
	
	For any $f\in C^+_\infty$ we have:
	\begin{align*}
		&\ \|\sigma\gamma^*f\|_{n,V',\phi}\\
		&=\sup\limits_{\xi \in V',\lambda\in\R_+}(|\lambda|+1)^n\F(\phi \sigma\gamma^*f)(\lambda\xi)\\
		&=\sup\limits_{\xi \in V',\lambda\in\R_+}(|\lambda|+1)^n\F((\phi-\psi_\xi)\sigma\gamma^*f)(\lambda\xi)\\
		&\leq\sup\limits_{\xi,k \in V',\lambda\in\R_+}(|\lambda|+1)^n\F((\phi-\psi_\xi)\sigma\gamma^*f)(\lambda k)\\
	\end{align*}
	where in the second step we have subtracted $0$ by Proposition \ref{WFcancel}. Using that \[\chi(\phi-\psi_\xi)=(\phi-\psi_\xi)\] and then Proposition \ref{Hormestimate}, we can then estimate (with constants independent of $f$):
	\begin{align*}
		&\ \|\sigma\gamma^*f\|_{n,V',\phi}\\
		&\leq\sup\limits_{\xi,k \in V',\lambda\in\R_+}(|\lambda|+1)^n\F(\chi(\phi-\psi_\xi)\sigma\gamma^*f)(\lambda k)\\
		&=\sup\limits_{\xi\in V'}\|(\phi-\psi_\xi)(\sigma\gamma^*f)\|_{n,V',\chi}\\
		&\leq C\sup\limits_{\xi\in V'} \|\phi-\psi_\xi\|_{C^{n}}(\|\sigma\gamma^*f\|_{n,V'',\chi}+\|\chi\sigma\gamma^*f\|_{L^1})\\
		&\leq C(\|\phi\|_{C^{n}}+\sup\limits_{\xi\in V'}\|\psi_\xi\|_{C^{n}})(\|\sigma\gamma^*f\|_{n,V'',\chi}+\|\chi\sigma\gamma^*f\|_{L^1})\\
		&= C'(\|\sigma\gamma^*f\|_{n,V'',\chi}+\|\chi\sigma\gamma^*f\|_{L^1})\\
	\end{align*}
	As $0\notin\supp(\chi)$, the first summand is one of the seminorms we have already estimated. 
	The second summand can be estimated by
	\[\|\chi\sigma\gamma^*f\|_{L^1}\leq\|\chi\|_{L^1}\|\sigma\gamma^*f\|_{L^\infty(\supp(\chi))}\leq \|\chi\|_{L^1}\|f\|_{L^\infty(\gamma(\supp(\chi)))}\]
	and is thus bounded by $C^+_\infty$ seminorms of $f$.
	
	Define $\chi^c:=\mathbbm{1}_{B_1(0)}(1-\chi)$. We have
	\[\psi=\chi^c\psi+\chi\psi\]
	and
	\[\phi\chi^c\psi=\chi^c\psi\]
	Using Proposition \ref{Hormestimate} again, we can estimate:
	\begin{align*}
		\|\sigma\gamma^*f\|_{n,V,\psi}&\leq\|\chi^c\psi\sigma\gamma^*f\|_{n,V,\phi}+\|\sigma\gamma^*f\|_{n,V,\chi\psi}\\
		&\leq C(\|\sigma\gamma^*f\|_{n,V',\phi}+\|\phi\sigma\gamma^*f)\|_{L^1})+\|\sigma\gamma^*f\|_{n,V,\chi\psi},\\
	\end{align*}
	The first summand we already estimated, for the second we obtain as for the second summand above
	\[\|\phi(\sigma\gamma^*f)\|_{L^1}\leq\|\phi\|_{L^1}\|f\|_{L^\infty(\gamma(\supp(\phi)))}.\]
	Overall, we obtain
	\[\|\sigma\gamma^*f\|_{n,V,\psi}\leq C(\|\sigma\gamma^*f\|_{n,V'',\chi}+\|\sigma\gamma^*f\|_{n,V,\chi\psi}+\|f\|_{L^\infty(\gamma(\supp(\chi)))}+\|f\|_{L^\infty(\gamma(\supp(\phi)))}),\]
	for some constant $C$ independent of $f$. The first two summands are wavefront norms away from zero, which are bounded by $C^+_\infty$-seminorms by the first part of this proof. The other two summands are already $C^+_\infty$-seminorms. Thus we have bounded all seminorms of $\sigma\gamma^*f$ in $\D'_\Lambda(\R^d)$ in terms of the seminorms of $f$ in $C^+_\infty$. We can conclude that $\sigma\gamma^*$ is a continuous map between those spaces.	 
\end{proof}
To show the properties we want about the Riesz distributions, we now need to study those functions in $C^+_\infty$ that get mapped to Riesz distributions by $\sigma\gamma^*$.
\begin{defprop}\label{Xa}
	The family of functions
	\[X_\alpha(x):=\mathbbm{1}_{(0,\infty)}(x)x^\alpha\]
	is holomorphic on $\Re(\alpha)>0$ as a map into $C^+_\infty$ (any undefined powers of $x$ are set to $0$ but are irrelevant anyways due to the characteristic function).
\end{defprop}
\begin{proof}
	The obvious candidate for a derivative is given by the pointwise derivative:
	\[\partial_\alpha X_\alpha:=\mathbbm{1}_{(0,\infty)}(x)ln(x)x^\alpha.\]
	This is still in $C^+_\infty$, as $ln(x)$ grows slower than $x^{-\alpha}$ for $x\rightarrow 0$. We have for non-zero $h\in \C$ such that $|h|<\Re(\alpha)$:
	\begin{align*}
		E_h(x)&:=\frac{X_{\alpha+h}(x)-X_\alpha(x)}{h}-\partial_\alpha X_\alpha(x)\\
		&=\int\limits_0^1\partial_\alpha X_{\alpha+sh}(x)-\partial_\alpha X_\alpha(x)ds\\
		&=\partial_\alpha X_\alpha(x)\int\limits_0^1(x^{sh}-1)ds.
	\end{align*}
	We need to show that $E_h$ converges to $0$ in all $C^+_\infty$-seminorms for $h\rightarrow 0$. Let $\epsilon\in(0,1)$ be arbitrary.
	Let $K\subseteq\R$ be compact and $x\in K$ be positive (everything is $0$ for $x\leq0$). If $x\in(0,\epsilon)$, we can estimate $|x^{sh}-1|<1$ to obtain
	\[|\partial_\alpha X_\alpha(x)\int\limits_0^1(x^{sh}-1)ds|\leq \sup\limits_{x'\in(0,\epsilon)}|\partial_\alpha X_\alpha(x')|.\]
	For $x\geq \epsilon$, we get
	\[|\partial_\alpha X_\alpha(x)\int\limits_0^1(x^{sh}-1)ds|\leq\sup\limits_{x'\in K} |\partial_\alpha X_\alpha(x')|\sup\limits_{y\in [\epsilon,K]}|y^h-1|\leq C\max\{\max(K)^h-1,1-\epsilon^h\}.\]
	Thus we obtain overall
	\begin{align*}
		\limsup\limits_{h\rightarrow 0}\|E_h\|_{L^\infty(K)}&\leq\limsup\limits_{\epsilon\rightarrow 0}\limsup\limits_{h\rightarrow 0}\max\{\sup\limits_{x\in(0,\epsilon)}|\partial_\alpha X_\alpha(x)|,C(\max(K)^h-1),C(1-\epsilon^h)\}\\
		&=\limsup\limits_{\epsilon\rightarrow 0}\max\{ \sup\limits_{x\in(0,\epsilon)}|\partial_\alpha X_\alpha(x)|,0,0\}\\
		&=\limsup\limits_{x\rightarrow 0}|\partial_\alpha X_\alpha(x)|\\
		&=0.
	\end{align*}
	To show convegence of the $C^\infty((0,\infty))$-norms of $E_h$, it suffices to show that 
	\[\int\limits_0^1(x^{sh}-1)ds\]
	converges to zero in these seminorms, as multiplication with a smooth function is bounded in all $C^k$-norms. Thus we compute for $|h|<1$, $k\geq 1$, $K\subseteq (0,\infty)$ compact and $x\in K$:
	\begin{align*}
		|\partial_x^k\int\limits_0^1(x^{sh}-1)ds|&=|\int\limits_0^1 \prod\limits_{j=0}^{k-1}(sh-j)x^{sh-k}ds|\\
		&\leq |h|\prod\limits_{j=1}^{k-1}(j+|h|) \int\limits_0^1\max\{|x|^{-1-k},|x|^{1-k}\}ds\\
		&\leq k! \max\{\min(K)^{-1-k},\min(K)^{1-k}\} |h|.
	\end{align*}
	We can conclude that $E_h$ converges to 0 in all $C^k(K)$-seminorm. As we have already shown uniform convergence on $\R$, this means that the differential quotient converges to $\partial_\alpha X_\alpha$ in all seminorms of $C^+_\infty$, so $X_\alpha$ is holomorphic in that space.
\end{proof}
The main result of this section is now easy to prove, using the above and the extension technique also employed in \cite{BGP}:
\begin{prop}
	\label{WFhol}
	The family $R(\alpha)=R_+(\alpha)-R_-(\alpha)$ is holomorphic (as a function of $\alpha$) in $\D'_\Lambda$ (for $\Lambda$ as in definition \ref{deflambda}).
\end{prop}
\begin{proof}
	For $\Re(\alpha)>d$, we have
	\[R(\alpha)(x)=c_{\alpha}\gamma(x)^{\frac{\alpha-d}{2}}\sigma(x)\mathbbm{1}_{\gamma(x)>0}=c_\alpha\sigma\gamma^*(X_{\frac{\alpha-d}{2}})(x).\]
	$X_\frac{\alpha-d}{2}$ is holomorphic in $C^+_\infty$ for $\Re(\alpha)>d$ and $\sigma\gamma^*$  maps this continuously into $\D'_\Lambda$, so $R(\alpha)$ is holomorphic in $\D'_\Lambda$ for $\Re(\alpha)>d$. Let $k\in \N$ be arbitrary. As $\square$ maps $\D'_\Lambda$ to itself continuously, we can conclude that
	\[R(\alpha-2k)=\square^kR(\alpha)\]
	is holomorphic in $\D'_\Lambda$ for $\Re(\alpha)>d$. Thus $R(\alpha)$ is holomorphic in $\D'_\Lambda$ for $\Re(\alpha)>d-2k$. As $k$ was arbitrary, it is holomorphic on all of $\C$.
\end{proof}
\section{Integrating along a timeline}
\label{timeint2}
We now seek to integrate Riesz distributions and Green's kernels along a timelike curve $w\colon I\rightarrow M$, as outlined in the beginning of this section. Evaluating the Green's kernel at some $w(t)$ is not well defined, so we instead consider the pullback under $w$.  The ill-defined expression $\intR G_x(w(t)) f(\frac{t}{s})dt$ can then be replaced with $w^*(G_x)[f(\tfrac{\cdot}{s})]$ (in the case $E=M\times \C$; for vector bundles, we need to take a bundle coordinate).\newline
We start by fixing some notation:
\begin{df}
	\label{evenodd}
	For a function $f$ on $\R$, let 
	\[(f)_{odd}(t):=\frac{1}{2}(f(t)-f(-t))\]
	and
	\[(f)_{even}(t):=\frac{1}{2}(f(t)+f(-t))\]
	denote the odd and even part of $f$. If $f$ is defined only on a subset of $\R$, extend it by $0$.
\end{df}
\begin{gn}
	\label{dfw}
	For the remainder of this section, let $U\subseteq M$ be open and convex, $I\subseteq\R$ an open interval containing $0$ and $x\in U$. Furthermore, let $w\colon I\rightarrow U$ be a future oriented smooth curve with $w(0)=x$.
\end{gn}
In several of the following theorems, there will be a "more precisely" part that keeps track of estimates for the remainder terms. These will be needed later on to provide some sort of uniformity for global considerations, but for understanding what is going on, the reader may focus on the less precise first statements.

\begin{df}
	\label{dfnu}
	Define $\nu_w\colon I\rightarrow \C$ via
	\[\nu_w(t):=\frac{\Gamma^U(w(t))}{t^2},\]
	extended continuously to $0$.
\end{df}
In some sense $\nu_w$ measures the failure of $w$ to be a geodesic. Later on, this $\nu_w$ will be a problem and we will have to restrict to the case where $w$ is a geodesic. For now, however, we include this for the sake of generality.
\begin{lemma}
	$\nu_w$ is a well-defined, smooth, strictly positive function and 
	\[\nu_w(0)=\gamma(w'(0)).\]
	If $w$ is a timelike unit speed geodesic, we have
	\[\nu_w=1\]
	constantly.
\end{lemma}
\begin{rem}
	Unfortunately, the presence of $\nu_w$ will pose problems later on, so we will eventually have to restrict to the case where $w$ is a unit speed geodesic.
\end{rem}
\begin{proof}
	Let $\tilde w:=\exp_x^{-1}\circ w$. We have
	\[\Gamma^U(w(t))=\gamma(\tilde w(t))=\gamma(\tilde w(0))+(\gamma\circ\tilde w)'(0)t+E(t)t^2,\]
	where $E$ is a smooth function with 
	\[E(0)=\frac{1}{2}(\gamma\circ\tilde w)''(0).\]
	As $\tilde w(0)=0$, we have
	\[\gamma(\tilde w(0))=-g(\tilde w,\tilde w)(0)=0,\]
	\[(\gamma\circ\tilde w)'(0)=-2g(\tilde w,\tilde w')(0)=0\]
	and
	\[(\gamma\circ\tilde w)''(0)=-2(g(\tilde w',\tilde w')+ g(\tilde w,\tilde w''))(0)=-2g(w'(0),w'(0))=2\gamma(w'(0)),\]
	where we used that $\tilde w'(0)$ coincides with $w'(0)$ (when canonically identifying $T_0T_xM$ and $T_xM$), as $d\exp_x$ at $x$ is the identity. We thus conclude that
	\[\frac{\Gamma^U(w(t))}{t^2}=E(t)\]
	extends smoothly to $0$, where it takes the value
	\[E(0)=\gamma(w'(0))\]
	(which is strictly positive, as $w'(0)$ must be timelike). For $t\neq 0$, $w(t)$ is in the interior of $J_+(x)$ or $J_-(x)$ and thus $\tilde w(t)$ is in the interior of $J_+(0)$ or $J_-(0)$ (i.e. timelike), where $\gamma$ is strictly positive. Thus $\gamma(\tilde w(t))$ and hence $\nu_w(t)$ are strictly positive.
	
	If $w$ is a timelike geodesic, we have $\tilde w (t)=tw'(0)$, as the exponential map maps geodesics through zero to geodesics through $x$ with the same initial derivative. Thus we have
	\[\nu_w(t)=\frac{\gamma(t w'(0))}{t^2}=\gamma( w'(0))\]
	if $w$ is additionally unit speed, this is $1$.
\end{proof}
We shall use a variant of the Mellin transform where some poles are removed:
\begin{df}
	\label{dfMprime}
	In the following we abbreviate:
	\[\M'(g):=\frac{\M(g)}{\Gamma(\frac{\cdot+1}{2})},\]
	where $\M$ denotes the Mellin transform (see subsection \ref{Mellin}).
\end{df}
For odd smooth $g$, the Gamma function in the denominator will cancel the poles of the Mellin transform, allowing us to evaluate the modified Mellin transform everywhere. 
We want to insert the Hadamard expansion into $w^*G_x[f(\tfrac{\cdot}{s})]$. We start by evaluating the pull-back of a single summand $V^kR(2k+2)$. The $W$ in the following should thus be thought of as a Hadamard coefficient (or a component thereof).
\begin{prop}
	\label{WRint}
	For any $W\in C^\infty(U)$ and $\alpha\in \C$, $w^*(WR^U(\alpha,x))$ is well-defined and we have for any $g\in C_c^\infty(I)$:
	\[w^*(WR^U(\alpha,x))[g]=\frac{2^{2-\alpha}\pi^\frac{2-d}{2}}{\Gamma(\frac{\alpha}{2})} \M'\left(\left(\nu_w^\frac{\alpha-d}{2}(W\circ w)g\right)_{odd}\right)({\alpha-d+1}).\]
\end{prop}
\begin{proof}
	Let again  $\tilde w:=exp_x^{-1}\circ w$. As $w^*=\tilde w^*\circ exp_x^*$, it suffices to show that the pull-back of $exp_x^{*}R^U(\alpha,x)=R(\alpha)$ under $\tilde w$ is well defined with the correct value. By Theorem \ref{WFhol}, $R(\alpha)$ is holomorphic in $\D'_\Lambda(\R^d)$. Recall that $\Lambda$ only contains causal vectors at lightlike or zero basepoints. 
	
	For $t\neq 0$, $w(t)$ can be reached from $x$ by a timelike geodesic by Proposition \ref{exJ}. Thus $\tilde w(t)$ is not lightlike, so $\Lambda|_{\tilde w(t)}$ is empty. This means if  $v\in \Lambda|_{\tilde w(t)}$, then $t=0$ and $v$ is causal. For any non-zero vector $\lambda\partial_t$ in $T_0(\R)$, we have
	\[\<(d\tilde w_0)^*v,\lambda\partial_t\>=\lambda\<v,\tilde w'(0)\>.\]
	The scalar product is non-zero, as $v$ is causal and $\tilde w'(0)$ is timelike and thus, by \ref{orthotime},  they cannot be orthogonal. Thus $\tilde w^*\Lambda$ does not intersect the zero section, whence $\tilde w^*$ maps $\D'_\Lambda(T_x(M))$ continuously into $\D'(I)$.
		
	Assume $\Re(\alpha)>d$. Then $R^U(\alpha,x)$ is continuous. We have  
	\[\Gamma^U(w(t))=\nu_w(t)t^2.\]	
	Thus we can compute
	\begin{align*}
		&w^*(WR^U(\alpha,x))[g]\\
		&=\int\limits_I W(w(t))c_{\alpha}\sign(t)\Gamma^U(w(t))^\frac{\alpha-d}{2}g(t)dt\\
		&=c_{\alpha}\int\limits_I \sign(t)W(w(t))\nu_w(t)^\frac{\alpha-d}{2}|t|^{\alpha-d}g(t)dt\\
		&=2c_\alpha\intP\left(\nu_w^\frac{\alpha-d}{2}(W\circ w)g\right)_{odd}(t) t^{\alpha-d}dt\\
		&=2c_\alpha \M\left(\left(\nu_w^\frac{\alpha-d}{2}(W\circ w)g\right)_{odd}\right)({\alpha-d+1})\\
		&=\frac{2^{2-\alpha}\pi^\frac{2-d}{2}}{\Gamma(\frac{\alpha}{2})} \M'\left(\left(\nu_w^\frac{\alpha-d}{2}(W\circ w)g\right)_{odd}\right)({\alpha-d+1}).
	\end{align*}
	We show that this is true for arbitrary $\alpha$ by showing that both sides are holomorphic.

	Multiplication with a smooth function is a continuous map of $\D'_\Lambda(T_x M)$ to itself. Evaluation of a distribution is also continuous. As continuous linear maps preserve holomorphicity and the family of Riesz distributions is holomorphic in $\D'_\Lambda(T_xM)$, we can conclude that $\tilde w^*((W\circ exp_x) R(\alpha))[g]=w^*(WR^U(\alpha,x))[g]$ is holomorphic.
	
	For the right hand side, consider the family $X_\alpha$ as in Proposition \ref{Xa}. This is holomorphic on $\Re(\alpha)>0$ both in $D'(\R)$ and  $C^\infty(\R\O)$, as $C^+_\infty$ embeds continuously into both spaces.
	As 
	\[X_{\alpha-1}=\frac{1}{\alpha}X_\alpha',\]
	it extends meromorphically to all of $\C$. By meromorphic continuation,
	\[\M(f)(\alpha)=X_{\alpha-1}[f]\]
	holds on all of $\C$. We thus have
	\begin{align*}
		&\M\left(\left(\nu_w^\frac{\alpha-d}{2}(W\circ w)g\right)_{odd}\right)({\alpha-d+1})\\
		&=X_{\alpha-d}\left[\left((X_\frac{\alpha-d}{2}\circ\nu_w)(W\circ w)g\right)_{odd}\right].
	\end{align*}
	As $\nu_w$ is strictly positive, the argument is meromorphic in $C^\infty_c(\R)$. As evaluation is hypocontinuous, the result is meromorphic.
	
	This implies that both sides of the claimed result are meromorphic in $\alpha$, so equality holds for all $\alpha\in\C$ (with the fraction of meromorphic functions being extended continuously to $\alpha$ where both numerator and denominator have a pole).
\end{proof}
In order to retain more information after distributional evaluation, we equip the function we evaluate at with an additional parameter $s$.
\begin{gn}
	\label{dff}
	Let $f\in C_c^\infty(\R)$ be an odd function. For $s>0$ and $t\in \R$, define \[f_s(t):=f\left(\frac{t}{s}\right).\] 
	Let $I_f$ be a closed finite interval whose interior contains  $0$ and $\supp(f)$.
\end{gn}
This $f$ corresponds roughly to (the derivative of the Fourier transform of) the cutoff function in the spectral action, $s$ corresponding to the cutoff parameter there. The $f$ here being odd corresponds to that in the spectral action being even.
 The assumption that $f$ is odd is not crucial for the following considerations, but it makes things somewhat simpler and adding an even part would not yield anything helpful for our purposes (see remark \ref{whatf}).

Testing a distribution against $f_s$ for small $s$ corresponds to investigating its behaviour around $0$. Multiplying with $f_s$ and taking the Mellin transform does something similar. This is investigated in the following two lemmas.
\begin{lemma}
	\label{IntRem}
	If $h$ is a $C^{k+1}$ function whose first $k$ derivatives at $0$ vanish and $\alpha\in \C$ with $\Re(\alpha)\geq -k$, we have for $s\in (0,1)$
	\[|\M(hf_s)(\alpha)|\leq C\|h\|_{C^{k}(I_f)}s^{k+\Re(\alpha)}\]
	and
	\[\intR h(t)f_s(t)dt\leq C\|h\|_{C^{k}(I_f)}s^{k+1},\]
	where $C$ is a constant depending on $\alpha$, $k$ and $f$.
\end{lemma}
\begin{proof}
	By Taylor's theorem (with vanishing Taylor series), we have
	\[h(t)= t^kE(t)\]
	for some continuous function $E$. $|E|$ is bounded on $I_f$ by $C\|h\|_{C^{k}(I_f)}$. We then have for $s\in (0,1)$:
	\begin{align*}
		| M (hf_s)(\alpha)|&=|M (Ef_s)(\alpha+k)|\\
		&\leq\intP \|h\|_{C^{k}(I_f)}|t^{k+\alpha-1}f\left(\frac{t}{s}\right)|dt\\
		&=|s^{k+\alpha}|\intP\frac{1}{s} \|h\|_{C^{k}(I_f)}|\left(\frac{t}{s}\right)^{k+\alpha-1}f\left(\frac{t}{s}\right)|dt\\
		&=\|h\|_{C^{k}(I_f)}|s^{k+\alpha}|\intP |t^{k+\alpha-1}f(t)|dt\\
		&=: C\|h\|_{C^{k}(I_f)}s^{k+\Re(\alpha)}.
	\end{align*}
	The second part of the claim follows either by doing an analogous calculation or by observing that
	\[\intR h(t)f_s(t)dt=\intP(h(t)-h(-t))f_s(t)dt= M ((h-h(-\cdot))f_s)(1).\qedhere\]
\end{proof}
\begin{lemma}
	\label{MsExp}
	Let $h$ be a smooth function defined on $I$ and $\alpha\in \C$. Then
	\[\M'( (hf_s)_{odd})(\alpha)\stackrel{s\rightarrow 0}{\sim}\insum{k}\frac{k!}{(2k)!}\binom{\frac{\alpha+2k-1}{2}}{k}h^{(2k)}(0) \M'(f)(\alpha+2k)s^{\alpha+2k}.\]
	More precisely, we have for $s$ small enough such that $f_s\subseteq C_c^{\infty}(I')$ for $I'\subseteq I$ closed and $N\in \N$:
	\[\left|\M'( (hf_s)_{odd})(\alpha)-\sum\limits_{k=0}^N\frac{k!}{(2k)!}\binom{\frac{\alpha+2k-1}{2}}{k}h^{(2k)}(0) \M'(f)(\alpha+2k)s^{\alpha+2k}\right|\leq C\|h\|_{C^{2N}(I')}s^{2N+\Re(\alpha)},\]
	for some constant $C$ depending on $N$, $\alpha$, $f$ and $I'$.
\end{lemma}
\begin{proof}
	As all even derivatives of $(hf_s)_{odd}$ at $0$ vanish, its Mellin transform  has poles only at odd negative integers. $\Gamma(\frac{\cdot+1}{2})$ also has poles there, so the quotient is defined everywhere.
	
	Write $t^k$ for the function $t\mapsto t^k$. For $k\in \N$ and $\Re(\alpha)\geq 0$, we have 
	\begin{align*}
		\M' (t^k f_s)(\alpha)&=\frac{1}{\Gamma(\frac{\alpha+1}{2})}\intP t^{\alpha+k-1}f\left(\frac{t}{s}\right)dt\\
		&=\frac{1}{\Gamma(\frac{\alpha+1}{2})}s^{\alpha+k}\intP\frac{1}{s}\left( \frac{t}{s}\right)^{\alpha+k-1}f\left(\frac{t}{s}\right)dt\\
		&=\frac{1}{\Gamma(\frac{\alpha+1}{2})}s^{\alpha+k}\intP t^{\alpha+k-1}f(t)dt\\
		&= \frac{1}{\Gamma(\frac{\alpha+1}{2})}\M(f)(\alpha+k)s^{\alpha+k}.\\
	\end{align*}
	By analytic continuation, this holds for all $\alpha\in \C$. Now fix $\alpha\in \C$. As $f$ is odd, we have $(hf_s)_{odd}=h_{even}f_s$.
	We write $(h)_{even}$ in a Taylor expansion
	\[(h)_{even}(t)=\sum\limits_{k=0}^N\frac{1}{(2k)!}h^{(2k)}(0)t^{2k}+ E(t),\]
	where the first $2N$ derivatives of $E$ at $0$ vanish and 
	\[\|E\|_{C^k(I)}\leq C\|h\|_{C^k(I)}\]
	for any $k\geq 2N$.
	We obtain for $N\in \N$ such that $2N\geq-\Re(\alpha)$ :
	\begin{align*}
		 &\M'((hf_s))(\alpha)\\
		 &=\sum\limits_{k=0}^N\frac{1}{(2k)!}h^{(2k)}(0)  \left(\frac{\Gamma(\frac{\cdot+2k+1}{2})\M (t^{2k}f_s)}{\Gamma(\frac{\cdot+2k+1}{2})\Gamma(\frac{\cdot+1}{2})}\right)(\alpha)
		 + \M' ( E(t)f_s)(\alpha)\\
		 &=\sum\limits_{k=0}^N\frac{k!}{(2k)!}h^{(2k)}(0)  \left(\frac{\Gamma(\frac{\cdot+1}{2}+k)}{k!\Gamma(\frac{\cdot+1}{2})}\frac{\M (f_s)(\cdot+2k)}{\Gamma(\frac{\cdot+2k+1}{2})}\right)(\alpha)
		 + \M' (E(t)f_s)(\alpha)\\
		  &=\sum\limits_{k=0}^N\frac{k!}{(2k)!}\binom{\frac{\alpha+2k-1}{2}}{k}h^{(2k)}(0) \M'(f)(\alpha+2k)s^{\alpha+2k}+ \M' (E(t)f_s)(\alpha)\\
	\end{align*}
	Using the Lemma \ref{IntRem} for the remainder term, we get
	\[\frac{1}{\Gamma(\frac{\alpha+1}{2})} \M (E(t)f_s)(\alpha)\leq C\|h\|_{C^{2N}(I)}s^{2N+\Re(\alpha)}\]
	for some $C\in \R$, as desired.
\end{proof}

The coefficients appearing in the next theorem will occur several times throughout this thesis, so we shall give them a name for later reference.
\begin{df}
	\label{dfa}
	For $k,n\in\N$, define
	\[a(k,n):=a(k,n,f,d):=\frac{\pi^\frac{2-d}{2}n!}{4^kk!(2n)!} \binom{k+n+1-\frac{d}{2}}{n} \M'(f)({2k+2n+3-d}).\]
\end{df}
Putting together what we have done so far, we obtain an abstract preliminary version of this sections main result:

\begin{thm}
	\label{Wexp}
	Assume that $K$ is a distribution on $U$ that vanishes outside of $J(x)$ and there are $W_k\in C^\infty(U)$ such that
	\[K\sim \insum{k}W_kR^U(2k+2,x).\]
	Then $w^*(K)$ is well-defined and we have
	\begin{align*}
		w^*(K)[f_s]\stackrel{s\rightarrow 0}{\sim}&\insum{k}\insum{n}a(k,n)\left(\nu_w^{k-\tfrac{d}{2}+1}W_k\circ w\right)^{(2n)}(0)s^{2k+2n+3-d}
	\end{align*}
for $a(k,n)$ as in Definition \ref{dfa}.

	More precisely, if $s$ is small enough such that $f_s\in C_c^\infty (I')$ for $I'\subseteq I$ closed and if we fix $m\in \N$, we have for $N',N\in \N$ large enough the following estimate for the remainder terms:
	\begin{align*}
		&\left|w^*(K)[f_s]-\sum\limits_{k=0}^N\sum\limits_{n=0}^{N'}a(k,n)\left(\nu_w^{k-\tfrac{d}{2}+1}W_k\circ w\right)^{(2n)}(0)s^{2k+2n+3-d}\right|\\
		&\leq C\left(N'\max\limits_{k\leq N'}\Big\|\nu_w^{k-\tfrac{d}{2}+1}W_k\circ w\Big\|_{C^{N'}(I')}+\Big\|\Big(K- \sum\limits_{k=0}^N W_kR^U(2k+2,x)\Big)\circ w\Big\|_{C^{m}(I')}\right)s^m
	\end{align*}
\end{thm}
\begin{rem}
	We mainly want to use this with $G_x$ and $V_x^{k,U}$ instead of $K$ and $W_k$. However, we keep this general so that it could also be applied to other expansions, like, for example, that for $G^{\pm m}_x$.
\end{rem}
\begin{proof}
	Without loss of generality, we may assume that $I_f\subseteq I$, otherwise rescale $f$ until it is supported in $I$. Let $s\in (0,1)$. Let $m\in \N$ be arbitrary and let $N\in\N$ be large enough such that
	\[F:=K-\sum\limits_{k=0}^N(W_kR^U(2k+2,x))\in C^m(U).\]
	As both the sum and the remainder have well-defined pull-backs under $w$, the same is true for $K$. As both $K$ and all $R^U(2k+2,x)$ vanish outside of $J(x)$, the same is true for $F$ and thus also for its derivatives. As $x\in \partial J(x)$ by Proposition \ref{xinbound}, this means that all derivatives of order up to $m$ of $F$ vanish at $x$. Therefore the first $m$ derivatives of $w^*F=F\circ w$ at 0 vanish.
	By Lemma \ref{IntRem}, we can conclude that
	\[|w^*(F)[f_s]|=\Big|\intR F\circ w(t)f_s(t)dt\Big|\leq C\|F\circ w\|_{C^{m}(I_f)}s^m.\]
	For each of the summands, we can apply Proposition \ref{WRint} and Lemma \ref{MsExp} to obtain for $N'\geq \frac{m+d}{2}$ and $k<N$:
	\begin{align*}
		&w^*(W_kR^U(2k+2,x))[f_s]\\
		&=\frac{\pi^\frac{2-d}{2}}{4^kk!} \M'\left(\left((\nu_w^{k-\tfrac{d}{2}+1}W_k\circ w)f_s\right)_{odd}\right)({2k+3-d})\\
		&=\frac{\pi^\frac{2-d}{2}}{4^kk!} \sum\limits_{n=0}^{N'}\frac{n!}{(2n)!}\binom{k+n+1-\frac{d}{2}}{n}\left(\nu_w^{k-\tfrac{d}{2}+1}W_k\circ w\right)^{(2n)}(0)\M'(f)({2k+2n+3-d})s^{2k+2n+3-d}+F_k\\
	\end{align*}
with 
\[|F_k|\leq C\|\nu_w^{k-\tfrac{d}{2}+1}W_k\circ w\|_{C^{2N'}(I_f)}s^{2N'+2k+3-d}\leq C\|\nu_w^{k-\tfrac{d}{2}+1}W_k\circ w\|_{C^{2N'}(I_f)}s^{m}.\]
	We then have
	\begin{align*}
		w^*(K)[f_s]&=\sum\limits_{k=0}^Nw^*(W_kR^U(2k+2,x))[f_s]+w^*(F)[f_s]\\
		&=\sum\limits_{k=0}^N\frac{\pi^\frac{2-d}{2}}{4^kk!} \sum\limits_{n=0}^{N'}\frac{n!}{(2n)!}\binom{k+n+1-\frac{d}{2}}{n}\left(\nu_w^{k-\tfrac{d}{2}+1}W_k\circ w\right)^{(2n)}(0)\\&\cdot \M'(f)({2k+2n+3-d})s^{2k+2n+3-d}+F_k+w^*(F)[f_s]\\
		&=\sum\limits_{k=0}^N\sum\limits_{n=0}^{N'}\frac{\pi^\frac{2-d}{2}n!}{4^kk!(2n)!} \binom{k+n+1-\frac{d}{2}}{n}\left(\nu_w^{k-\tfrac{d}{2}+1}W_k\circ w\right)^{(2n)}(0)\\&\cdot \M'(f)({2k+2n+3-d})s^{2k+2n+3-d}+O(s^m).
	\end{align*}
As $m$ was arbitrary, this concludes the proof, the more precise estimate following from the estimates for the error terms made above.
\end{proof}
We want to apply this theorem with $G_x$ and the Hadamard coefficients. This works immediately if $P$ acts on scalar functions. In general, we get problems with ``integrating'' a function that takes values in different fibres of a vector bundle. To solve this, we compose with an arbitrary bundle coordinate. We then get this chapter's main result:

\begin{thm}
	\label{intexp}
	Assume that $U$ is a GE set. Let $A:E\otimes E_x^*\rightarrow\C$ be an arbitrary bundle coordinate. Then $w^*(AG_x)$ is well-defined and we have
	\begin{align*}
		w^*(A(G_x))[f_s]\stackrel{s\rightarrow 0}{\sim}&\insum{k}\insum{n}a(k,n)\left(\nu_w^{k-\tfrac{d}{2}+1}AV^{k,U}_x\circ w\right)^{(2n)}(0)s^{2k+2n+3-d}
	\end{align*}
with $a(k,n)$ as in Definition \ref{dfa}.

More precisely, for $s$ small enough such that $f_s\in C_c^\infty (I')$ with $I'\subseteq I$ closed and for every $m\in \N$, if we choose $N,N'$ large enough, we can estimate the remainder terms by:
\begin{align*}
	&\left| w^*(A(G_x))[f_s]-\sum\limits_{k=0}^{N}\sum\limits_{n=0}^{N'}a(k,n)\left(\nu_w^{k-\tfrac{d}{2}+1}AV^{k,U}_x\circ w\right)^{(2n)}(0)s^{2k+2n+3-d}\right|\\
	&\leq C\left(\max\limits_{k\leq N'}\Big\|\nu_w^{k-\tfrac{d}{2}+1}AV^{k,U}_x\circ w\Big\|_{C^{N'}(I)}+\Big\|A\Big(G_x- \sum\limits_{k=0}^N V^{k,U}_xR^U(2k+2,x)\Big)\circ w\Big\|_{C^{m}(I')}\right)s^m.
\end{align*}

\end{thm}
\begin{proof}
	As $A$ is linear, preserves differentiability orders and we can pull out scalar distributions, we have on $U$
	\[AG_x\sim\insum{k} (A(V^{k,U}_x))R^U(2k+2,x).\]
	Thus we can apply \ref{Wexp} to obtain the desired result.
\end{proof}
Now we will need the remainder estimates that we have been dragging around with us. These give us the uniformity we need to integrate this asymptotic expanson, giving the following generalization of the result:
\begin{thm}
	\label{smearedexp}
	Let $M'\subset M$ be any submanifold and let $\chi\in C_c^\infty(M')$ be arbitrary. Let $A$ be any bundle coordinate on $E\boxtimes E^*$  and let $A_x$ be its restriction to $E\otimes E^*_x$.
	Let $(w_x)_{x\in M'}$ be a family of timelike curves defined on an interval $I$ containing $0$ such that $w_x(0)=x$ and
	\[\phi(x,t):= w_x(t)\]
	is smooth. Then we have
	\begin{align*}
		\int\limits_{M'} \chi(x)w_x^*(A_xG_x)[f_s]dx\stackrel{s\rightarrow 0}{\sim}&\insum{k}\insum{n}a(k,n)\int\limits_{M'}\chi(x)\left(\nu_{w_x}^{k-\tfrac{d}{2}+1}A_xV^{k}_x\circ w_x\right)^{(2n)}(0)dx\ s^{2k+2n+3-d}
	\end{align*}
	with $a(k,n)$ as in Definition \ref{dfa}.
	\end{thm}
\begin{proof}
	Any $x\in \supp(\chi)$ has a GE neighborhood $U_x$ in $M$ by Proposition \ref{GEN}. As $\phi^{-1}(U)$ is open, we can find a relatively compact neighborhood $O_x$ of $x$ in $M'$ and $s_x>0$ such that $\phi(O_x\times s_xI_f)\subseteq U_x$. As $\supp(\chi)$ is compact, there is a finite subcover $(O_x)_{x\in J}$ of $\supp(\chi)$. Choose \[s_0\leq\min\limits_{x\in J}s_x,\] such that $s_0 I_f\subseteq I$. In the following, assume $s<s_0$. 
	By choice of $s_0$, we see that for any $t\in s_0 I_f$ and $y\in \supp(\chi)$, $w_y(t)$ is well-defined and inside a GE neighborhood of $y$ independent of $t$. Restricting $w_y$ to a neighborhood of $s_0I_f$ that still gets mapped to this $GE$ neighborhood, we can apply the local result at each point.
	
	Let $m\in \N$ be arbitrary and choose sufficiently large $N$ and $N'$ as in Theorem \ref{intexp}. For any $i\in J$, we can then estimate
	\begin{align*}
		&\int\limits_{O_i}|\chi(x)|\Big| w_x^*(A_xG_x)[f_s]-\sum\limits_{k=0}^{N}\sum\limits_{n=0}^{N'}a(k,n)\left(\nu_{w_x}^{k-\tfrac{d}{2}+1}A_xV^{k,U_i}_x\circ w_x\right)^{(2n)}(0)\ s^{2k+2n+3-d}\Big|dx\\
		&\leq\int\limits_{O_i} |\chi(x)|C\Big(N'\max\limits_{k\leq N'}\|\nu_{w_x}^{k-\tfrac{d}{2}+1}A_xV^{k,U_i}_x\circ w_x\|_{C^{N'}(s_0I_f)}\\&+\|A_x(G_x- \sum\limits_{k=0}^N V^{k,U_i}_xR^{U_i}(2k+2,x))\circ w_x\|_{C^{m}(s_0I_f)}\Big)dx\ s^m\\
		&\leq C' s^m.
	\end{align*}
In the last estimate, we used that $O_i$ is relatively compact and thus for any $C^K$-function  $g\colon I\times O_i\rightarrow \C$, we have 
\[\sup_{x\in O_i}\|g(\cdot,x)\|_{C^K(s_0I_f)}\leq \|g\|_{C^K(s_0I_f\times O_i)}<\infty.\]
As the $O_i$ cover $\supp(\chi)$, we can conclude
\begin{align*}
	&\Big|\int\limits_{M'}\chi(x) w_x^*(A_xG_x)[f_s]dx-\sum\limits_{k=0}^{N}\sum\limits_{n=0}^{N'}a(k,n)\int\limits_{M'}\chi(x)\left(\nu_{w_x}^{k-\tfrac{d}{2}+1}A_xV^{k,U_i}_x\circ w_x\right)^{(2n)}(0)dx\ s^{2k+2n+3-d}\Big|\\
	&\leq\sum\limits_{i\in I}\int\limits_{O_i}|\chi(x)|\Big|\chi(x) w_x^*(A_xG_x)[f_s]-\sum\limits_{k=0}^{N}\sum\limits_{n=0}^{N'}a(k,n)\left(\nu_{w_x}^{k-\tfrac{d}{2}+1}A_xV^{k,U_i}_x\circ w_x\right)^{(2n)}(0)\ s^{2k+2n+3-d}\Big|dx\\
	&\leq C s^m.
\end{align*}
As $m$ was arbitrary, this concludes the proof.
\end{proof}
\newpage
\chapter{Extracting Hadamard coefficients}
\label{Hadamextract}
We now seek to use the result of the previous chapter to obtain the Hadamard coefficients on the diagonal, i.e. $V^{k}_x(x)$. The problem with the asymptotic expansion we derived is that we cannot distinguish between the contribution coming from the $k$-th Hadamard coefficient and that from the $2j$-th derivative of the $(k-j)$-th coefficient. This problem does not arise from the choice of action-like function we construct, but is inherent in the asymptotic expansion we are using.

The problem is similar to that of determining the coefficients of a power series 
\[\insum{k}a_k(x)x^k\]
where the coefficients are themselves depending on $x$. Not even the values of the coefficients at $x=0$ (except for the first) are determined by this expansion, as we may always add some function $g$ to one of the coefficients and subtract $xg$ from the previous one without changing the sum.

This problem also occurs with the Hadamard expansion: As
\[\Gamma R(2k,x)=2k(2k-d+2)R(2k+2,x),\]
we may (for $k\notin \{0, \frac{d}{2}-1\}$) add some $g$ to $V^{k}_x$ and subtract $\frac{\Gamma g}{2k(2k-d+2)}$ from $V^{k+1}_x$ without changing the sum in the asymptotic expansion of $G_x$. Thus $V^{0}$ and $V^{\frac{d}{2}-1}$ (in case $d$ is even) are the only Hadamard coefficients whose value at the diagonal we may hope to obtain by using only this expansion (and those we actually do obtain from our procedure).

We thus need to include some extra information about the Hadamard coefficients, if we want to succeed. We do this  by looking at the Hadamard expansion of $P-z$ for variable $z\in \C$ and using Proposition \ref{Vz} to express the Hadamard coefficients of $P-z$ in terms of those of $P$. An alternative approach would be to use the asymptotic expansion for powers of $G^\pm$ instead.
\section{Hadamard coefficients for $k<\frac{d}{2}$ and $d$ even}
\label{smallk}
In case $d$ is even, the coefficient $V^{\frac{d}{2}-1}_x(x)$ can be read off fairly immediately, either directly from the Green's kernel or from the expansion of the action analogue that we constructed. Knowing this coefficient for any $P-z$ also gives us all coefficients of lower order. We shall elaborate on this first, before we continue with deriving our more general but also more complicated formula for arbitrary coefficients. A further advantage of the formula in this special case is that it works for arbitrary timelike paths, while the more general formula only works for geodesics.

 We will need some notation for the coefficients to specific monomials in a polynomial or an asymptotic expansion.
\begin{df}
	\label{monocoeff}
	For a function $F$ with an asymptotic expansion
	\[F(s)\stackrel{s\rightarrow0}{\sim}\insum k a_k s^{b_k}\]
	with $(b_k)$ strictly increasing, we define
	\[F[[b_k]]:=F(s)[[s^{b_k}]]:=a_k.\]
	If instead $F$ depends on two parameters, we define
	\[F[[\alpha,\beta]]:=F(s,z)[[s^\alpha z^\beta]]:=(F(s,z)[[s^\alpha]])[[z^\beta]],\]
	whenever the right hand side is defined.
\end{df}
Note that polynomials are a special case of an asymptotic expansion (with almost all summands equal to zero). As the asyptotic expansion used in this definition is unique (up to adding summands with zero coefficients), $F[[b_k]]$ is well-defined. It can be calculated by looking at the residues of the Mellin transform of $F$. In case all $b_k$ are of the form $b_k=b_0+n_k$ for $n_k\in \N$ (as will always be the case here), they can be calculated more easily as
\[F[[b_0+n]]=\frac{1}{n!}\partial_s^n s^{-b_0}F(s)|_{s=0}.\]
We now proceed to show how the first few hadamard coefficients in even dimensions can be computed from the Green's kernel.
\begin{prop}
	Assume that $d$ is even and $w$ is a future oriented timelike curve contained in a GE subset $U$ of $M$ with $w(0)=x$. Then for $t\neq 0$, $G_x$ is a continuous function around $w(t)$ and we have
	\[V^{\frac{d}{2}-1}_x(x)=\lim\limits_{t\rightarrow 0}\ 2\sign(t)(4\pi)^{\frac{d}{2}-1}(\tfrac{d}{2}-1)!G_x(w(t)).\]
	More generally, we have for $k<\frac{d}{2}$
	\[V_x^k(x)=2(\tfrac{d}{2}-1-k)!k!(4\pi)^{\frac{d}{2}-1}\Big(\lim\limits_{t\rightarrow 0}\ \sign(t)G_{P-z,x}(w(t))\Big)[[z^{\frac{d}{2}-1-k}]].\]
\end{prop}
\begin{proof}
	We have on $U$
	\[G_x=\sum\limits_{k=0}^{\frac{d}{2}-1} V_x^kR^U(2k+2,x)+E\]
	for some continuous function $E$ (it suffices to sum up to $\frac {d}{2}-1$, as all higher order Riesz distributions are continuous). For $n<d$ even, the Riesz distributions $R^U(n,x)$ (for even $d$) are supported on the boundary of $J^\pm$ (see \cite{BGP}, Proposition 1.4.2 (8)). $w(t)$ for $t\neq 0$ is in the timelike future or past of $x$, i.e. the interior of $J^\pm$ (recall that by \ref{exJ}, the causal and timelike futures of $x$ are the images of the corresponding cones in $T_xM$, so everything works out). Thus all summands for $k<\frac{d}{2}-1$ in the expansion vanish around $w(t)$. 
	As
	\[R^U(d,x)=c_d\Gamma^0(\mathbbm{1}_{J_+(x)}-\mathbbm{1}_{J_-(x)})=\frac{(4\pi)^{1-\frac{d}{2}}}{2(\frac{d}{2}-1)!}(\mathbbm{1}_{J_+(x)}-\mathbbm{1}_{J_-(x)})\]
	is constant in the future/past of $x$ (and thus around $w(t)$), we can conclude that, in a neighborhood of $w(t)$, $G_x$ coincides with the continuous function
	\[\frac{(4\pi)^{1-\frac{d}{2}}}{2(\tfrac{d}{2}-1)!}\sign(t)V_x^{\frac{d}{2}-1}+E.\]
	As $E$ has to be zero outside of $J(x)$ (because the same is true for $G_x$ and the Riesz distributions), it vanishes at $x$. Since $E$ is continuous, we can conclude that
	\begin{align*}
	\lim\limits_{t\rightarrow 0}\ \sign(t)G_x(w(t))&=\lim\limits_{t\rightarrow 0}\frac{(4\pi)^{1-\frac{d}{2}}}{2(\tfrac{d}{2}-1)!}V_x^{\frac{d}{2}-1}(w(t))+\sign(t)E(w(t))\\
	&=\frac{(4\pi)^{1-\frac{d}{2}}}{2(\tfrac{d}{2}-1)!}V_x^{\frac{d}{2}-1}(x).
	\end{align*}
	This implies the first claim. 
	
	Applying this for $P-z$ with $z\in \C$ and using \ref{Vz}, we get for $k<\frac{d}{2}$
	\begin{align*}&\Big(\lim\limits_{t\rightarrow 0}\ \sign(t)G_{P-z,x}(w(t))\Big)[[z^{\frac{d}{2}-1-k}]]\\
	&=\frac{(4\pi)^{1-\frac{d}{2}}}{2(\frac{d}{2}-1)!}\Big(V_x^k(z)(x)\Big)[[z^{\frac{d}{2}-1-k}]]\\
	&=\frac{(4\pi)^{1-\frac{d}{2}}}{2(\frac{d}{2}-1)!}\Big(\sum\limits_{n=0}^{\frac{d}{2}-1}\binom{\frac{d}{2}-1}{n}z^{\frac{d}{2}-1-n}V_x^n(x)\Big)[[z^{\frac{d}{2}-1-k}]]\\
	&=\frac{(4\pi)^{1-\frac{d}{2}}}{2(\frac{d}{2}-1)!}\binom{\frac{d}{2}-1}{k}V_x^k(x)\\
	&=\frac{(4\pi)^{1-\frac{d}{2}}}{2(\frac{d}{2}-1-k)!k!}V_x^k(x).\\
	\end{align*}
	Thus we have
	\[V_x^k(x)=2(\tfrac{d}{2}-1-k)!k!(4\pi)^{\frac{d}{2}-1}\Big(\lim\limits_{t\rightarrow 0}\ \sign(t)G_{P-z,x}(w(t))\Big)[[z^{\frac{d}{2}-1-k}]].\qedhere\]
\end{proof}
We also get an analogous formula from the asymptotic expansion derived in the previous chapter.
\begin{thm}
	\label{smallkthm}
	Assume that $d$ is even, $\M'(f)(1)\neq 0$ and the assumptions of Theorem \ref{smearedexp} hold. Then we have for $k<\tfrac{d}{2}$:
	\[\int\limits_{M'}\chi(x)A_xV^{k}_x(x)dx=\frac{(4\pi)^{\frac{d}{2}-1}{(\frac{d}{2}-1-k)}!k!}{\M'(f)(1)}\int\limits_{M'}\chi(x)w_x^*(A_xG_{P-z,x})[f_s]dx[[s^1z^{\frac{d}{2}-1-k}]].\]
\end{thm}
\begin{rem}
	Taking $M'=\{x\}$ (and $\chi=1$) we obtain the pointwise result
	\[A_xV^{k}_x(x)=\frac{(4\pi)^{\frac{d}{2}-1}{(\frac{d}{2}-1-k)}!k!}{\M'(f)(1)}\big(w_x^*(A_xG_{P-z,x})[f_s]\big)[[s^1z^{\frac{d}{2}-1-k}]].\]
\end{rem}
\begin{proof}
	For integer $n>0$ and $k+n=\tfrac{d}{2}-1$, we have
	\[\binom{k+n+1-\frac{d}{2}}{n}=\binom{0}{n}=0\]
	and thus 
	\[a_{k,n}=0.\]
	This means that the only summand in the expansion of Theorem \ref{smearedexp} that gives a non-zero contribution of order $s^{1}$ (corresponding to $n+k=\frac{d}{2}-1$) is the summand for $k=\tfrac{d}{2}-1$ and $n=0$. We thus obtain from \ref{smearedexp}:
	\begin{align*}
		\int\limits_{M'}\chi(x)w_x^*(A_xG_x)[f_s]dx[[s^1]]&=a(\tfrac{d}{2}-1,0)\int\limits_{M'}\chi(x)A_xV^{\frac{d}{2}-1}_x(w_x(0))dx\\
		&=\frac{\M'(f)(1)}{(4\pi)^{\frac{d}{2}-1}{(\frac{d}{2}-1)}!}\int\limits_{M'}\chi(x)A_xV^{\frac{d}{2}-1}_x(x)dx.\\
	\end{align*}
Applying this for $P-z$ instead of $P$ and using Proposition \ref{Vz} to express the Hadamard coefficient for $P-z$ in terms of those for $P$, we obtain
\begin{align*}
	&\int\limits_{M'}\chi(x)w_x^*(A_xG_{P-z,x})[f_s]dx[[s^1z^m]]\\
	&=\frac{\M'(f)(1)}{(4\pi)^{\frac{d}{2}-1}{(\frac{d}{2}-1)}!}\int\limits_{M'}\chi(x)A_xV^{\frac{d}{2}-1}_x(z)(x)dx[[z^m]]\\
	&=\frac{\M'(f)(1)}{(4\pi)^{\frac{d}{2}-1}{(\frac{d}{2}-1)}!}\int\limits_{M'}\chi(x)A_x\sum\limits_{m'=0}^{\frac{d}{2}-1}\binom{\frac{d}{2}-1}{m'}V^{\frac{d}{2}-1-m'}_x(x)z^{m'}dx[[z^m]]\\
	&=\frac{\M'(f)(1)}{(4\pi)^{\frac{d}{2}-1}{(\frac{d}{2}-1)}!}\binom{\frac{d}{2}-1}{m}\int\limits_{M'}\chi(x)A_xV^{\frac{d}{2}-1-m}_x(x)dx.\\
	&=\frac{\M'(f)(1)}{(4\pi)^{\frac{d}{2}-1}{m!(\frac{d}{2}-1-m)}!}\int\limits_{M'}\chi(x)A_xV^{\frac{d}{2}-1-m}_x(x)dx.\\
\end{align*}
Applying this with $m=\frac{d}{2}-1-k$ and moving the prefactors to the other side, we obtain
\[\int\limits_{M'}\chi(x)A_xV^{k}_x(x)dx=\frac{(4\pi)^{\frac{d}{2}-1}{(\frac{d}{2}-1-k)}!k!}{\M'(f)(1)}\int\limits_{M'}\chi(x)w_x^*(A_xG_{P-z,x})[f_s]dx[[s^1z^{\frac{d}{2}-1-k}]].\qedhere\]
\end{proof}
In the special case $k=1$, we can use this to obtain a formula for the scalar curvature.
\begin{cor}
	Let $U$, $w$ and $x$ be as above. Assume that $P=\square$ is the  d'Alembertian for the metric $g$ (acting on scalar functions). For $d>2$ even, the local scalar curvature of $M$ is given by
	\[\scal(x)=\frac{6(4\pi)^{\frac{d}{2}-1}(\frac{d}{2}-2)!}{\M'(f)({1})}\big( w^*(G_{\square-z,x})[f_s]\big)[[s^1z^{\tfrac{d}{2}-2}]].\]
	In the special case $d=4$, we obtain
	\[\scal(x)=\frac{24\pi}{\M'(f)({1})} \big(w^*(G_{\square,x})[f_s]\big)[[s^1]].\]
\end{cor}
\begin{proof}
	We have (see \cite{BGP}, last line before Remark 2.3.2)
	\[V^1_x(x)=\frac{1}{6}\scal(x).\]
	Thus we obtain from the above (with $M'=\{x\}$, $A=id$ and $K=1$)
	\begin{align*}
		\scal(x)&=6V^1_x(x)\\
		&=\frac{6(4\pi)^{\frac{d}{2}-1}(\frac{d}{2}-2)!}{\M'(f)({1})} w^*(G_{\square-z,x})[f_s][[s^1z^{\tfrac{d}{2}-2}]].
	\end{align*}
	In case $d=4$, we can use the fact that $w^*(G_{\square-z,x})[f_s][[s^1]]$ has  no negative $z$-powers and thus the $z^0$-coefficient is just the value at $z=0$. We obtain
	\[\scal(x)=\frac{24\pi}{\M'(f)({1})} w^*(G_{\square,x})[f_s][[s^1]].\qedhere\]
\end{proof}
\section{Arbitrary Hadamard coefficients}
\label{allk}
We now turn towards deriving a formula for arbitrary Hadamard coefficients. In this case we will have to look at coefficients for various powers of $s$ and $z$ in the action-like function for $P-z$ and combine these via linear algebra.
\begin{prop}
	\label{Expwithz}
	Under the assumptions of Theorem \ref{intexp}, we have for any $z\in \C$ the following asymptotic expansion in $s$:
	\begin{align*}
		w^*(A(G_{P-z,x}))[f_s]\stackrel{s\rightarrow 0}{\sim}&\insum{l,m,n}\frac{\pi^\frac{2-d}{2}n!}{4^{l+m}l!m!(2n)!} \binom{l+m+n+1-\frac{d}{2}}{n}z^m\left(\nu_w^{l+m-\tfrac{d}{2}+1}AV^{l}_x\circ w\right)^{(2n)}(0)\\&\cdot \M'(f)({2l+2m+2n+3-d})s^{2l+2m+2n+3-d}\\
	\end{align*}
\end{prop}
\begin{proof}
	Using Theorem \ref{intexp} and Proposition \ref{Vz}, we obtain
	\begin{align*}
		&w^*(A(G_{P-z,x}))[f_s]\\
		\stackrel{s\rightarrow 0}{\sim}&\insum{k}\insum{n}\frac{\pi^\frac{2-d}{2}n!}{4^kk!(2n)!} \binom{k+n+1-\frac{d}{2}}{n}\left(\nu_w^{k-\tfrac{d}{2}+1}AV^{k}_x(z)\circ w\right)^{(2n)}(0)\\&\cdot \M'(f)({2k+2n+3-d})s^{2k+2n+3-d}\\
		=&\insum{k}\insum{n}\frac{\pi^\frac{2-d}{2}n!}{4^kk!(2n)!} \binom{k+n+1-\frac{d}{2}}{n}\sum\limits_{m=0}^k\binom{k}{m}z^m\left(\nu_w^{k-\tfrac{d}{2}+1}AV^{k-m}_x\circ w\right)^{(2n)}(0)\\&\cdot \M'(f)({2k+2n+3-d})s^{2k+2n+3-d}\\
		=&\insum{l,m,n}\frac{\pi^\frac{2-d}{2}n!}{4^{l+m}(l+m)!(2n)!} \binom{l+m+n+1-\frac{d}{2}}{n}\binom{l+m}{m}z^m\left(\nu_w^{l+m-\tfrac{d}{2}+1}AV^{l}_x\circ w\right)^{(2n)}(0)\\&\cdot \M'(f)({2l+2m+2n+3-d})s^{2l+2m+2n+3-d}\\
		=&\insum{l,m,n}\frac{\pi^\frac{2-d}{2}n!}{4^{l+m}l!m!(2n)!} \binom{l+m+n+1-\frac{d}{2}}{n}z^m\left(\nu_w^{l+m-\tfrac{d}{2}+1}AV^{l}_x\circ w\right)^{(2n)}(0)\\&\cdot \M'(f)({2l+2m+2n+3-d})s^{2l+2m+2n+3-d}\qedhere
	\end{align*}
\end{proof}
This is the point where $\nu_w$ becomes a problem. It makes the expression
\[\left(\nu_w^{l+m-\tfrac{d}{2}+1}AV^{l}_x\circ w\right)^{(2n)}(0)\]
depend on $m$, in a way that can not be easily separated. This would be a problem for our later calculations, so from now on we have to assume that $w$ is a timelike unit speed geodesic and hence $\nu_w$ is $1$, removing the $m$-dependence.
We consider a slightly more general version of the above expansion for such $w$ in the following, both to shorten notation and to be able to apply the results not only for the Hadamard coefficients themselves but also for their integral over some subset of $M$.
\begin{gn}
	\label{dfWL}
	For this chapter, assume that $W_{l,n}$ are complex numbers for $l,n\in \N$, that $I$ is an interval containing zero and that
	\[L\colon I\times I\rightarrow \C\]
	is a function such that we have the asymptotic expansion
	\begin{align*}
		L(s,z)\stackrel{s\rightarrow 0}{\sim}&\insum{l,m,n}\frac{4^{-m}}{m!} \binom{l+m+n+1-\frac{d}{2}}{n}z^mW_{l,n}\\&\cdot \M'(f)({2l+2m+2n+3-d})s^{2l+2m+2n+3-d}\\
	\end{align*}
\end{gn}
The main example is, of course,
\[L(s,z)=w^*(A(G_{P-z,x}))[f_s]\]
and 
\[W_{l,n}=\frac{\pi^\frac{2-d}{2}n!}{4^{l}l!(2n)!}\left(AV^{l}_x\circ w\right)^{(2n)}(0).\]
Our goal is to find $W_{k,0}$, which corresponds to a multiple of $AV^k_x(x)$.
The information we can extract from the asymptotic expansion of $L$ consists of the coefficients $L[[k,m]]$. We want to be able to express the $W_{k,0}$ in terms of these.
\begin{lemma}
	\label{Powercoeff}
	We have for $K,m\in\N$:
	\[L[[2K+2m+3-d,m]]=\sum\limits_{n=0}^K \frac{4^{-m}}{m!} \binom{K+m+1-\frac{d}{2}}{n} \M'(f)({2K+2m+3-d})W_{K-n,n}\]
\end{lemma}
\begin{proof}
	As
	\[L(s,z)\stackrel{s\rightarrow 0}{\sim} \insum k \sum\limits_{\ l+n+m'=k}\frac{4^{-m'}}{m'!} \binom{k+1-\frac{d}{2}}{n}z^{m'}W_{l,n} \M'(f)({2k+3-d})s^{2k+3-d},\]
	we have
	\begin{align*}
		&L(s,z)[[s^{2K+2m+3-d}]]\\
		&=\sum\limits_{l+n+m'=K+m}\frac{4^{-m'}}{m'!} \binom{K+m+1-\frac{d}{2}}{n}z^{m'}W_{l,n} \M'(f)({2K+2m+3-d})\\
		&=\sum\limits_{m'=0}^K\sum\limits_{\ \ l+n=K+m-m'}\frac{4^{-m'}}{m'!} \binom{K+m+1-\frac{d}{2}}{n}W_{l,n} \M'(f)({2K+2m+3-d})z^{m'}.\\
	\end{align*}
	We obtain
	\begin{align*}
		&L[[2K+2m+3-d,m]]\\
		&=\sum\limits_{l+n=K}\frac{4^{-m}}{m!} \binom{K+m+1-\frac{d}{2}}{n}W_{l,n} \M'(f)({2K+2m+3-d})\\
		&=\sum\limits_{n=0}^K \frac{4^{-m}}{m!} \binom{K+m+1-\frac{d}{2}}{n} \M'(f)({2K+2m+3-d})W_{K-n,n}.\qedhere
	\end{align*}
\end{proof}
For fixed $K$, the above equation has the general form of a system of linear equations
\[x_m=\sum\limits_{n=0}^{K}a_{mn}y_n.\]
Basic linear algebra tells us how to solve these equations for the $y_n=W_{K-n,n}$, given a left inverse for the matrix with entries $a_{mn}$. This left inverse is highly non-unique, as we are dealing with an ``$\infty\times (K+1)$''-matrix ($m$ can range over all of $\N$). We will proceed by computing the inverse of a $(K+1)\times (K+1)$ block that starts at $m=o$ for some offset $o$, in order to compute $W_{K,0}$. We will later obtain certain special cases by fixing $o$ to some specific value.
\begin{gn}
	For the remainder of this chapter, fix $o,K\in \N$. Set $\delta:=\frac{d}{2}-1-o$.
\end{gn}

We will also need some efficient notation for a matrix or vector with entries specified by a function that cannot be easily written in an argument-free way. To achieve this, we reserve the letters $i$ and $j$ as placeholder variables for vector or matrix indices for the remainder of this chapter.
\begin{df}
	\label{dfmatvec}
	 Let
	\[\left[a(i,j)\right]\]
	denote the $(K+1)\times (K+1)$ matrix with $ij$-th entry $a(i,j)$ (the indexing variables are set to be always $i$ and $j$). 
	Let
	\[\diag(b(i)):=\left[b(i)\delta _{ij}\right]\]
	denote the diagonal matrix with entries $b(i)$ and let
	\[\left<b(i)\right>\]
	denote the $K+1$-vector with $i$-th component $b(i)$.	
	All indices are assumed to start at $0$ (and thus run up to $K$).
\end{df}

An example calculation with this notation would be for any $a\in \R$:
\[\left[\binom{i}{j}\right]\<a^i\>=\<\sum\limits_{l=0}^K\binom{i}{l}a^l\>=\<(a+1)^i\>.\]

Lemma \ref{Powercoeff} can be expressed as a matrix-vector multiplication: Considering both sides as the $(m-o)$-th component of a vector and using the definition of $\delta$ to slightly simplify notation, the statement of \ref{Powercoeff} for $o\leq m\leq o+K$ is equivalent to
\[\<L[[2K+2i+1-2\delta,i+o]]\>=\left[\frac{4^{-(i+o)}}{(i+o)!} \binom{K+i-\delta}{j} \M'(f)({2K+2i+1-2\delta})\right]\<W_{K-i,i}\>.\]
We thus aim to find a left inverse for the matrix 
\begin{align*}
	&\left[\frac{4^{-(i+o)}}{(i+o)!} \binom{K+i-\delta}{j} \M'(f)({2K+2i+1-2\delta})\right]\\ &=\diag\left(\frac{4^{-(i+o)}}{(i+o)!}\M'(f)({2K+2i-2\delta+1})\right)\left[\binom{K+i-\delta}{j}\right].
\end{align*}
As the diagonal matrix is easy to invert as long as the Mellin transform of $f$ is non-vanishing at the relevant integers, we may concentrate on finding an inverse to the second factor. We will do so in two steps, first reducing to another matrix of binomial coefficients, which we can then find an inverse to. 
\begin{lemma}
	\label{D}
	We have 
	\[\left[(-1)^{i-j}\binom{i}{j}\right]\left[\binom{K+i-\delta}{j}\right]=\left[\binom{K-\delta}{j-i}\right].\]
\end{lemma}
\begin{proof}
	Define the (right) discrete derivative of a function $\phi$ by
	\[D_m\phi(m):=(D\phi)(m):=\phi(m+1)-\phi(m).\]
	This acts on binomial coefficients as follows:
	\[D_m \binom{a+m}{k}=\binom{a+m+1}{k}-\binom{a+m}{k}=\binom{a+m}{k-1}.\]
	Applying this inductively, we obtain
	\[\left.D_i^l\binom{K+i-\delta}{j}\right|_{i=0}=\binom{K-\delta}{j-l}.\]
	We have with $I$ denoting the identity and $A\phi(x):=\phi(x+1)$:
	\[D^l\phi(0)=(A-I)^l\phi(0)=\sum\limits_{m=0}^l\binom{l}{m}(-I)^{l-m}A^m\phi(0)=\sum\limits_{m=0}^l(-1)^{l-m}\binom{l}{m}\phi(m).\]
	Combining the two, we obtain
	\begin{align*}
		\left[(-1)^{i-j}\binom{i}{j}\right]\left[\binom{K+i-\delta}{j}\right]&=\left[\sum\limits_{m=0}^K(-1)^{i-m}\binom{i}{m}\binom{K+m-\delta}{j}\right]\\
		&=\left[\sum\limits_{m=0}^{i}(-1)^{i-m}\binom{i}{m}\binom{K+m-\delta}{j}\right]\\
		&=\left[\left.D_m^{i}\binom{K+m-\delta}{j}\right|_{m=0}\right]\\
		&=\left[\binom{K-\delta}{j-i}\right].\qedhere
	\end{align*}
\end{proof}
We now proceed to find an inverse for the latter.
\begin{lemma}
	\label{in1}
	$\left[\binom{K-\delta}{j-i}\right]$ is invertible with inverse
	\[\left[\binom{\delta-K}{j-i}\right].\]
\end{lemma}
\begin{proof}
	With the Chou-Vandermonde identity
	\[\sum\limits_{l=0}^{k}\binom{m}{k-l}\binom{n}{l}=\binom{m+n}{k},\]
	(for $k\in \N$ and $m,n\in \C$ arbitrary), we obtain (taking sums to be zero if the upper bound is smaller than the lower bound of summation)
	\begin{align*}
		\left[\binom{\delta-K}{j-i}\right]\left[\binom{K-\delta}{j-i}\right]
		&=\left[\sum\limits_{l=0}^K \binom{\delta-K}{l-i}\binom{K-\delta}{j-l}\right]\\
		&=\left[\sum\limits_{l=i}^j \binom{\delta-K}{l-i}\binom{K-\delta}{j-l}\right]\\
		&=\left[\sum\limits_{l=0}^{j-i}\binom{\delta-K}{j-i-l}\binom{K-\delta}{l}\right]\\
		&=\left[\binom{0}{j-i}\right]\\
		&=[\delta_{ij}]\\
		&=1. \qedhere
	\end{align*}
\end{proof}
We now put these results together. As we will have to divide by some values of $\M'(f)$, we have to assume that those are non-zero.
\begin{gn}
	\label{Mfnonzero}
	From now on, assume that $\M'(f)$ is non-zero at integers.
\end{gn}
\begin{lemma}
	\label{in2}
	\[\left[\frac{4^{-(i+o)}}{(i+o)!} \M'(f)({2K+2i-2\delta+1}) \binom{K+i-\delta}{j}\right]\] is invertible with inverse 
	\[\left[\sum\limits_{l=0}^K \frac{4^{j+o}(j+o)!(-1)^{l-j}}{\M'(f)({2K+2j-2\delta+1})}\binom{\delta-K}{l-i}\binom{l}{j}\right].\]
\end{lemma}
\begin{proof}
	Abbreviate
	\[h(i):=\frac{4^{-(i+o)}}{(i+o)!}\M'(f)({2K+2i-2\delta+1})\]
	Using lemmas \ref{D} and \ref{in1}, we obtain
	\begin{align*}
		1&=\left[\binom{\delta-K}{j-i}\right]\left[\binom{K-\delta}{j-i}\right]\\
		&=\left[\binom{\delta-K}{j-i}\right]\left[(-1)^{i-j}\binom{i}{j}\right]\left[\binom{K+i-\delta}{j}\right]\\
		&=\left[\sum\limits_{l=0}^K (-1)^{l-j}\binom{\delta-K}{l-i}\binom{l}{j}\right]\diag(h(i))^{-1}\left[h(i) \binom{K+i-\delta}{j}\right]\\
		&=\left[\sum\limits_{l=0}^K \frac{(-1)^{l-j}}{h(j)}\binom{\delta-K}{l-i}\binom{l}{j}\right]\left[h(i) \binom{K+i-\delta}{j}\right],
	\end{align*}
	Thus
	\[\left[\sum\limits_{l=0}^K \frac{(-1)^{l-j}}{h(j)}\binom{\delta-K}{l-i}\binom{l}{j}\right]\]
	is the inverse for 
	\[\left[h(i) \binom{K+i-\delta}{j}\right].\]
	Inserting the definition of $h$ yields the desired result.
\end{proof}
With this, we can now obtain a formula for $W_{K,0}$ (and thus the $K$-th Hadamard coefficient at the diagonal), yielding the abstractified version of our main result.

\begin{thm}
	\label{Wfinal}
	We have
	\begin{align*}
		W_{K,0}=&
		\sum\limits_{m=0}^K \frac{4^{m+o}(m+o)!}{\M'(f)({2K+2m-2\delta+1})}\binom{\delta-K}{m}\binom{2K-\delta}{K-m} L[[2K+2m-2\delta+1,m+o]].
	\end{align*}
\end{thm}
\begin{proof}
	The proof basically consists of combining lemmas \ref{Powercoeff} and \ref{in2}, with some extra calculations done in advance to simplify the result.
	Using the identity
	\[\binom{a}{b}\binom{b}{c}=\binom{a}{c}\binom{a-c}{b-c},\]
	we get for any $l,m\leq K$:
	\[\binom{\delta-K}{l}\binom{l}{m}=\binom{\delta-K}{m}\binom{\delta-K-m}{l-m}.\]
	Using (for integer $b$)
	\[\binom{-a}{b}=(-1)^b\binom{a+b-1}{b}\]
	and
	\[\sum\limits_{k=0}^n\binom{a+k}{k}=\binom{a+n+1}{n}\]
	we get
	\begin{align*}
		\sum\limits_{l=0}^K(-1)^{l-m}\binom{\delta-K-m}{l-m}&=\sum\limits_{l=0}^K\binom{K-\delta+l-1}{l-m}\\
		&=\sum\limits_{l=m}^K\binom{K-\delta+l-1}{l-m}\\
		&=\sum\limits_{l=0}^{K-m}\binom{K-\delta+m-1+l}{l}\\
		&=\binom{K-\delta+m-1+K-m+1}{K-m}\\
		&=\binom{2K-\delta}{K-m}.
	\end{align*}
	We obtain
	\begin{align*}
		\sum\limits_{l=0}^K(-1)^{l-m}\binom{\delta-K}{l}\binom{l}{m}
		&=\sum\limits_{l=0}^K(-1)^{l-m}\binom{\delta-K}{m}\binom{\delta-K-m}{l-m}\\
		&=\binom{\delta-K}{m}\binom{2K-\delta}{K-m}.\\
	\end{align*}
	By viewing Lemma \ref{Powercoeff} as a vector identity, we have 
	\[\<L[[2K+2i+1-2\delta,i+o]]\>=\left[\frac{\M'(f)({2K+2i-2\delta+1})}{4^{i+o}(i+o)!}\binom{K+i-\delta}{j}\right]\left<W_{K-i,i}\right>.\]
	By \ref{in2} and the previous calculation, we obtain
	\begin{align*}
		&W_{K,0}\\
		=&\<W_{K-i,i}\>_0\\
		=&\left(\left[\frac{\M'(f)({2K+2i-2\delta+1})}{4^{i+o}(i+o)!}\binom{K+i-\delta}{j}\right]^{-1}\left<L[[2K+2i-2\delta+1,i+o]]\right>\right)_0\\
		=&\left(\left[\sum\limits_{l=0}^K \frac{4^{j+o}(j+o)!(-1)^{l-j}}{\M'(f)({2K+2j-2\delta+1})}\binom{\delta-K}{l-i}\binom{l}{j}\right]\left<L[[2K+2i-2\delta+1,i+o]]\right>\right)_0\\
		=&\sum\limits_{m,l=0}^K \frac{4^{m+o}(m+o)!(-1)^{l-m}}{\M'(f)({2K+2m-2\delta+1})}\binom{\delta-K}{l}\binom{l}{m}L[[2K+2m-2\delta+1,m+o]]\\
		=&\sum\limits_{m=0}^K \frac{4^{m+o}(m+o)!}{\M'(f)({2K+2m-2\delta+1})}\binom{\delta-K}{m}\binom{2K-\delta}{K-m}L[[2K+2m-2\delta+1,m+o]].\qedhere
	\end{align*}
\end{proof}
The linear combination obtained above, including the extra factors hidden in the definition of $W$, will come up several times in the following results. We thus fix some notation for it.
\begin{df}
	\label{dfXi}
	For any function $\mathcal{L}$ in two variables that has a suitable asymptotic expansion around $0$ and $K,o\in \N$, define
	\begin{align*}
		\Xi_{K,o}(\mathcal{L}(s,z))_{s,z}:=\Xi_{K,o}(\mathcal{L}):=&\sum\limits_{m=0}^K \frac{\pi^{\frac{d}{2}-1}4^{K+m+o}(m+o)!K!}{\M'(f)({2K+2m+2o-d+3})}\binom{\frac{d}{2}-1-o-K}{m}\\&\cdot\binom{2K+o+1-\frac{d}{2}}{K-m}\mathcal{L}[[2K+2m+2o-d+3,m+o]].
	\end{align*}
\end{df}
We obtain our main formula for the Hadamard coefficients on the diagonal by replacing, in the abstract version above, $W$ and $L$ with the things they were created to resemble.
\begin{thm}
	\label{finalgen}
	Let $o,K\in\N$. Let $U\subset M$ be GE, $I$ an interval containing $0$, $w\colon I\rightarrow U$ a timelike unit speed geodesic and $x=w(0)$. Let $f\in C_c^\infty(\R)$ be odd and assume that $\M'(f)$ is non-zero on $\Z$. Let $A:E\otimes E_x^*\rightarrow \C$ be an arbitrary bundle coordinate. Then we have
	\begin{align*}
		AV^{K,U}_x(x)=&\Xi_{K,o}(w^*(AG_{P-z,x})[f_s])_{s,z}\\
		=&\sum\limits_{m=0}^K \frac{\pi^{\frac{d}{2}-1}4^{K+m+o}(m+o)!K!}{\M'(f)({2K+2m+2o-d+3})}\binom{\frac{d}{2}-1-o-K}{m}\binom{2K+o+1-\frac{d}{2}}{K-m}\\
		&\cdot w^*(AG_{P-z,x})[f_s][[s^{2K+2m+2o-d+3}z^{m+o}]].
	\end{align*}
\end{thm}
\begin{rem}
	In case $M$ is not globally hyperbolic, we can still find an open, convex, globally hyperbolic negihborhood $U$ for any point in $M$ by \cite[Corollary 2]{Min} (causal compatibility will be unobtainable if $M$ is not strongly causal). We can then apply the theorem to $U$ instead of $M$, as $U$ is always causally compatible to itself. As the Hadamard coefficients on $U$ only depend on $P|_U$ (and $g|_U$), we still obtain a formula for the Hadamard coefficients on $U$. However, it will be in terms of the Green's kernel for $P|_U$, which may not coincide with the restriction of the Green's kernel for $P$ (which need not even exist). In case $M$ is strongly causal a GE neighborhood $U$ does exist (\cite[Remark 14]{Min}). For this, \cite[Theorem 3.5.1]{BGP} implies that any Green's operator for $P$ restricts to the one for $P|_U$. Thus global hyperbolicity may be relaxed to strong causality if we are given a Green's operator on $M$ in advance.
\end{rem}
\begin{proof}
	By Proposition \ref{Expwithz}, setting
	\[L(s,z):=w^*(A(G_{P-z,x}))[f_s]\]
	and 
	\[W_{l,n}:=\frac{\pi^\frac{2-d}{2}n!}{4^{l}l!(2n)!}\left(AV^{l}_x\circ w\right)^{(2n)}(0)\]
	satisfies the conditions assumed for this chapter. We may thus apply Theorem \ref{Wfinal} (and bring the prefactor in $W_{K,0}$ to the right hand side) to obtain the desired result.
\end{proof}
We can also integrate this expansion over any submanifold, using a cut-off in case it is non-compact:
\begin{thm}
	\label{smearedfinalgen}
	Let $N\subseteq M$ be a submanifold and $\chi\in C_c^\infty(M)$. Let $A:E\boxtimes E^*\rightarrow \C$ be an arbitrary bundle coordinate and $A_x:=A|_{E\otimes E_x^*}$. Let $(w_x)_{x\in N}$ be a family of timelike unit speed geodesics defined on some interval $I$ containing zero such that $w_x(0)=x$ and $(x,t)\mapsto w_x(t)$ is smooth. Then we have
	\[\int\limits_N \chi(x)A_xV^{K,U}_x(x)dx=\Xi_{K,o}\Big(\int\limits_N\chi(x)w_x^*(A_x(G_{P-z,x}))[f_s]dx\Big)_{s,z},\]
	with $\Xi$ as defined in Definition \ref{dfXi}.
\end{thm}
\begin{proof}
	Set 
	\[L(s,z):=\int\limits_N\chi(x)w_x^*(A_x(G_{P-z,x}))[f_s]dx\]
	and 
	\[W_{l,n}:=\int\limits_N\chi(x)\frac{\pi^\frac{2-d}{2}n!}{4^{l}l!(2n)!}\left(A_xV^{l}_x\circ w_x\right)^{(2n)}(0)dx.\]
	By Theorem \ref{smearedexp}, this satisfies the assumptions for this chapter, so we may apply Theorem \ref{Wfinal} which yields the desired result.
\end{proof}
The formula for the Hadamard coefficients can be simplified in various ways by choosing specific values for $o$ for which $\Xi_{K,o}$ takes a simpler form.
\begin{thm}
	\label{finalspec}
	Assume that $\mathcal{L}$ is a function in two variables with an iterated asymptotic expansion. Let $K\in \N$.
	\begin{enumerate}
		\item We have \begin{align*} 
			\Xi_{K,0}(\mathcal{L})=&\sum\limits_{m=0}^K \frac{\pi^{\frac{d}{2}-1}4^{K+m}m!K!}{\M'(f)({2K+2m-d+3})}\binom{\frac{d}{2}-1-K}{m}\binom{2K+1-\frac{d}{2}}{K-m}\\&\cdot
			 \mathcal{L}[[2K+2m-d+3,m]]
		\end{align*}
		\item If $d$ is even, we have
		\begin{align*}
			\Xi_{K,\frac{d}{2}-1}(\mathcal{L})=&\sum\limits_{m=0}^K \frac{(4\pi)^{\frac{d}{2}-1}4^{K+m+\frac{d}{2}-1}(m+\frac{d}{2}-1)!K!}{\M'(f)({2K+2m+1})}\binom{-K}{m}\binom{2K}{K-m}\\
			&\cdot \mathcal{L}[[2K+2m+1,m+\tfrac{d}{2}-1]].
		\end{align*}
	\item If $d$ is even and $K<\frac{d}{2}$, we have
	\begin{align*}
		\Xi_{K,\frac{d}{2}-1-K}(\mathcal{L})=&\frac{(4\pi)^{\frac{d}{2}-1}(\frac{d}{2}-1-K)!K!}{\M'(f)({1})} \mathcal{L}[[1,\tfrac{d}{2}-1-K]]\\
	\end{align*}
	\end{enumerate}
\end{thm}
\begin{rem}
	The third formula is just a re-creation of the special-case-formula we considered before.
\end{rem}
\begin{proof}
	\begin{enumerate}
		The first two statements follow immediately by inserting the appropriate value of $o$. For the last one, we use that $\binom{0}{m}=\delta_{m0}$. Thus all summands except $m=0$ vanish and we obtain
		\begin{align*}
			\Xi_{K,\frac{d}{2}-1}(\mathcal{L})=&\sum\limits_{m=0}^K \frac{\pi^{\frac{d}{2}-1}4^{m+\frac{d}{2}-1}(m+\frac{d}{2}-1-K)!K!}{\M'(f)({2m+1})}\binom{0}{m}\binom{K}{K-m}\\
			&\cdot \mathcal{L}[[{2m+1},{m+\tfrac{d}{2}-1-K}]]\\
			&=\frac{(4\pi)^{\frac{d}{2}-1}(\frac{d}{2}-1-K)!K!}{\M'(f)({1})}  \mathcal{L}[[1,\tfrac{d}{2}-1-K]].\qedhere
		\end{align*}
	\end{enumerate}
\end{proof}

Instead of using Green's operators for $P-z$, one can also use powers of the Green's operators for $P$ in all the formulas above:
\begin{prop}
	Under the conditions of Theorem \ref{finalgen}, we have for any $m,j\in \N$:
	\[w^*(AG_{P-z,x})[f_s][[s^{2j+3-d}z^m]]=w^*(A(G_x^{+\ m+1}-G_x^{-\ m+1}))[f_s][[s^{2j+3-d}]].\]
\end{prop}
\begin{rem}
	Morally, the reason for this is that 
	\[G^\pm_{P-z}[[z^m]]=G^{\pm m+1}\]
	in a way that is compatible with all the operations involved.
\end{rem}
\begin{proof}
	By Theorem \ref{PowExp}, we have (using that $\binom{k}{n}$ vanishes for $k>n$)
	\begin{align*}
		A(G_x^{+\ m+1}-G_x^{-\ m+1})&\sim \insum{k} \binom{m+k}{m}AV^{k,U}_xR^U(2k+2m+2,x)\\
		&=\insum{k} \binom{k}{m}AV^{k-m,U}_xR^U(2k+2,x)
	\end{align*}
Note that $G^{\pm m}_x$ is supported in the future/past of $x$, as $G^{\pm 0}_x=\delta_x$ is supported at $x$ and $G^{\pm m+1}_x=G^\pm G^{\pm m}_x$ is supported in the future/past of $G^{\pm m}_x$.
We can thus use Proposition \ref{Wexp} to obtain
\[w^*(A(G_x^{+\ m+1}-G_x^{-\ m+1}))[f_s]\stackrel{s\rightarrow 0}{\sim}\insum{k,n}a(k,n)\binom{k}{m}(AV^{k-m}_x\circ w)^{(2n)}(0)s^{2k+2n+3-d},\]
for $a(k,n)$ as in Definition $\ref{dfa}$. We thus obtain
\[w^*(A(G_x^{+\ m+1}-G_x^{-\ m+1}))[f_s][[s^{2j+3-d}]]=\sum\limits_{k+n=j}a(k,n)\binom{k}{m}(AV^{k-m}_x\circ w)^{(2n)}(0).\]
On the other hand, as in Proposition \ref{Expwithz}, we obtain from Theorem \ref{intexp} and Proposition \ref{Vz}:
\[w^*(AG_{P-z,x})[f_s]\stackrel{s\rightarrow 0}{\sim} \insum{k,n}a(k,n)\sum\limits_{m'=0}^k\binom{k}{m'}z^{m'}(AV^{k-m'}_x\circ w)^{(2n)}(0)s^{2k+2n+3-d}.\]
Thus (using again that $\binom{k}{m'}$ vanishes for $k>m'$)
\begin{align*}
	w^*(AG_{P-z,x})[f_s][[s^{2j+3-d}]]&=\sum\limits_{n+k=j}\sum\limits_{m'=0}^ka(k,n)\binom{k}{m'}(AV^{k-m'}_x\circ w)^{(2n)}(0)z^{m'}\\
	&=\sum\limits_{m'=0}^j\sum\limits_{n+k=j}a(k,n)\binom{k}{m'}(AV^{k-m'}_x\circ w)^{(2n)}(0)z^{m'}
\end{align*}
and hence
\begin{align*}
	w^*(AG_{P-z,x})[f_s][[s^{2j+3-d}z^m]]&=\sum\limits_{n+k=j}a(k,n)\binom{k}{m}(AV^{k-m}_x\circ w)^{(2n)}(0)\\
	&=w^*(A(G_x^{+\ m+1}-G_x^{-\ m+1}))[f_s][[s^{2j+3-d}]]\qedhere
\end{align*}

\end{proof}
\begin{rem}
	\label{whatf}
	One might wonder what would have happened if we had not required $f$ to be odd, as this assumption was never really crucial to any part of the process.
	Without this assumption, we would have instead obtained
	\[w^*(AG_{P-z,x})[f_s]=L^{even}+L^{odd},\]
	where $L^{even}$ would be equal to the expression we obtained here for $f_{odd}$ instead of $f$ (i.e. with an asymptotic expansion involving even derivatives), while $L^{odd}$ would be a similar expression depending on the Mellin transform of $f_{even}$ and odd derivatives of the Hadamard coefficients instead. In the process of extracting the diagonal values, $L^{odd}$ would have been useless, as it contains no immediate information about the value of $V_K$ at the diagonal or its even derivatives. Thus in our inversion process, we would have had a completely independent matrix block to left-invert, which would not have influenced the inversion process for the first block. Overall, the end result we would have obtained for the (integrated) Hadamard coefficients would have been the same, except with $(f)_{odd}$ instead of $f$.
\end{rem}
\newpage
\chapter{Global formulations}
\label{globalchap}
\section{Global formulation in terms of Green's operators}
\label{globalsec1}
The formula for the Hadamard coefficients we have obtained thus far is still local in the sense that it depends on some geodesic $w$ and the restriction of the kernel of $G^\pm_{P-z}$ to that geodesic. This means it is not suitable for noncommutative settings, where geometric notions like geodesics no longer make sense a priori. In this chapter, we will investigate ways of making this at least partially global, so that it is defined only in terms of objects that could also exist on a noncommutative space. The latter condition is not really a rigorous one, as the notion of a Lorentzian noncommutative space is not yet fully formed. 

Our overall goal is to find a formula for 
\[\intM \chi(x) \tr(V^k_x(x))dx\]
where $\chi$ is some cut-off function. This corresponds to the global heat coefficients in Riemannian geometry, apart from the cut-off, which we need to introduce as our manifold is non-compact. We want to express this integral in  terms of the trace of some operator related to $G^\pm_{P-z}$, basically by integrating what we have done before over the diagonal.

We will still need to specify some direction as ``time'', in order to have some timelines along which to integrate. In the local version, we chose a single curve $w$. For a global formula, we need to make this choice at every point, i.e. we need a family of curves $w_x$.
\begin{gn}
	\label{dflow}
	For the remainder of this chapter, let $(w_x)_{x\in M}$ be a family of timelike unit speed geodesics, defined on a intervals containing $0$, with $w_x(0)=x$ such that $\bigcup\limits_{x\in M}\{x\}\times \Dom(w_x)$ is open in $M\times \R$ and $(x,t)\mapsto w_x(t)$ is a smooth submersion.
	Let $\Phi$ denote the induced flow in the first coordinate of $M\times M$, i.e. \[\Dom(\Phi):=\{(y,x,t)\in M\times M\times\R|t\in \Dom(w_y)\}\] and
	\begin{align*}
		\Phi\colon \Dom(\Phi)&\rightarrow M\times M\\
		(y,x,t)&\mapsto(w_y(t),x).\\
	\end{align*}
\end{gn}
\begin{rem}
One example of this would be the flow curves of a timelike geodesic unit vector field. Vector fields also make sense as derivations on a noncommutative space, and, since connections also exist in that context, the condition of being geodesic can also be formulated there, so in this way one could probably define a special case of the above also in the noncommutative setting, though in general there is no guarantee that such a vector field exists. In case $d$ is even and we only care about Hadamard coefficients for $k<\frac{d}{2}$, we may  use the results of Section \ref{smallk} instead of those of Section \ref{allk} and drop the condition that timelines are geodesics.
\end{rem}

We have to work on $M\times M$, as we can no longer work with a fixed basepoint. We will need to define a lot of notation for the following propositions, mostly due to the fact that we cannot insert arguments into distributions and thus have to define everything in an argument-free way.
\begin{gn}
	\label{pronotation}
	For $s>0$, define
	\begin{align*}
		F_s\colon \Dom(\Phi)&\rightarrow\C\\
		F_s(y,x,t)&:=f(\tfrac{t}{s}).
	\end{align*}
	Let
	\[\pi\colon \Dom(\Phi)\rightarrow M\times M\]
	denote the canonical projection onto the first two factors and for $x\in M$,
	let 
	\[\iota_x\colon M\rightarrow M\times M\]
	denote the canonical inclusion defined by
	\[\iota_x(y):=(y,x).\]
	We shall sometimes need a cutoff to make sure we evaluate the flow only where it is defined: Fix $\chi_\Phi\in C^\infty(\Dom (\Phi))$ that is constantly $1$ in a neighborhood of $M\times M\times \{0\}$ and vanishes near the boundary of $\Dom(\Phi)$ in $M\times M\times \R$.\\
\end{gn}

For the following considerations we need to use some knowledge about the wavefront of the causal propagator (see \cite[Theorem 16]{St}):
\begin{thm}
	Identifying $T^*(M\times M)$ with $T^*M\times T^*M$, the wavefront of $\K(G)$ only contains covectors of the form $(v,v')$, where $v,v'\in T^*M$ are lightlike covectors.
\end{thm}
This implies the well-definedness of certain operations for $\K(G)$. In particular, we can restrict to a point in the second coordinate. Unsurprisingly, we end up with the $G_x$ we have used so far. We will need this to translate our previous results into the language of wavefront calculus.
\begin{lemma}
	\label{resGx}
	For any $x\in M$, and
	\[\Lambda:=\{v\in \Tdot M\mid v\text{ is lightlike}\},\]
	the map
	\[\iota_x^*\colon \D'_{\WF(\K(G))}(E\boxtimes E)\rightarrow \D'_{\Lambda}(E)\]
	is continuous.
	we have
	\[G_x=\iota_x^*(\K(G)).\]
\end{lemma}
\begin{proof}
	Let $x\in M$. As we have
	\[d\iota_x^*(v,v')=v',\]
	$\iota_x^*(\WF(\K(G)))$ only contains lightlike covectors (and in particular no zero vectors). By wavefront calculus, this implies the first claim.
	Let $L\in  E_x$ and $\psi\in \Gamma_c(E^*)$.
	For disambiguation we denote the Dirac distribution in $\D'(E\otimes E_x^*)$ by $\delta_x$ and that in $\D'(M)$ by $\delta_x^0$.
	Let $(\delta_n)$ be a sequence of functions in $C_c^\infty(M)$ supported in a fixed compact set $K$ and converging to $\delta^0_x$ in $\D'(M)$. Let $(g_n)$ be a sequence in $\Gamma_c(E\boxtimes E^*)$ that converges to $\K(G)$ in $\D'_{\WF(\K(G))}(E\boxtimes E^*)$. Let $B$ be a basis of $E_x$ containing $L$. For $b\in B$, choose a section $\hat b\in\Gamma_c(E)$ with $\hat b(x)=b$ and denote by $b^*$ the element of the associated dual basis corresponding to $b$. The Dirac distribution takes the form
	\[\delta_x=\sum\limits_{b\in B}\delta_0\hat b\otimes b^*.\]
	With this, we can calculate
	\begin{align*}
		G_x[\psi\otimes L]&=G(\delta_x)[\psi\otimes L]\\
		&=\sum\limits_{b\in B} G(\delta^0_x\hat b\otimes b^*)[\psi\otimes L]\\
		&=\sum\limits_{b\in B} G(\delta^0_x\hat b)[\psi] b^*(L)\\
		&=G(\delta^0_x \hat L)[\psi]\\
		&=\lim\limits_{n\rightarrow \infty} G(\delta_n\hat L)[\psi]\\
		&=\lim\limits_{n\rightarrow \infty} \K(G)[\psi\otimes\delta_n\hat L]\\
		&=\lim\limits_{n\rightarrow \infty}(1 \otimes \delta_n)\K(G)[\psi\otimes\hat L]\\
	\end{align*}
	The tensor product
	\[C^\infty(M)\otimes\colon \D'(M)\rightarrow \D'_{\{0\}\times\Tdot (M)}(M\times M)\]
	is separately continuous and multiplication of distributions as a map
	\[\D'_{\{0\}\times\Tdot (M)}(M\times M)\times \D'_{\WF(\K(G))}(E\boxtimes E)\rightarrow \D'(E\boxtimes E)\]
	is separately continuous, as the wavefront of $\K(G)$ contains no vectors that vanish in the first component.
	 This allows us to calculate
	\begin{align*}
		G_x[\psi\otimes L]&=\lim\limits_{n\rightarrow \infty}(1\otimes\delta_n)\K(G)[\psi\otimes\hat L]\\
		&=(1 \otimes\delta_x^0)\K(G)[\psi\otimes\hat L]\\
		&=\lim\limits_{m\rightarrow \infty}(1\otimes\delta_x^0)g_m[\psi\otimes\hat L]\\
		&=\lim\limits_{m\rightarrow \infty}\lim\limits_{n\rightarrow \infty}(1\otimes\delta_n)g_m[\psi\otimes\hat L]\\
		&=\lim\limits_{m\rightarrow \infty}\lim\limits_{n\rightarrow \infty}\int\limits_M\delta_n(z)\int\limits_Mg_m(y,z)(\psi(y)\otimes\hat L(z))dydz\\
		&=\lim\limits_{m\rightarrow \infty}\int\limits_Mg_m(y,x)(\psi(y)\otimes\hat L(x))dy\\
		&=\lim\limits_{m\rightarrow \infty}\iota_x^*g_m[\psi\otimes L]\\
		&=\iota_x^*\K(G)[\psi\otimes L].
	\end{align*}
	As every element of $\Gamma_c((E\otimes E_x^*)^*)=\Gamma_c(E^*\otimes E_x)$ is a finite linear combinations of simple tensors like $\psi\otimes L$, this implies the claim.
\end{proof}
In order to make sense of time translation in our vector bundles, we will use parallel transport to identify fibres at different times. Identifying them in any other way would yield mostly the same result.\\
\begin{gn}
	\label{dftansport}
	For $t\in\Dom(w_y)$, let $\Pi_w(y,t)$ denote parallel transport from $E_{w_y(t)}$ to $E_{y}$ along $w_y$ (with respect to the connection associated to $P$).
	Define a bundle homomorphism
	\[\Pi_\Phi\colon \Phi^*(E\boxtimes E^*)\rightarrow \pi^*(E\boxtimes E^*)\]
	by setting, for $T\in \Phi^*(E\boxtimes E^*)_{(y,x,t)}$, 
	\[\Pi_\Phi(T):=\Pi_w(y,t)\circ T.\] 
	For $\psi\in \Gamma(E)$ and $t\in\Dom(w_x)$, define time translation via
	\[(\tau_w(t)\psi)(x):=\Pi_w(x,t)\psi(w_x(t)).\]	
	Let $A$ be a bundle coordinate on $E\boxtimes E^*$ such that for any $t\in\Dom(w_x)$ and $T\in E_{w_x(t)}\otimes E^*_x$ we have
	$A(T)=\tr(\Pi_w(x,t)\circ T).$
	Let $A_x:=A\circ\iota_x$ be its restriction to $E\boxtimes E_x^*$.
\end{gn}
\begin{rem}
	The bundle coordinate $A$ is necessary due to the fact that $\tr(\Pi_w(x,t)\circ \cdot)$ cannot be applied to a distribution in $E\boxtimes E^*$ as it is only defined on a closed subset. However, eventually only the values on this subset will matter, so the way we extend is irrelevant. Such an extension can always be constructed in a smooth way by using a partition of unity.
\end{rem}

We are interested in the quantity that is morally given by
\[L_{s}(x,y)=\intR f(\tfrac{t}{s})\Pi_w(y,t)G(w_y(t),x)dt,\]
for $s$ sufficiently small. We particularly care about (morally)
\[tr(L_s(x,x))=\intR f(\tfrac{t}{s})\tr(\Pi_w(x,t)G(w_x(t),x))dt,\]
which we can relate to the trace of the Hadamard coefficients at the diagonal.
There are various ways to make this rigorous. In general, we need to include our cut-off $\chi_\Phi$ to make sure this is well-defined for all $s$. At each individual point, however, if $s$ is small enough the cut-off doesn't impact the result. We need different ways to write the above in a rigorous fashion:
\begin{itemize}
\item Writing
\[L_{s}(x,y)= \pi_* F_s\chi_\Phi\Pi_\Phi\Phi^*(\K(G))\]
will let us conclude by wavefront calculus that this is actually a smooth function.

\item Instead expressing the trace at the diagonal as
\[\tr(L_s(x,x))=w_x^*(A_xG_x)[\chi_\Phi(x,x,\cdot) f_s]\]
will allow us to relate this to the Hadamard coefficients via the results of the previous chapter.

\item Finally, on compact sets and for $s$ small enough such that the flow is well-defined without cut-off, the equality
\[L_s(x,y)=\K\left(\intR f(\tfrac{t}{s})\tau_w(t)Gdt\right)(x,y)\]
will allow us to relate the integral $\tr(L_s(x,x))$ to the trace of that operator (both multiplied with suitable cut-offs to make the integral convergent and the operator trace class).
\end{itemize}
We proceed to show that all three ways of making $L_s$ rigorous are well-defined and mutually compatible. The general strategy will always be the same: We use wavefront calculus to show that the expressions, as functions of $\K(G)$, are well-defined and continuous on distributions with wavefront set contained in that of $\K(G)$, then show that they are compatible on smooth functions to conclude that, by continuity, they also agree for $\K(G)$.

As a first step toward the first expression, we note that integration along timelines can be expressed as a pushforward by the corresponding projection:
\begin{lemma}
	\label{intproj}
	If F is a vector bundle on some manifold $N$, $U\subseteq N\times \R$ is open, $p\colon U\times \R$ denotes the projection onto the first component and $g\in \Gamma(p^*(E))$ has compact support in each fibre of $p$, then we have
	\[p_*g(x)=\intR g(x,t)dt.\]
\end{lemma}
\begin{proof}
	Both sides are smooth and for $\psi\in C_c^\infty(p(U))$ we have
	\begin{align*}
		p_*g[\psi]&=g[\psi\circ p]\\
		&=\int\limits_{p(U)}\int\limits_\R g(x,t)\psi(x)dtdx\\
		&=\int\limits_\R g(\cdot,t)dt[\psi].\\
	\end{align*}
As $\psi$ was arbitrary, this implies the claim.
\end{proof}
Thus, for a function $g$ on $M\times M$, integration along flow lines can be written as
\[\intR f(\tfrac{t}{s}) g(w_y(t),x) dt=\pi_*F_s\Phi^*(g)(y,x).\]
The right hand side also makes sense if $g$ is a distribution with suitable wavefront set. This shows that our first expression for $L_s(x,y)$ is well-defined.
\begin{prop}
	For any $s>0$, $g\mapsto \pi_*F_s\chi_\Phi\Pi_\Phi\Phi^*(g)$ gives a well defined and continuous map from $\D'_{\WF(\K(G))}(E\boxtimes E^*)$ to $\Gamma(E\boxtimes E^*)$.
	\label{smoothforward}
\end{prop}
\begin{proof}
	Multiplication with smooth functions and composition with bundle homomorphisms over the identity continuously map each wavefront space to itself. As $\Phi$ is a submersion and $\pi$ is proper on $\supp(\chi_\Phi)$, wavefront calculus tells us that $\pi_*F_s\chi_\Phi\Pi_\Phi\Phi^*$ is well defined and continuous as a map into $\D'_{\pi_*\Phi^*(\WF(\K(G)))}(E\boxtimes E^*)$. We will show that $\pi_*\Phi^*(\WF(\K(G)))$ is empty. Let $(y,x,t)\in \Dom(\Phi)$ be arbitrary, let $\xi_0=(v,v')$ be in $\WF(\K(G))\cap T_{(w_y(t),x)}(M\times M)$ and let $\partial_t$ denote the the unit vector in $\R$-direction at $(y,x,t)$. Then we have
	\begin{align*}
		d\Phi_{(x,y,t)}^*(\xi_0)(\partial_t)&=\xi_0(d\Phi_{(y,x,t)}\partial_t)\\
		&=\xi_0(\partial_t\Phi(y,x,t))\\
		&=\xi_0(\partial_t(w_y(t),x)\\
		&=\xi_0((w_y'(t),0))\\
		&=(v(w_y'(t)),0).
	\end{align*}
	As $v$ must be lightlike and  $w_y'$ is timelike, this is non-zero by \ref{orthotime}. We can conclude that all covectors in $\Phi^*(\WF(\K(G)))$ are non-zero when paired with $\partial_t$. For any $\xi\in T_{(x,y)}(M\times M)$, we have
	\[d\pi_{(x,y,t)}^*\xi(\partial_t)=\xi(\partial_t\pi(x,y,t))=\xi(0)=0.\]
	Thus $d\pi_{(x,y,t)}^*\xi\notin \Phi^*(\WF(A\K(G)))$. As $\xi$ was arbitrary, $\pi_*\Phi^*(\WF(A\K(G))$ is empty and hence $\pi_*F_s\chi_\Phi\Pi_\Phi\Phi^*$ maps $\D'_{\WF(\K(G))}(E\boxtimes E^*)$ into $\Gamma(E\boxtimes E^*)$ continuously. 
\end{proof}
With this we can define
\begin{df}
	\label{dfL}
	For $s\in\R$, define $L_s\in\Gamma(E\boxtimes E)$ by
	\[L_s:=\pi_*F_s\chi_\Phi\Pi_\Phi\Phi^*(\K(G)).\]
\end{df}
The above proposition implies that this is well-defined and smooth.

We now proceed to show that this coincides with the pointwise picture we have been using in the previous chapters:
\begin{prop}
	\label{pointres}
	For arbitrary $x\in M$, set 
	\[\chi_{x}(t):=\chi_\Phi(x,x,t).\]
	Then we have:
	\[\tr(L_s(x,x)):=\tr\pi_*F_s\chi_\Phi\Pi_\Phi\Phi^*(\K(G))(x,x)=w_x^*(A_x G_x)[\chi_{x}f_s].\]
\end{prop}
\begin{proof}
	Let $(x,x,t)$ be in $\Dom(\Phi)$. Let $\partial_t$ be the standard unit vector in $T^*_t(\R)$. We calculate for lightlike $v\in T^*_{w_x(t)}(M)$:
	\begin{align*}
		dw_x^*(v)(\partial_t)&=v(\partial_tw_x(t))\\
		&=v(w_x'(t))\\
		&\neq 0,
	\end{align*}
	as this is the pairing of a lightlike and a timelike vector. Thus $w_x^*(\Lambda)$, for $\Lambda$ as in Proposition \ref{resGx}, does not intersect the zero section. We can conclude that
	\[w_x^*\colon \D'_\Lambda(M)\rightarrow\D'(\R)\]
	is well-defined and continuous. Together with Proposition \ref{resGx}, this means that all maps in the right hand side of the claim are continuous. The left hand side is continuous as a function of $\K(G)$ by \ref{smoothforward}, as evaluation is continuous on $\Gamma(E\boxtimes E^*)$. We now calculate for $g\in \Gamma_c(E\boxtimes E^*)$:
	\begin{align*}
		\tr\pi_*F_s\chi_\Phi\Pi_\Phi\Phi^*(g)(x,x)&=\intR f(\tfrac{t}{s})\chi_\Phi(x,x,t)\tr(\Pi_w(x,t)\circ g(w_x(t),x))dt\\
		&=\intR f_s(t)\chi_{x}(t)A_x(\iota_x^*g(w_x(t)))dt\\
		&=w_y^*(A_x\iota_x^*(g))[\chi_{x}f_s].
	\end{align*}
	As both sides are continuous as functions of $g$ in $\D'_{\WF(\K(G))}(M\times M)$, the equality also holds for $\K(G)$ instead of $g$. As $G_x=\iota_x^*\K(G)$, this finishes the proof. 
\end{proof}
For the last expression, we need to establish continuity of operator evaluation in wavefront calculus.
\begin{df}
	\label{dfOp}
	For $\eta\in \D'(F\boxtimes E^*)$, denote by $\Op(\eta)$ the unique continuous  operator $\Gamma_c(E)\rightarrow \D'(F)$ with Schwartz kernel $\eta$.
\end{df}
\begin{lemma}
	\label{Opcont}
	For $\psi\in \Gamma_c(E)$, the map $g\mapsto \Op(g)\psi$ maps $\D'_{\WF(G)}(E\boxtimes E^*)$ to $\Gamma(E)$ continuously.
\end{lemma}
\begin{proof}
	Let $p_1\colon M\times M\rightarrow M$ be the projection onto the first component. For $g\in \Gamma_c(E\boxtimes E^*)$ and $\theta\in \Gamma_c(E^*)$, we have
	\begin{align*}
		\Op(g)\psi[\theta]&=g[\theta\otimes\psi]\\
		&=(1\otimes\psi)g[\theta\otimes 1]\\
		&=(1\otimes\psi)g[\theta\circ p_1]\\
		&=p_{1*}((1\otimes\psi)g)[\theta].
	\end{align*}
	As both sides are continuous as a function of $g\in D'(E\boxtimes E^*)$ (multiplichation by $1\otimes\psi$ ensures that the support condition for the pushforward theorem in wavefront calculus are satisfied), the equality holds for arbitrary distributions $g$.
	In particular, we thus have for $g\in \D'_{\WF(G)}(E\boxtimes E^*)$
	\[\Op(g)\psi=p_{1*}((1\otimes\psi)g).\]
	For $x,y\in M$ and $v\in T_xM$, we have $(dp_1)^*_{(x,y)} v=(v,0)$.
	As $\WF(G)$ does not contain vectors that vanish in the second component it does not intersect the range of $(dp_1)^*_{(x,y)}$. Thus $p_{1*}(\WF(G))$ is empty, and hence
	\[g\mapsto p_{1*}((1\otimes\psi)g)\]
	maps $\D'_{\WF(G)}(E\boxtimes E^*)$ to $\Gamma(E)$ continuously.
\end{proof}
We now get to the last expression for $L_s$. In order to take integrals/traces later on, we need to multiply by cut-off functions on $M$. This also allows us to get rid of our flow-related cut-off $\chi_\Phi$, if we choose $s$ small enough.
\begin{df}
	For a smooth function $\chi$, let $\mu_{\chi}$ denote multiplication by $\chi$. For a family of Operators $(T(t))_{t\in\R}$ on some function space, define the integrated operator by
	\[\intR T(t)dt\psi:=\intR T(t)\psi dt\]
	with the integral taken pointwise (whenever this is defined).
\end{df}
\begin{lemma}
	\label{kersmooth}
	For $\chi_1,\chi_2\in C_c^\infty(M)$, there is $s_0>0$ such that for $s<s_0$, the operator 
	\[\int\limits_\R\mu_{\chi_1}f(\tfrac{t}{s})\tau_w(t)G\mu_{\chi_2} dt\]
	is well defined and has a smooth Schartz kernel given by
	\[\K\left(\int\limits_\R\mu_{\chi_1}f(\tfrac{t}{s})\tau_w(t)G\mu_{\chi_2} dt\right)=(\chi_1\otimes\chi_2)L_s.\]	
	If $\Dom(\Phi)=M\times M\times\R$  and $\chi_\Phi=1$ this holds for arbitrary $s$ and we may also allow $\chi_1=\chi_2=1$.
\end{lemma}
\begin{proof}
	As $\supp(\chi_1)\times \supp(\chi_2)$ is compact and $\chi_\Phi^{-1}(1)$ is a neighborhood of $M\times M\times \{0\}$, there is $s_0$ such that
	\[\supp(\chi_1)\times \supp(\chi_2)\times s_0I_f\subseteq \chi_\Phi^{-1}(1)\subseteq\Dom(\Phi).\]
	Thus for $s<s_0$,
	\[(\chi_1\otimes\chi_2)\pi_*F_s\chi_\Phi(\cdot)=\pi_*(\chi_1\otimes\chi_2\otimes f_s)(\cdot)\]
	and $\chi_1\otimes f_s$ has compact support contained in the domain of the flow, i.e. the region where $\tau_w(\cdot)G(\chi_2\psi)$ is defined for any test function $\psi$. This means the operator under consideration is well-defined. If the flow is defined for all times and $\chi_\Phi=1$, we may instead consider arbitrary $s$.
	
	We need to define quantities analogous to $\Phi$, $\pi$, $F_s$ and $\Pi_\Phi$, only without the second argument.
	Define 
	\[\phi(x,t):=w_x(t)\]
	for all $(x,t)$ where this is defined. 
	Let $\tilde\pi\colon \Dom(\phi)\rightarrow M$ denote the projection on the first coordinate.
	Finally, define	
	\[\Pi_\phi\colon \phi^*(E)\rightarrow \tilde\pi^*(E)\]
	by setting for $v\in \phi^*(E)_{(x,t)}$
	\[\Pi_\phi(v):=\Pi_w(x,t)(v).\]
	With this we can write rewrite our time translation (for $h\in \Gamma(E)$) as
	\[\tau_w(t)(h)(x)=\Pi_\phi\phi^*(h)(x,t).\]
	Using this and Lemma \ref{intproj}, we get
	\begin{align*}
		&\K\left(\intR \mu_{\chi_1}f(\tfrac{t}{s})\tau_w(t)\Op(g)\mu_{\chi_2}\right)(\psi_1\otimes\psi_2)\\
		&=\intM\intR \chi_1(y)f_s(t)(\Pi_\phi\phi^*\Op(g)(\chi_2\psi_2))(y,t)dt\psi_1(y)dy\\
		&=\intM\tilde\pi_*( (\chi_1\otimes f_s)(\Pi_\phi\phi^*\Op(g)(\chi_2\psi_2)))(y)\psi_1(y)dy\\
		&=(\tilde\pi_*((\chi_1\otimes f_s)(\Pi_\phi\phi^*\Op(g)(\chi_2\psi_2))))[\psi_1].\\
	\end{align*}
	By Lemma \ref{Opcont}, $\Op(\cdot)[\chi_2\phi_2]$ maps $\D'_{\WF(G)}(E\boxtimes E^*)$ to $\Gamma(E)$ continuously. All other maps involved are continuous maps on smooth sections (multiplication with $\chi_1\otimes f_s$ guaranteeing that the support conditions for the pullback are met). Thus the above is continuous as a function of $g$.
	
	Since 
 \[g\mapsto (\chi_1\otimes\chi_2)\pi_*F_s\chi_\Phi\Pi_\Phi\Phi^*(g)\]
 is also continuous, we only need to show the desired equality for $g\in \Gamma_c(E\boxtimes E^*)$. In that case, we have for $\psi_1\in \Gamma_c(E^*)$ and $\psi_2\in \Gamma_c(E)$:
 \begin{align*}
 	&(\chi_1\otimes\chi_2)L_s[\psi_1\otimes\psi_2]\\
 	&=(\chi_1\otimes\chi_2)\pi_*F_s\chi_\Phi\Pi_\Phi\Phi^*(g)[\psi_1\otimes\psi_2]\\
 	&=\pi_*(\chi_1\otimes\chi_2\otimes f_s)\Pi_\Phi\Phi^*(g)[\psi_1\otimes\psi_2]\\
 	&=(\chi_1\otimes\chi_2\otimes f_s)\Pi_\Phi (g\circ\Phi)[(\psi_1\otimes\psi_2)\circ\pi]\\
 	&=\intR\intM\intM f(\tfrac{t}{s})(\chi_1(y)\psi_1(y))(\Pi_w(y,t)g(w_y(t),x)(\chi_2(x)\psi_2(x))dxdydt\\
 	&=\intM\psi_1(y)\intR f(\tfrac{t}{s})\chi_1(y) \Pi_w(y,t)\Op(g)(\chi_2\psi_2)(w_y(t))dtdy\\
 	&=\intM \psi_1(y)\intR f(\tfrac{t}{s})\chi_1(y)\tau_w(t)(\Op(g)(\chi_2\psi_2))(y)dtdy\\
 	&=\K\left(\intR f(\tfrac{t}{s}) \mu_{\chi_1}\tau_w(t)\Op(g)\mu_{\chi_2}dt\right)(\psi_1\otimes\psi_2).
 \end{align*}
By continuity, this also holds for $g=\K(G)$, proving the desired equality. Smoothness of the kernel then follows from Proposition \ref{smoothforward}.
\end{proof}
We can now show the global version of our main formula (for readers who wish to understand this result without reading all the previous parts, we point out that there is a notation index at the end of the thesis):
\begin{thm}
	\label{globalexp}
	For $\chi_1,\chi_2\in C_c^\infty(M)$, there is $s_0>0$, such that for $s<s_0$, the operator \[\int\limits_\R\mu_{\chi_1}f(\tfrac{t}{s})\tau_w(t)G\mu_{\chi_2} dt\]
	extends continuously to a trace class operator in $L^2(E)$ (with respect to any hermitian structure on $E$). Its trace has the asymptotic expansion
	\begin{align*}
		&\tr\left( \int\limits_\R\mu_{\chi_1}f(\tfrac{t}{s})\tau_w(t)G\mu_{\chi_2} dt\right)\\
		\stackrel{s\rightarrow 0}{\sim}&\insum{k}\insum{n}a(k,n)
		\intM \chi_1(x)\chi_2(x)\tr(\partial_t^{2n}|_{t=0}(\tau_w(t)V^k_x)(x))dx\ 
		s^{2k+2n+3-d},
\end{align*}
for $a(k,n)$ as in Definition \ref{dfa}.

We have for any $k,o\in \N$:
\[\intM \chi_1(x)\chi_2(x)\tr(V^{k}_x(x))dx=\Xi_{k,o}\left(\tr\left( \intR\mu_{\chi_1}f(\tfrac{t}{s})\tau_w(t)G_{P-z}\mu_{\chi_2} dt\right)\right)_{s,z}\]
with $\Xi$ as in Definition \ref{dfXi}.
\end{thm}
\begin{rem}
	The perfect analogue of the global heat coefficients would be 
	\[\frac{1}{k!}\intM \tr(V^k_x(x))dx.\]
	However, as $M$ is non-compact this generally does not exist and the best one can hope to obtain is an integral against some cut-off function $\chi$. Note that any cutoff can be written as a product $\chi=\chi_1\chi_2$.
\end{rem}
\begin{proof}
	Choose $s_0$ such that Lemma \ref{kersmooth} applies and such that \[\supp(\chi_1)\times\supp(\chi_1)\times s_0I_f\subseteq \chi_\Phi^{-1}(1).\]
	In the following, assume $s<s_0$.
	
	By Lemma \ref{kersmooth}, the operator
	\[\int\limits_\R\mu_{\chi_1}f(\tfrac{t}{s})\tau_w(t)G\mu_{\chi_2} dt\]
	has smooth compactly supported kernel. Thus it extends to a trace class operator on $L^2$-sections.

By Lemma \ref{kersmooth} and Proposition \ref{pointres}, using that $\chi_\Phi(x,x,t)f_s(t)=f_s(t)$ for $x\in \supp(\chi_1)$ by our choice of $s_0$,  we have for any $x\in M$
\[\tr\left(\K\left(\int\limits_\R\mu_{\chi_1}f(\tfrac{t}{s})\tau_w(t)G\mu_{\chi_2} dt\right)(x,x)\right)=\chi_1(x)\chi_2(x)w_x^*(A_x G_x)[f_s]\]

Putting everything together and using Theorem \ref{Mercer} to express the trace of an operator as the integral over its kernel, we get
\begin{align*}
	&\tr \left(\int\limits_\R\mu_{\chi_1}f(\tfrac{t}{s})\tau_w(t)G\mu_{\chi_2} dt\right)\\
	&=\intM\tr\left(\K\left(\int\limits_\R\mu_{\chi_1}f(\tfrac{t}{s})\tau_w(t)G\mu_{\chi_2} dt\right)(x,x)\right)dx\\
	&=\intM \chi_1(x)\chi_2(x)w_x^*(A_x G_x)[f_s]dx\\
\end{align*}
using Theorem \ref{smearedexp}, we obtain
\begin{align*}
&\tr \left(\int\limits_\R\mu_{\chi_1}f(\tfrac{t}{s})\tau_w(t)G\mu_{\chi_2} dt\right)\\
&\stackrel{s\rightarrow 0}{\sim}\insum{k}\insum{n}a(k,n)\intM \chi_1(x)\chi_2(x)\left(A_xV^{k,U}_x\circ w_x\right)^{(2n)}(0)dx\ 
s^{2k+2n+3-d}\\
&=\insum{k}\insum{n}a(k,n)\intM \chi_1(x)\chi_2(x)\tr(\partial_t^{2n}|_{t=0}(\tau_w(t)V^k_x)(x))dx\ 
s^{2k+2n+3-d}.
\end{align*}
Applying Theorem \ref{smearedfinalgen} instead and using the previous equality for $P-z$ instead of $P$, we obtain
	\begin{align*}
	&\intM \chi_1(x)\chi_2(x)\tr(V^{k}_x(x))dx\\
	&=\intM \chi_1(x)\chi_2(x)A_x(V^k_x(x))dx\\
	&=\Xi_{k,o}\Big(\int\limits_\R\chi_1(x)\chi_2(x)w_x^*(A_x(G_{P-z,x}))[f_s]dx\Big)_{s,z}\\
	&=\Xi_{k,o}\Big(\tr \Big(\int\limits_\R\mu_{\chi_1}f(\tfrac{t}{s})\tau_w(t)G_{P-z}\mu_{\chi_2} dt\Big)dx\Big)_{s,z}.\qedhere
	\end{align*}
\end{proof}

All the simplifications for special choices of $o$ can be done in the global case just like in the local case and we obtain analogous formulas for the integrated scalar curvature.
\section{Spacelike global formulation in terms of evolution operators}
\label{globalsec2}
One approach to creating a notion of noncommutative spacetimes would be to make space noncommutative while keeping time to be $\R$. This is motivated by the fact that every globally hyperbolic manifold is diffeomorphic (but generally not isometric!) to $\Sigma\times \R$. Thus one simple notion of a noncommutative globally hyperbolic spacetime would be to take families of spectral triples $(A, \pi, D(t))$ for $t\in \R$, where the algebras, Hilbert spaces and representations are independent of $t$, while the Dirac operators depend on time in some "smooth" way (to capture arbitrary globally hyperbolic spacetimes, one would also have to include a lapse function $\R\rightarrow A$, but we will restrict to cases where this is one, as we need the timelines to be geodesics). In such a setting, it makes sense to talk about specific times, while all "spacelike" quantities only make sense globally.

For this setting, we want to establish a version of our main theorem that is formulated in terms of evolution operators instead of Green's operators, and where we integrate over space while keeping things local in time. Some versions of noncommutative Lorentzian geometry that work in a setting with commutative time use the evolution operators to identify different time-slices (see e.g. \cite{Ha}). Formulated in this framing, we get a formula in terms of the time shift with respect to the chosen timelines.

\begin{gn}
	\label{foliate}
	From now on assume that $M=\Sigma\times \R$ such that all $\Sigma_t:=\Sigma\times \{t\}$ are spacelike Cauchy hypersurfaces. Assume furthermore that $w_{(x,r)}(t)=(x,r+t)$ are timelike unit speed geodesics and use the notation of the previous section for this family of timelines.
\end{gn}
\begin{rem}
 Not every globally hyperbolic spacetime can be brought to the form above. For example, every spacetime with finite time is excluded. If we assume furthermore that timelines are orthogonal to the $\Sigma_t$, which would be required to make sense of the evolution operators in terms of time derivatives, the condition that timelines are geodesics is equivalent to a constant lapse function.
\end{rem}
\begin{df}
	\label{Qtransport}
Let $Q(t,r)$ denote the evolution operator, as defined in Definition \ref{dfQ}.
Let $\Pi(r,t)$ denote parrallel transport along timelines from $\Sigma_t$ to $\Sigma_r$ and for $\psi\in \Gamma(E|_{\Sigma_t})$, let
\[(\tau(r,t)\psi)(x,r):=\Pi(r,t)\psi(x,t).\]
\end{df}
We first formulate the Schwartz kernel we looked at before in terms of the evolution operator, by using the connection between the latter and the causal propagator (Proposition \ref{GtoQ}).
\begin{lemma}
	\label{kerintQ}
	For any $r\in \R$, the operator 
	\[\intR f(\tfrac{t-r}{s})\tau(r,t)Q(t,r)dt\]
	has Schwartz kernel given by
	\[\K\Big(\intR f(\tfrac{t-r}{s})\tau(r,t)Q(t,r)dt\Big)=\K\Big(\intR f(\tfrac{t-r}{s})\tau_w(t-r)Gdt\Big)\Big|_{\Sigma_r\times\Sigma_r}.\]
\end{lemma}
\begin{proof}
	Let $\iota_r$ denote the inclusion of $\Sigma_r$ into $M$ and let $\iota_{r,r}=\iota_r\times\iota_r$ denote the inclusion of $\Sigma_r\times \Sigma_r$ into $M\times M$. 
	Let $\psi\in \Gamma_c(E|_{\Sigma_r})$ and $\rho\in \Gamma_c(E^*|_{\Sigma_r})$ be arbitrary. 
	Abbreviate
	\[K:= \K\Big(\intR f(\tfrac{t-r}{s})\tau_w(t-r)Gdt\Big).\]
	As $K$ is smooth (\ref{kersmooth}), it may be paired with arbitrary distributions.
	 By Proposition \ref{GtoQ}, we have
	\begin{align*}
		&\rho\left[\intR f(\tfrac{t-r}{s})\tau(r,t)Q(t,r)\psi dt\right]\\
		&=\rho\left[\intR f(\tfrac{t-r}{s})\tau(r,t)(G(\iota_{r*}\psi)|_{\Sigma_r}) dt\right]\\
		&=\rho\left[\intR f(\tfrac{t-r}{s})\tau_w(t-r)G(\iota_{r*}\psi)dt|_{\Sigma_r} \right]\\
		&=(\iota_{r*}\rho)\left[\intR f(\tfrac{t-r}{s})\tau_w(t-r)G(\iota_{r*}\psi)dt\right]\\
		&=K[(\iota_{r*}\rho)\otimes (\iota_{r*}\psi)]\\
		&=K[\iota_{r,r*}(\rho\otimes\psi)]\\
		&=\iota_{r,r*}(\rho\otimes\psi)[K]\\
		&=\rho\otimes\psi[K|_{\Sigma_r\times\Sigma_r}]\\
		&=K|_{\Sigma_r\times\Sigma_r}[\rho\otimes\psi]
	\end{align*}
	Thus $K|_{\Sigma_r\times\Sigma_r}$ is the Schwartz kernel of 
	\[\intR f(\tfrac{t-r}{s})\tau(r,t)Q(t,r) dt.\qedhere\]
\end{proof}
We can use this to formulate our final result in terms of evolution operators instead of Green's operators:
\begin{thm}
	For any $r\in\R$ and $\chi_1,\chi_2\in C_c^\infty(\Sigma_r)$, the operator 
	\[\mu_{\chi_1}\intR f(\tfrac{t-r}{s})\tau(r,t)Q(t,r)dt\mu_{\chi_2}\]
	is trace class and its trace has the asymptotic expansion
	\begin{align*}
		&\tr\Big(\mu_{\chi_1}\intR f(\tfrac{t-r}{s})\tau(r,t)Q(t,r) dt\mu_{\chi_2}\Big)\\
		&\stackrel{s\rightarrow 0}{\sim}\insum{k}\insum{n}a(k,n)\int\limits_{\Sigma_r}\chi_1(x)\chi_2(x)\tr(\partial_t^{2n}|_{t=0}(\tau_w(t)V^k_x)(x))dx\ s^{2k+2n+3-d},
	\end{align*}
	with $a(k,n)$ as in Definition \ref{dfa}.
	
	Let $Q_{P-z}(t,r)$ denote the evolution operator associated to $P-z$ instead of $P$.
	We have for any $k,o\in \N$
	\[\int\limits_{\Sigma_r}\chi_1(x)\chi_2(x)\tr(V^k_x(x))dx=\Xi_{k,o}\Big(\tr\Big(\mu_{\chi_1}\intR f(\tfrac{t-r}{s})\tau(r,t)Q_{P-z}(t,r) dt\mu_{\chi_2}\Big)\Big)_{s,z},\]
	with $\Xi$ as in Definition \ref{dfXi}. 
\end{thm}
\begin{rem}
	If $\Sigma_r$ is compact, one may omit the cutoffs $\chi_1$ and $\chi_2$  to make the result look a little simpler.
\end{rem}
\begin{proof}
	By Lemma \ref{kerintQ} and Lemma \ref{kersmooth}, the operator has smooth Schwartz kernel. Due to the cut-off functions, the kernel is also compactly supported, so the operator is trace class by \ref{Mercer}.
	These Lemmas and Proposition \ref{pointres} also yield
	\begin{align*}
		&\tr\Big(\mu_{\chi_1}\intR f(\tfrac{t-r}{s})\tau(r,t)Q(t,r) dt\mu_{\chi_2}\Big)\\
		&=\int\limits_{\Sigma_r}\chi_1(x)\chi_2(x)\K\Big(\intR f(\tfrac{t-r}{s})\tau_w(t-r)Gdt\Big)(x,x)dx\\
		&=\int\limits_{\Sigma_r}\chi_1(x)\chi_2(x)w_x^*(A_xG_{x})[f_s]dx.
	\end{align*}
 Thus by Theorem \ref{smearedexp}, we have
 \begin{align*}
 	&\tr\Big(\mu_{\chi_1}\intR f(\tfrac{t-r}{s})\tau(r,t)Q(t,r) dt\mu_{\chi_2}\Big)\\
 	&\stackrel{s\rightarrow 0}{\sim}\insum{k}\insum{n}a(k,n)\int\limits_{\Sigma_r}\chi_1(x)\chi_2(x)(A_xV^k_x\circ w_x)^{(2n)}(0)dx\ s^{2k+2n+3-d}\\
 	&=\insum{k}\insum{n}a(k,n)\int\limits_{\Sigma_r}\chi_1(x)\chi_2(x)\tr(\partial_t^{2n}|_{t=0}(\tau_w(t)V^k_x)(x))dx\ s^{2k+2n+3-d}.
 	\end{align*}
	Replacing $P$ with $P-z$, we get
	\begin{align*}
		&\tr\Big(\mu_{\chi_1}\intR f(\tfrac{t-r}{s})\tau(r,t)Q_{P-z}(t,r) dt\mu_{\chi_2}\Big)\\
		&=\int\limits_{\Sigma_r}\chi_1(x)\chi_2(x)w_x^*(A_xG_{P-z,x})[f_s]dx.
	\end{align*}
	 Inserting this into Theorem \ref{smearedfinalgen} yields
	\begin{align*}
		\int\limits_{\Sigma_r}\chi_1(x)\chi_2(x)\tr(V^k_x(x))dx&=\Xi_{k,o}\Big(\tr\Big(\mu_{\chi_1}\intR f(\tfrac{t-r}{s})\tau(r,t)Q_{P-z}(t,r) dt\mu_{\chi_2}\Big)\Big)_{s,z}.\qedhere
	\end{align*}

\end{proof} 
\newpage
\chapter*{Acknowledgements}
\addcontentsline{toc}{chapter}{Acknowledgements}
I would like to thank my advisor, Matthias Lesch, for his advice and support in writing this thesis. In particular, his suggestion of considering something like a wave trace proved a valuable starting point for the considerations of this thesis. 

I am especially grateful to Koen van den Dungen, who, even though he was not formally my advisor, acted like one and supported me throughout the writing process. I am particularly grateful for his suggestion of the topic and his recommendation of the paper \cite{DaWr}, which helped me overcome the biggest obstacle of my thesis.

I would like to thank Michal Wrochna for teaching me more about the context of my thesis in Lorentzian geometry and QFT on curved spacetimes.

Finally, I would like to thank Christoph Brinkmann for the inspiring talks we had during the creation of this thesis. 

I also thank all of the above people for proofreading parts or all of my thesis and thus helping me reduce the number of mistakes the reader might still encounter, as well as pointing out other possible improvements.
\newpage
\chapter*{Appendix: wavefront calculus on manifolds}
\addcontentsline{toc}{chapter}{Appendix: wavefront calculus on manifolds}
\label{append}
In this appendix, we want to prove Theorem \ref{WFMFD}, i.e. show that the results of wavefront calculus on $\R^d$ carry over to vector bundles over manifolds.

	In the following we omit the restriction to chart domains when composing with charts, and automatically restrict compositions involving charts to the sets where they are defined (which may be empty).
	We work in the setting of Theorem \ref{WFMFD}, i.e. we have vector bundles $(E_l)_{l\leq 3}$ over a manifold $X$, a vector bundle $F$ over a manifold $Y$ and closed conic subsets $\Lambda, \Lambda'\subseteq \Tdot(X)$ and $\Theta\subseteq \Tdot(Y)$.
	
	Before we proceed with the claims of Theorem \ref{WFMFD}, we need to establish that it is sufficient to consider only wavefront norms for a fixed open cover of charts, with local coordinate frames: 
	
	Define $E_4:=f^*(F)$.
	Choose charts $(\psi_i)_{\i\in I}$ of $X$ whose domains $(U_i)_{i\in I}$ form a locally finite cover of $X$ such that all relevant vector bundles are trivial over all $U_i$ and let $(\chi_i)_{i\in I}$ be a partition of unity subordinate to these. Let $(\psi'_i)_{i\in I'}$, $(U'_i)_{i\in I'}$ and $(\chi'_i)_{i\in I'}$ be the same for $Y$. For each $i\in I$ and $l\leq 4$, let $(A^l_{ij})_{j\leq \rk(E_l)}$ be a frame of $E_l^*|_{U_i}$, let $(A^{l*}_{ij})_{j\leq \rk(E_l)}$ denote the corresponding dual frame. Let $(B_{ij})_{j\leq \rk(F)}$ be a frame for $F^*|_{U_i}$, with dual frame $(B^*_{ij})$. Elements in the frames of the dual bundle can be interpreted as bundle coordinates. For a section or distribution $\eta$ in some $E_l$ resp. $F$ write
	\[\eta_{ij}:=\psi_{i}^{-1*}(A^l_{ij}\eta|_{U_i})\]
	resp.
	\[\eta_{ij}:={\psi'}_{i}^{-1*}(B_{ij}\eta|_{U'_i}).\]
	\begin{lemma}
		\label{WFcoord}
		\begin{enumerate}
		\item We have
		\[\eta=\sum\limits_{i\in I, j\leq\rk(E_1)}\chi_i(\psi_i^*\eta_{ij})A^{1*}_{ij}.\]
		\item The wavefront of $\eta$ is the union of all sets
		\[\psi_i^*(\WF(\eta_{ij})),\]
		for $i\in I$ and $j\leq \rk(E_1)$.
		\item The seminorms of a distribution $\eta\in \D'_\Lambda(E_1)$ are equivalent to those of its components $(\eta_{ij})$ in the corresponding spaces $\D'_{\psi_i^{-1*}\Lambda}(\Ran(\psi_i))$. This means for every $i\in I$ and $j\leq \rk(E_1)$, the seminorms of $\eta_{ij}$ are bounded in terms of finitely many (actually, one) seminorms of $\eta$ and conversely each seminorm of $\eta$ is bounded in terms of finitely many seminorms of finitely many $\eta_{ij}$. Thus the topology of $\D'_\Lambda(E_1)$ is equivalently described by the seminorms
		\[\eta\mapsto \rho(\eta_{ij}),\]
		where $\rho$ is any seminorm of $\D'_{(\psi_i^{-1})^*(\Lambda)}(\Ran(\psi_i))$.
		\end{enumerate}
	\end{lemma}
\begin{proof}
	\begin{enumerate}
		\item We have
		\begin{align*}
		\sum\limits_{i\in I, j\leq\rk(E_1)}\chi_i(\psi_i^*\eta_{ij})A^{1*}_{ij}
		&=\sum\limits_{i\in I, j\leq\rk(E_1)}\chi_i\psi_i^*(\psi_i^{-1*}A^1_{ij}\eta) A^{1*}_{ij}\\
		&=\sum\limits_{i\in I, j\leq\rk(E_1)}\chi_i(A^1_{ij}\eta)A^{1*}_{ij}\\
		&=\sum\limits_{i\in I}\chi_i\eta\\
		&=\eta.
		\end{align*}
		Note that the sum is locally finite and thus each seminorm vanishes on all but finitely many summands, so it is a convergent sum in $\D'_{\Lambda}(E_1)$.
		\item For any chart $\psi$ and bundle coordinate $A$, we have
		\[\psi^{-1*}(A\eta)=\sum\limits_{i\in I, j\leq\rk(E_1)}((A\chi_iA^{1*}_{ij})\circ \psi^{-1}) \psi^{-1*}(\psi_i^*\eta_{ij}).\]
		 wavefront calculus on $\R^d$ gives us
		\[\WF(\psi^{-1*} (A\eta))\subseteq\bigcup_{i\in I,j\leq \rk(E_1)}(\psi_i\circ\psi^{-1})^*\WF(\eta_{ij}).\]
		Thus we have
		\[\WF(\eta)=\bigcup\limits_{A,\psi}\psi^*\WF(\psi^{-1*} (A\eta))\subseteq \bigcup_{i\in I,j\leq \rk(E_1)}(\psi_i)^*\WF(\eta_{ij}).\]
		The reverse inclusion follows from the fact that all $\psi_i$ and $A^1_{ij}$ are allowed choices for $A$ and $\psi$.
		\item For $\chi\in C_c^\infty(X)$, let $I_\chi$ denote the  (finite) set of $i\in I$ such that $\supp(\chi)\cap U_i$ is nonempty. For a bundle coordinate $A$, we have (on $U_i$)
		\[A=\sum\limits_{j\leq \rk(E_1)}A(A^{1*}_{ij})A^1_{ij}.\] Thus for any $\eta\in \D'_{\Lambda}(E_1)$ and any $k,V,\chi,\psi,A$ such that the following is a viable seminorm, we have
		\begin{align*}
		\|\eta\|_{k,V,\chi,\psi,A}&=\sup\limits_{\xi\in V}(|\xi|+1)^k|\F((\psi^{-1})^*(\chi A\eta))|\\
		&\leq\sum\limits_{j\leq \rk(E_1),i\in I_\chi}\sup\limits_{\xi\in V}(|\xi|+1)^k|\F((\psi^{-1})^*(\chi A(\chi_i \psi_i^*(\eta_{ij})A^{1*}_{ij})(\xi)|\\
		&=\sum\limits_{j\leq \rk(E_1),i\in I_\chi}\|(\psi_i\circ\psi^{-1})^*(((A(\chi_iA^{1*}_{ij}))\circ\psi_i^{-1})\eta_{ij})\|_{k,V,\chi\circ\psi}.
		\end{align*}
		As pushforwards by diffeomorphisms and multiplication with $C_c^\infty$-functions are continuous on wavefront spaces in $\R^n$, we obtain that any wavefront seminorm of $\eta$ is bounded by the wavefront seminorms of the $\eta_{ij}$ in $\D'_{(\psi_i^{-1})^*(\Lambda)}(\Ran(\psi_i))$. Conversely for any of the latter seminorms, we have
		\[\|\eta_{ij}\|_{k,V,\chi}=\|\eta\|_{k,V,\chi\circ\psi_i^{-1}, \psi_i, A_{ij}}.\]
		Concerning the non-wavefront seminorms, i.e. those of the strong topology on $\D'$, we know that pull-back by a diffeomorphism and application of a bundle coordinate are continuous, so the $\D'$-seminorms of $\eta_{ij}$ are bounded in terms of those of $\eta$. Conversely, by the formula in the first part of this lemma and continuity of multiplying with a smooth section and extending compactly supported distributions by 0, 
		the $\D'$-seminorms of $\eta$ are bounded in terms of those of the $\eta_{ij}$.\qedhere
	\end{enumerate}
\end{proof}
	
	We proceed to prove Theorem \ref{WFMFD}:
	\begin{thm*}[Wavefront calculus on manifolds]\ 
		\begin{enumerate}
			\item The seminorms defined for $\D'_\Lambda(E_1)$ are finite. In case $X$ is a subset of $\R^n$ and $E_1=X\times \C$, the new definition coincides with the earlier one.
			\item $\D'_\Lambda(E_1)$ is complete.
			\item $\Gamma_c(E_1)$ is dense in any $\D'_\Lambda(E_1)$.
			\item $\D'_\emptyset(E_1)$ coincides with $\Gamma(E_1)$ as a topological vector space.
			\item Differential operators between sections of $E_1$ and $E_2$ map $\D'_\Lambda(E_1)$ to $\D'_\Lambda(E_2)$ continuously.
			\item Let $f\colon X\rightarrow  Y$  be a smooth map, let $\Lambda, \Lambda'\subseteq \Tdot(X)$ and $\Theta\subseteq \Tdot(Y)$ be closed conic subsets.
			\begin{itemize}
				\item If $f^*(\Theta)$ does not intersect the zero section, $f^*$ extends to a continuous map from $\D'_\Theta(F)$ to $\D'_{f^*(\Theta)}(f^*(F))$.
				\item If $C\subseteq X$ is a closed set such that $f|_C$ is proper, $f_*$ defines a continuous map from $\D'_\Lambda(f^*F)|_C$ to $\D'_{f_*(\Lambda)}(F)$.
				\item If $\Lambda\bar + \Lambda'$ does not intersect the zero section,  pointwise "multiplication" of sections in $E_1\otimes E_2^*$ and $E_2\otimes E_3$, defined at each point by
				\[(a\otimes b)\cdot(c\otimes d):=b(c)\cdot(a\otimes d)\]
				extends to a hypocontinuous bilinear map
				\[\D'_\Lambda(E_1\otimes E_2^*)\times \D'_{\Lambda'}(E_2\otimes E_3)\rightarrow \D'_{\Lambda\bar +\Lambda'}(E_1\otimes E_3).\]
				\item The tensor product defines a hypocontinuous bilinear map
				\[\D'_\Lambda(E_1)\times D'_\Theta(F)\rightarrow \D'_{\Lambda\bar\times \Theta}(E_1\boxtimes F).\]
				\item If $\Lambda\bar + \Lambda'$ does not intersect the zero section and $K\subseteq X$ is compact, evaluation of distributions extends to hypocontinuous bilinear maps
				\[\D'_\Lambda(E_1)\times \D'_{\Lambda'}(E_1^*)|_K\rightarrow \C\]
				and
				\[\D'_{\Lambda'}(E_1^*)|_K\times \D'_\Lambda(E_1)\rightarrow \C.\]
				These are symmetric, in the sense that
				\[\zeta[\eta]=\eta[\zeta].\]
			\end{itemize}
		\end{enumerate}
	\end{thm*}
\begin{proof}
	\begin{enumerate}
		\item Let $\|\cdot\|_{k,V,\chi,\psi,A}$ be a seminorm of $\D'_\Lambda(E_1)$ and let $\eta\in \D'_\Lambda(E_1)$. Then $\supp(\chi\circ\psi^{-1})\times V$ does not intersect $(\psi^{-1})^*(\Lambda)$ and \[\psi^*(\WF((\psi^{-1})^*( A\eta)))\subseteq \WF(\eta)\subseteq \Lambda.\]
		Thus 
		$\supp(\chi\circ\psi^{-1})\times V$
		does not intersect $\WF((\psi^{-1})^*(\chi A\eta))$, hence $\|\cdot\|_{k,V,\chi\circ\psi}$ is a seminorm of $\D'_{\WF((\psi^{-1})^*( A\eta))}$, whence
		\[\|\eta\|_{k,V,\chi,\psi,A}=\|(\psi^{-1})^*(A\eta)\|_{k,V,\chi\circ\psi}\]
		is finite.
		
		In case $X\subseteq \R^n$ and $E_1=X\times \C$, choosing $\psi$ to be the identity and $A$ the map that forgets the basepoint, we obtain the original wavefront norms. As this is a cover with a set of local frames, Lemma \ref{WFcoord} implies that the two sets of seminorms are equivalent. This proves the first claim.
		\item Suppose $(\eta_s)_{s\in S}$ is a Cauchy net in $\D'_\Lambda(E_1)$. Then in particular it is also Cauchy in $\D'(E_1)$, which has less seminorms, so it converges to a limit $\eta$ there. This implies that all components $(\eta_s)_{ij}$ converge to the components $\eta_{ij}$ of the limit distributionally. All components $(\eta_s)_{ij}$ are also Cauchy nets in their corresponding wavefront spaces. As the wavefront spaces in $\R^n$ are complete, these components converge. As the wavefront spaces are continuously embedded in the spaces of all distributions, the limits in the wavefront space must also be $\eta_{ij}$. As the seminorms of the $\eta_{ij}-(\eta_s)_{ij}$ are equivalent to those of $\eta-\eta_s$, this implies that $\eta_s$ converges to $\eta$ in $\D'_\Lambda(E_1)$.
		
		\item For each $i$ and $j$, let $(\eta_{ijn})_{n\in \N}$ be a sequence of smooth functions converging to $\eta_{ij}$ in $\D'_{(\psi_i^{-1})^*(\Lambda)}(\Ran(\psi_i))$. Set
		\[\eta_n:=\sum\limits_{i\in I, j\leq\rk(E_1)}\chi_i(\psi_i^*\eta_{ijn})A^{1*}_{ij}.\]
		For any viable wavefront seminorm, we have
		\begin{align*}
		\|\chi_i(\psi_i^*\eta_{ijn})A^{1*}_{ij}-\chi_i(\psi_i^*\eta_{ij})A^{1*}_{ij}\|_{k,V,\chi,\psi,A}&=\|(((AA^{1*}_{ij})\chi_i)\circ\psi)(\psi_i\circ\psi^{-1})^*(\eta_{ijn}-\eta_{ij})\|_{k,V,\chi\circ\psi^{-1}}.
		\end{align*}
		As multiplication and pullback are continuous in wavefront norms on $\R^d$, this converges to zero. As all operations involved are distributionally continuous, we also have convergence in all seminorms of $\D'(E_1)$. We thus have as a limit in $\D'_\Gamma(E_1)$:
		\begin{align*}
		\limn \eta_n&=\sum\limits_{i\in I, j\leq\rk(E_1)}\limn\chi_i(\psi_i^*\eta_{ijn})A^{1*}_{ij}\\
		&=\sum\limits_{i\in I, j\leq\rk(E_1)}\chi_i(\psi_i^*\eta_{ij})A^{1*}_{ij}\\
		&=\eta.
		\end{align*}
		As $\eta$ was arbitrary and all $\eta_n$ are smooth, this concludes the proof.
		
		\item As $\eta$ is smooth if and only if all $\eta_{ij}$ are smooth and is in $\D'_\emptyset(E_1)$ is and only if all $\eta_{ij}$ are in $\D'_{\emptyset}(\Ran(\psi_i))$, equality as sets follows from the corresponding statement on $\R^n$. Moreover, we also know that the $C^\infty$- and $\D'_\emptyset$-norms of $\eta_{ij}$ are equivalent.  We have already seen that the $D'_\emptyset$-seminorms of the $\eta_{ij}$ are equivalent to those of $\eta$. To conclude the proof, we thus have to show that the $C^\infty$-seminorms of $\eta$ and of the $\eta_{ij}$ are equivalent. We once again use the equalities
		\[\eta_{ij}=\psi_{i}^{-1*}(A^1_{ij}\eta)\]
		and
		\[\eta=\sum\limits_{i\in I, j\leq\rk(E_1)}\chi_i(\psi_i^*\eta_{ij})A^{1*}_{ij}.\]
		As all opertations involved are continuous with respect to $C^\infty$-seminorms, we can conclude that the $C^\infty$-seminorms of $\eta$ are equivalent to those of the $\eta_{ij}$. Putting everything together, we obtain that the $D'_\emptyset$ - and $C^\infty$-norms of $\eta$ are equivalent.
		\item Differential operators also act as differential operators in coordinates, i.e. for every differential operator $D$ there are differential operators $D_{ijk}$ such that $(D\eta)_{ij}=\sum\limits_{k\leq \rk(E_1)}D_{ijk}\eta_{ik}$. As differential operators are continuous with respect to wavefront seminorms on $\R^n$, $D$ is continuous with respect to the seminorms of the component and hence with respect to the seminorms of $\D'_\Lambda(E_1)$ and $\D'_\Lambda(E_2)$.
		\item The idea of proof is always to consider the components $\eta_{ij}$ and thus reduce the problem to wavefront calculus on $\R^n$.
		\begin{itemize}
			\item \head{pull-back:}
			Assume without loss of generality that for any $i\in I$, there is $l(i)\in I'$ such that $U_i\subseteq f^{-1}(U'_{l(i)})$. This can always be arranged by refining the $U_i$. The $B_{l(i)j}$ then pull back to a frame of $f^*(F)^*$ over $U_i$, which we shall use as our frame to define $(f^*\eta)_{ij}$, i.e. we set for this part
			\[A^4_{ij}:=f^*(B_{l(i),j}).\] Suppose that $\eta\in \Gamma(F)$. If $f^*(\Theta)$
			does not intersect the zero section, then the same is true of
			\[(\psi_{l(i)}'\circ f\circ\psi_i^{-1})^*({\psi_{l(i)}'}^{-1*}(\Theta))\subseteq \psi_i^{-1*}(f^*\Theta)\]
			for any $i\in I$.
			Thus $(\psi_{l(i)}'\circ f\circ\psi_i^{-1})^*$ maps $\D'_{{\psi_{l(i)}'}^{-1*}(\Theta)}(\Ran(\psi_{l(i)}'))$
			to 
			$\D'_{\psi_i^{-1*}(f^*\Theta)}(\Ran(\psi_i))$ continuously. This means the corresponding seminorms of 
			\[(f^*\eta)_{ij}=(\psi_{l(i)}'\circ f\circ\psi_i^{-1})^*(\eta_{ij})\]
			are bounded in terms of those of $\eta_{ij}$. We can conclude that the seminorms of $f^*(\eta)$ in $\D'_{f^*\Theta}(f^*(F))$ are bounded in terms of the seminorms of $\eta$ in $D'_\Theta(F)$, so $f^*$ extends to a continuous map between those two spaces.
			\item \head{push-forward}
			Choose the cover of $Y$ in such a way that every $U_i'$ is precompact and consider an arbitrary $i'\in I'$. Then $f^{-1}(U_i)\cap C$ is contained in a compact set and thus only intersects the support of finitely many $\chi_i$. Let $I_{i'}\subseteq I$ be the indices of these functions, i.e. the set of indices $i$ such that
			\[\supp(\chi_i)\cap f^{-1}(U'_{i'})\cap C\neq \emptyset.\]
			For any $\eta\in \Gamma(E_1)$ with support in $C$, we then have for $i\notin I_{i'}$:
			\[\supp(\chi_i\eta)\cap(f^{-1}(U'_{i'}))=\emptyset.\]
			For any chart $\psi$, let $\rho_\psi$ be the smooth density function such that $\psi^{*}=\rho_\psi\psi^{-1}_*$ (i.e. the volume density for the coordinate $\psi$). We then have
			\[\psi^{-1*}=(\psi^*)^{-1}=(\rho_\psi\psi^{-1}_*)^{-1}=\psi_*\Big(\rho_\psi^{-1}\cdot\Big).\]
			Using the above, we have for $j'\leq \rk(F)$:
			\begin{align*}
			f_*(\eta)_{i',j'}&=f_*\Big(\sum\limits_{i\in I, j\leq\rk(E_1)}\chi_i(\psi_i^*\eta_{ij})A^{4*}_{ij}\Big)_{i',j'}\\
			&=\sum\limits_{i\in I_{i'},j\leq\rk(F)}{\psi'_{i'}}^{-1*}B_{i',j'}f_*(\chi_i(\psi_i^*\eta_{ij})A^{4*}_{ij})\\
			&=\sum\limits_{i\in I_{i'},j\leq\rk(F)}\psi'_{i'*}\rho_{\psi'_{i'}}B_{i',j'}f_*(\chi_i(\rho_{\psi_i}^{-1}{\psi_{i*}^{-1}}\eta_{ij})A^{4*}_{ij})\\
			&=\sum\limits_{i\in I_{i'},j\leq\rk(F)}\psi'_{i'*}f_*{\psi_{i*}^{-1}}((f^*(\rho_{\psi'_{i'}}^{-1}B_{i',j'})\chi_i\rho_{\psi_i}A^{4*}_{ij})\circ \psi_i^{-1}\ \eta_{ij})\\
			&=:\sum\limits_{i\in I_{i'},j\leq\rk(F)} (\psi'_{i'}\circ f\circ \psi_i^{-1})_*(\nu \eta_{ij}),
			\end{align*}
			where $\nu$ denotes the smooth compactly supported function $\eta_{ij}$ is multiplied with in the line before. Multiplication with $\nu$ is continuous in the appropriate wavefront space and ensures that the support condition for aplying the pushforward theorem on $\R^d$ are met (every continuous function from a compact space to a Hausdorff space is proper). Thus all relevant seminorms of 
			\[(\psi'_{i'}\circ f\circ \psi_i^{-1})^*(\nu \eta_{ij})\]
			are bounded in terms of those of  the $\eta_{ij}$ and hence in terms of the seminorms of $\eta$. By the above, the same is then true for the seminorms of $f_*(\eta)_{i',j'}$ and thus for those of $f_*(\eta)$. This means that $f_*$ extends to a continuous map from $\D'_\Lambda(f^*(F))|_C$ to $\D'_{f_*(\Lambda)}(F)$.
			\item \head{product:}
			On tensor product bundles, use the tensor products of the individual frames as the new reference frames, indexed by pairs of individual indices. 
			Fix a bounded subset $B\subseteq \D'_\Lambda(E_1\otimes E_2^*)$. Then for any $i\in I$, $j_1\leq \rk(E_1)$ and $j_2\leq \rk(E_2)$, the set 
			\[B_{i,(j_1,j_2)}:=\{\eta_{i,(j_1,j_2)}\mid \eta \in B\}\]
			is bounded in $\D'_{\psi_i^{-1*}(\Lambda)}(\Ran(\psi_i))$, as the seminorms of the $\eta_{i(j_1,j_2)}$ there are bounded in terms of the seminorms of $\eta$, which must be bounded on $B$. 
			Consider $\eta\in B$ and $\zeta\in \Gamma(E_2\otimes E_3)$ and fix $i$, $j_1$ and $j_2$ as above.  We have
			\begin{align*}
			(\eta\zeta)_{i(j_1,j_2)}&:=\psi_i^{-1*}((A^1_{ij_1}\otimes A^3_{ij_2})\eta\zeta)\\
			&=\psi_i^{-1*}\Big((A^1_{ij_1}\otimes A^3_{ij_2})\sum\limits_{k_1\leq \rk(E_1)}\sum\limits_{l_1\leq \rk(E_2)}\psi_i^*(\eta_{i(l_1,k_1)})A^{1*}_{il_1}\otimes A^2_{ik_1}\\
			&\cdot\sum\limits_{l_2\leq \rk(E_2)}\sum\limits_{k_2\leq \rk(E_3)}\psi_i^*(\zeta_{i(l_2,k_2)})A^{2*}_{il_2}\otimes A^{3*}_{ik_2}\Big)\\
			&=\sum\limits_{l_1,k_1,l_2,k_2\text{ as above}}\eta_{i(l_1,k_1)}\zeta_{i(l_2,k_2)}\psi_i^{-1*}(A^1_{ij_1}(A^{1*}_{il_1})A^2_{ik_1}(A^{2*}_{il_2})A^3_{ij_2}(A^{3*}_{ik_2}))\\
			&=\sum\limits_{l_1,k_1,l_2,k_2\text{ as above}}\eta_{i(l_1,k_1)}\zeta_{i(l_2,k_2)}\delta_{j_1l_1}\delta_{k_1l_2}\delta_{j_2k_2}\\
			&=\sum\limits_{k\leq \rk(E_2)}\eta_{i(j_1,k)}\zeta_{i(k,j_2)}
			\end{align*}
			The product as a map
			\[\D'_{\psi_i^{-1*}(\Lambda)}(\Ran(\psi_i))\times \D'_{\psi_i^{-1*}(\Lambda')}(\Ran(\psi_i))\rightarrow \D'_{\psi_i^{-1*}(\Lambda\bar + \Lambda')}(\Ran(\psi_i))\]
			is hypocontinuous by wavefront calculus on $\R^d$.
			Thus for every seminorm $\rho$ of $\D'_{\psi_i^{-1*}(\Lambda\bar + \Lambda')}(\Ran(\psi_i))$, there are seminorms $(\rho'_{kl})$ of $\D'_{\psi_i^{-1*}(\Lambda')}(\Ran(\psi_i))$ and seminorms $(\rho''_l)$ of $\D'_{\Lambda'}(E_2\otimes E_3)$ such that for some $N,N'\in\N$ and $C_k,C\in \R$ (depending on $B$, but not $\eta$), we have
			\begin{align*}
			\rho((\eta\zeta)_{i(j_1,j_2)})&\leq \sum\limits_{k\leq \rk(E_2)}\rho(\eta_{i(j_1,k)}\zeta_{i(k,j_2)})\\
			&\leq \sum\limits_{k\leq \rk(E_2)}C_k\sum\limits_{l=0}^N \rho'_{kl}(\zeta_{i(k,j_2)})\\
			&\leq \sum\limits_{l=0}^{N'}C\rho''_l(\zeta).
			\end{align*}
			As the seminorms of $\eta\zeta$ in $D'_{\Lambda\bar +\Lambda'}$ can be estimated by the component norms above, $\zeta\mapsto \eta\zeta$ extends contiuosly to $\D'_{\Lambda'}(E_2\otimes E_3)$ for any fixed $\eta$, and the extension satisfies the first half of the definition of hypocontinuity, i.e. the part with bounded sets in the first component.
			As multiplication is (up to renaming and reordering vector bundles) symmetric in both arguments, we also obtain an extension of multiplication of sections that satisfies the second half of the definition of hypocontinuity. As both maps are continuous on products of bounded sets and convergent sequences are bounded, approximation by sequences of smooth sections shows that both extensions coincide. Thus we get a single extension that is hypocontinuous.
			\item\head{tensor product:} To avoid having to do a similar proof as for multiplication again, we reduce this to what we have already shown. Let $\pi_1$ and $\pi_2$ denote the projections of $X\times Y$ on the first and second factor. Then for smooth sections $\eta$ and $\zeta$ of $E_1$ and $F$, we have
			\[\eta\otimes \zeta=\pi_1^*(\eta)\cdot \pi_2^*(\zeta).\]
			The previous results yield that
			\[\pi_1^*\colon \D'_{\Lambda}(E_1)\rightarrow\D'_{\Lambda\times 0}(\pi_1^*(E_1))\]
			and
			\[\pi_2^*\colon \D'_{\Theta}(F)\rightarrow \D'_{0\times \Theta}(\pi_2^*(F))\]
			are continuous and that multiplication is continuous as a map
			\[\cdot\colon \D'_{\Lambda\times 0}(\pi_1^*(E_1))\otimes \D'_{0\times \Theta}(\pi_2^*(F))\rightarrow \D'_{\Lambda\bar\times \Theta}(E_1\boxtimes F),\]
			as non-zero vectors in different components cannot add up to zero. Thus the composition
			\[\cdot\circ(\pi_1^*\times\pi_2^*)\]
			yields a hypocontinuous extension of the tensor product.
			\item \head{evaluation}	Fix $\chi\in C_c^\infty(X)$ that is 1 around $K$.
			For $\phi\in \Gamma(E_1^*)$ supported in $K$ and $\eta\in \D'_\Lambda(E_1)$, we have
			\[\eta[\phi]=\eta[\phi\chi]=\phi\eta[\chi].\]
			As multiplication is hypocontinuous on the appropriate spaces and evaluation is continuous, the map 
			\[(\eta,\zeta)\mapsto \eta[\zeta]:=\zeta\eta(\chi)\]
			is a hypocontinuous extension of evaluation to $\D'_\Lambda(E_1)\times \D'_{\Lambda'}(E_1^*)|_K$.
			In case both $\eta$ and $\zeta$ are smooth, we have
			\[\zeta[\eta]=\eta[\zeta].\]
			Thus, defining
			\[\zeta[\eta]:=\eta[\zeta]\]
			for more general distributions yields a hypocontinuous map on $\D'_{\Lambda'}(E_1^*)|_K\times\D'_\Lambda(E_1)$ that agrees with evaluation in case both arguments are smooth. As smooth fuctions are dense in the appropriate spaces, it agrees with evaluation wherever evaluation is defined.\qedhere
		\end{itemize}

	\end{enumerate}
\end{proof}
\newpage
\chapter*{Notation index}
\addcontentsline{toc}{chapter}{Notation index}
This is a list of notation used in this thesis. Further notation conventions that are not tied to specific symbols are described in subsection \ref{abuseofnotation}.
\begin{description}
	\item [$\mathbbm{1}_S$:] indicator function of a set $S$
	\item [$\|\cdot\|_{k,V,\chi}$:] wavefront norm associated to $k$, $V$ and $\chi$, \ref{wfnorm}
	\item [$\|\cdot\|_{k,V,\chi,\psi,A}$:] wavefront norm on a manifold, \ref{dfwfmfd}
	\item [${[ \cdot]}$:] evaluation as a distribution
	\item [${[ \cdot]}$:] in Chapter \ref{Hadamextract}, matrix or vector with given coefficients, \ref{dfmatvec}
	\item[${[[\cdot]]}$:] coefficient corresponding to a monomial, \ref{monocoeff}
	\item [$\bar +$:] "sum" of wavefronts (wavefront of product), \ref{WFops}
	\item [$\bar\times$:] "product" of wavefronts (wavefront of tensor product), \ref{WFops}
	\item[$\boxtimes$:] external tensor product of vector bundles
	\item[$\nabla$:] connection induced by $P$, \ref{nabla}
	\item[$\hat\cdot$:] Fourier transform
	\item[$\partial$:] boundary of a topological spacce or manifold
	\item[$\partial_x$:] partial derivative with respect to a variable $x$
	\item [$\sim$:] asymptotic expansion in differentiability orders, \ref{dfasympt}
	\item [$\sim_x$:] asymptotic expansion in differentiability orders, including differentiability in $x$, \ref{dfasympt}
	\item [$\stackrel{s\rightarrow 0}{\sim}$:] asymptotic expansion in parameter $s$, \ref{dfasympt}
	\item[$\cdot^*$:] dual, adjoint or pullback:
	\subitem applied to a vector space or vector bundle: dual space/bundle
	\subitem applied to an operator: formal adjoint, \ref{df*}
	\subitem applied to other (fibrewise) linear maps: (fibrewise) adjoint
	\subitem applied to a function: pullback, \ref{pullforward} and \ref{WFops}
	\item [$f_*$:] pushforward by a function $f$, \ref{pullforward} and \ref{WFops}
	\item[$*$:] convolution
	\item[$a(k,n)$:] expansion coefficients occuring in \ref{Wexp}, defined in \ref{dfa}
	\item[$A$:] in Chapter \ref{globalchap}: Extension of $\tr(\Pi_w\circ\cdot)$, \ref{dftansport}
	\item[$A_x$:] restriction of $A$ to $x$ in the second coordinate, \ref{dftansport}
	\item[$\chi_\Phi$:] cut-off supported in $\Dom(\Phi)$, \ref{pronotation}
	\item [$C^{(')}$:] arbitrary constant
	\item [$C^k$:] Space of $k$ times continuously differentiable functions (to $\C$ unless specified otherwise)
	\subitem $C^k_c$: compactly supported $C^k$-functions
	\item [$C_\infty^+$:] Intersection of $C_0(\R)$ and $C_c^\infty(0,\infty)$, \ref{deflambda}
	\item[$c_\alpha$:] holomorphic prefactor in the Riesz distributions, \ref{dfRiesz}
	\item [$\delta$:] in Chapter \ref{Hadamextract}: $\frac{d}{2}-1-o$
	\item[$\delta_{x}$:] Dirac distribution (at $x$)
	\item [$\delta_{ij}$] Kronecker delta
	\item [$d$:] dimension of $M$, \ref{gnM}
	\item[$\diag$:] diagonal matrix with given entries, \ref{dfmatvec}
	\item [$\D'(F)$:] distributions in a vector bundle $F$
		\subitem $\D'_{(s)\pm}$: (strictly) past/future compactly supported distributions
		\subitem $\D'_\Lambda$: distributions with wavefront in $\Lambda$
		\subitem $\D'(\cdot)|_A$: distributions supported in $A$
	\item[$\Dom$:] domain of a function
	\item [$E$:] vector bundle over $M$, \ref{gnP}
	\item[$(\cdot)_{even}$:] even part of a function, \ref{evenodd}
	\item [$\exp_x$:] Riemannian exponential map at point $x$
	\item [$f$:] odd function in $C_c^\infty(\R)$, starting at \ref{dff}
	\subitem starting at \ref{Mfnonzero}, assume that $\M'(f)$ does not vanish on integers
	\item [$f_s$:] $f(\tfrac{\cdot}{s})$, \ref{dff}
	\item[$F_s$:] $f_s$ with two extra arguments, \ref{pronotation}
	\item[$\F$:] Fourier transform
	\item [$\gamma$:] Lorentzian "norm squared", \ref{dfgamma}
	\item [$\Gamma$:] Lorentzian "distance squared", \ref{dfGamma}
	\item [$\Gamma$:] Gamma function
	\item [$\Gamma(F)$:] sections in a vector bundle $F$ (smooth, unless otherwise specified), \ref{dfsections}
	\subitem $\Gamma^k$: $C^k$-sections
	\subitem $\Gamma_c$: compactly supported sections
	\subitem $\Gamma_A$: sections supported in $A$
	\subitem $\Gamma_{(s)\pm}$: (strictly) past/future compactly supported sections
	\item [$g$:] Lorentzian metric on $M$, \ref{gnM}
	\item [$G^\pm$:] advanced retarded Green's operators (of $P$, unless subscript specifies otherwise), \ref{dfG}, \ref{extendedG}
	\subitem without $^\pm$: difference of both, see \ref{dfcp}
	\item [$G_x^\pm$:] Green's kernel at $x$ (of $P$, unless subscript specifies otherwise), \ref{dfGx}
	\subitem without $^\pm$: difference of both, see \ref{dfcp}
	\item[GE:] open, convex, globally hyperbolic and causally compatible, \ref{dfGE}
	\item[$\iota_x$:] inclusion of $M$ into $M\times M$ with $x$ in the second component, \ref{pronotation}
	\item[$I$:]domain of $w$ in \ref{timeint2}
	\item[$I_f$:] closed interval around $\supp(f)$ and $0$, \ref{dff}
	\item [$J_{(\pm)}:$] causal future/past, \ref{futurepast}
	\item [$\K$:] Schwartz kernel, \ref{SKnotation}
	\item [$\Lambda$:] some wavefront set
	\subitem in Chapter \ref{timeint}: as defined in Definition \ref{deflambda}
	\subitem in Section \ref{globalsec1}: as defined in Lemma \ref{resGx}
	\item [$L$:] in Chapter \ref{Hadamextract}: function with a specified asymptotic expansion, \ref{dfWL}
	\item[$L_s$:] rigorous version of the Green's kernel integrated against $f_s$ along a timeline, \ref{dfL}
	\item[$\mu_\chi$:] multiplication by a function $\chi$
	\item [$M$:] time-oriented Lorentzian manifold, \ref{gnM}
	\subitem assumed to be globally hyperbolic after theorem \ref{Ghyper}
	\subitem assumed to be foliated into Cauchy hypersurfaces in section \ref{evopchap}
	\subitem in Section \ref{globalsec2}: of the form $\Sigma\times \R$  with unit speed geodesic timelines, \ref{foliate}
	\item [$\M(\cdot)$:] Mellin transform, \ref{Mellin}
	\item [$\M'(\cdot)$:] modified Mellin transform, \ref{dfMprime} 
	\item [$\nu_w$:] quotient of $\Gamma(w(t))$ and $t^2$, \ref{dfnu}
	\item[$n$:] normal vector field to the foliation in Section \ref{evopchap}
	\item[$o$:] natural number used as offset parameter in Chapter \ref{Hadamextract}
	\item [$O_\xi$:] "reflection" associated to spacelike $\xi$, as defined in \ref{Oxi}
	\item[$(\cdot)_{odd}$:] odd part of a function, \ref{evenodd}
	\item [$\Op$:] Operator associated to a kernel, \ref{dfOp}
	\item[$\Phi$:] "flow" along geodesics $w_x$ in the first component, \ref{dflow}
	\item[$\pi$:] projection of $\Dom(\Phi)$ onto $M\times M$, \ref{pronotation}
	\item[$\Pi_w$:] Parallel transport along some $w_x$, \ref{dftansport}
	\item[$\Pi_\Phi$:] Parallel transport along some $w_x$ with transport time determined by the last coordinate in $M\times M\times \R$, \ref{dftansport}
	\item[$\Pi(r,t)$:] Parallel transport along timelines from $\Sigma_t$ to $\Sigma_r$, \ref{Qtransport}
	\item [$P$:] normally hyperbolic operator on $E$, \ref{gnP}
	\item [$Q(t,s)$:] Evolution operator between Cauchy surfaces $\Sigma_t$ and $\Sigma_s$, with respect to $P$ unless specified otherwise, \ref{dfQ}
	\item [$\rho_x^U$:] differential operator on a convex subset $U$ (for basepoint $x\in U$), involved in a "Leibniz rule" for Riesz distributions and in the transport equations, \ref{dfrho}
	\item[$\rk$:] rank of a vector bundle
	\item [$\Re$:] real part of a complex number
	\item[$R_\pm(\alpha)$:] Riesz distribution on a Lorentzian vector space, \ref{dfRiesz}
	\subitem without $_\pm$: difference of both, see \ref{dfcp}
	\item[$R_\pm^U(\alpha,x)$:] Riesz distribution on a convex set $U$ around basepoint $x$, \ref{dfGamma}
	\subitem without $_\pm$: difference of both, see \ref{dfcp}
	\item[$R_\pm(z,2m,x)$:] resolvent Riesz distribution (on $U$) around basepoint $x$, \ref{Rieszres}	
	\item[$\Ran$:] range of a function
	\item[$\sigma$:] sign of the time component, \ref{dfsigma}
	\item[$\Sigma_t$] cauchy hypersurface in a foliation of $M$ in sections \ref{evopchap} and \ref{globalsec2}
	\item[$\scal$:] scalar curvature (of $M$)
	\item[$\tau_w$:] time translation along the curves $w_x$, \ref{dftansport}
	\item[$\tau(r,t)$:] time translation from $\Sigma_t$ to $\Sigma_r$, \ref{Qtransport}
	\item[$\tr$:] trace
	\item [$T^{(*)}_{(x)} X$:] (Co-) Tangent space of a manifold $X$ (at $x$)
	\item[$\Tdot_{(x)}X$:] cotangent space of a manifold $X$ (at $x$) without zero
	\item[$U$:] fixed convex open subset of $M$ in sections $\ref{Rieszhadamard}$, \ref{Greschap} and \ref{timeint2}
	\item [$V^{k,U}_x$:] $k$-th Hadamard coefficient at basepoint $x$ on a convex set $U$, \ref{Vdef}
	\item [$V^{k}_x(z)$:] $k$-th Hadamard coefficient at basepoint $x$ of $P-z$, \ref{defVz}
	\item[$w$:] timelike curve (in $U$) in Section \ref{timeint2}, see \ref{dfw}
	\item[$w_x$:] timelike unit speed geodesic through $x$, \ref{dflow}
	\subitem assumed to be of the form $w_{(y,r)}(t)=(y,r+t)$ in Section \ref{globalsec2}, see \ref{foliate}
	\item[$W_{l,n}$:] coefficients in the expansion of $L$ in Chapter \ref{Hadamextract}, corresponding to derivatives of Hadamard coefficients, \ref{dfWL}
	\item[$\WF$:] wavefront of a distribution,\ref{dfwfrn} and \ref{dfwfmfd}
	\item[$\Xi$:] linear combination occuring in the final formulas, defined in \ref{dfXi}
	\item[$X$:] open subset of $\R^n$ in Section \ref{wfrn}
	\item[$X_\alpha$:] power function on $\Rp$, \ref{Xa}
	\item[$Y$:] open subset of $\R^n$ in Section \ref{wfrn}
	\item[$y_\xi$:] reflection vector for spacelike $\xi\in \R^d$, as in \ref{yxi}
	\chapter*{Bibliography}
	\addcontentsline{toc}{chapter}{Bibliography}
\end{description}
\printbibliography[heading=none]
\end{document}